\documentclass[11pt]{amsart}

\usepackage[marginratio=1:1,height=665pt,width=480pt,tmargin=70pt]{geometry}

\usepackage{amssymb,latexsym,amsmath,amsthm,amscd}
\usepackage{mathrsfs}   
\usepackage{fancyhdr}
 \usepackage{ifthen}

\usepackage{etoolbox}

\newcommand{\im}{\lrcorner\,}

\appto\appendix{\addtocontents{toc}{\protect\setcounter{tocdepth}{1}}}
\appto\listoffigures{\addtocontents{lof}{\protect\setcounter{tocdepth}{1}}}
\appto\listoftables{\addtocontents{lot}{\protect\setcounter{tocdepth}{1}}}
\theoremstyle{plain}
\newtheorem{theorem}{Theorem}[section]
\newtheorem{corollary}[theorem]{Corollary}

\newtheorem*{theorem*}{Theorem}
\newtheorem*{corollary*}{Corollary}
\newtheorem{proposition}[theorem]{Proposition}
\newtheorem{definition}[theorem]{Definition}

\theoremstyle{remark}

\newtheorem{remark}[theorem]{Remark}

\newtheorem{example}[theorem]{Example}
\newtheorem*{example*}{Example}

\numberwithin{equation}{section}

\usepackage{array,rotating}  
\usepackage{color}
\usepackage{xcolor}
\definecolor{carmine}{rgb}{0.59, 0.0, 0.09}
\definecolor{mediumpersianblue}{rgb}{0.0, 0.4, 0.65}
\definecolor{persianplum}{rgb}{0.44, 0.11, 0.11}
\usepackage[colorlinks=true,
            linkcolor=persianplum, 
            urlcolor=olive,
            citecolor=mediumpersianblue, 
            backref=page]{hyperref}
\usepackage{amsmath}
\usepackage{amsfonts}
\usepackage{amssymb}
 
\newcommand{\cC}{\mathcal{C}}
\newcommand{\cE}{\mathcal{E}}
\newcommand{\cG}{\mathcal{G}}

\newcommand{\cL}{\mathcal{L}}

\newcommand{\cP}{\mathcal{P}}

\newcommand{\cJ}{\mathcal{J}}

\newcommand{\cI}{\mathcal{I}}
\newcommand{\cN}{\mathcal{N}}

\newcommand{\scX}{\mathcal{X}}
\newcommand{\scV}{\mathcal{V}}
\newcommand{\scD}{\mathcal{D}}
\newcommand{\scE}{\mathcal{E}}
\newcommand{\scK}{\mathcal{K}}
\newcommand{\scB}{\mathcal{B}}
\newcommand{\scC}{\mathcal{C}}

\newcommand{\fg}{\mathfrak{g}}
\newcommand{\fp}{\mathfrak{p}}

\newcommand{\fs}{\mathfrak{s}}

\newcommand{\R}{\mathbb{R}}
\newcommand{\RR}{\mathbb{R}}
\newcommand{\FF}{\mathbb{F}}

\newcommand{\C}{\mathbb{C}}
\newcommand{\CC}{\mathbb{C}}

\newcommand{\PP}{\mathbb{P}}

\newcommand{\bC}{\mathbf{C}}

\newcommand{\bH}{\mathbf{H}}

\newcommand{\bT}{\mathbf{T}}

\newcommand{\w}{{\,{\wedge}\;}}

\newcommand{\exd}{\mathrm{d}}

\newcommand{\half}{\textstyle{\frac 12}}
\newcommand{\Dt}{X_F}  

\makeatletter
\renewcommand*{\p@section}{\S\,}
\renewcommand*{\p@subsection}{\S\,}
\renewcommand*{\p@subsubsection}{\S\,}
 
\makeatother
\usepackage{float}

\begin{document}

\author{Wojciech Kry\'nski}\author{Omid Makhmali}

 \address{\newline Wojciech Kry\'nski\\\newline
   Institute of Mathematics, Polish Academy of Sciences, \'Sniadeckich 8, 00-656 Warszawa, Poland\\\newline
   \textit{Email address: }{\href{mailto:krynski@impan.pl}{\texttt{krynski@impan.pl}}}\\\newline\newline
 Omid Makhmali\\\newline
Departamento de Geometr\'{\i}a y Topolog\'{\i}a and IMAG, Universidad de Granada, Granada 18071, Spain \\\newline
   \textit{Email address: }{\href{mailto:omakhmali@ugr.es}{\texttt{omakhmali@ugr.es}}}\\\newline
Department of Mathematics and Statistics, UiT The Arctic University of Norway, Troms\o\  90-37,Norway \\\newline
    \textit{Email address: }{\href{mailto:omid.makhmali@uit.no}{\texttt{omid.makhmali@uit.no}}}
 }

\title[]
{Lewy curves in para-CR geometry} 
\date{\today}

\begin{abstract}
We define a class of curves, referred to as \emph{Lewy curves}, in para-CR geometry,   following H. Lewy's original definition in CR geometry.  
We give a  characterization of path geometries defined by para-CR Lewy curves.    In dimension 3 our characterization is given by a set of  necessary and sufficient conditions which, with the exception of one condition, are easily computationally verifiable. Furthermore, we show that Lewy curves of a para-CR 3-manifold coincide with chains of some (para-)CR 3-manifold if and only if it is flat.   Subsequently, it follows that Lewy curves determine the para-CR structure up to the sign of the almost para-complex structure. In higher dimensions we show that para-CR Lewy curves define a path geometry if and only if the para-CR structure is flat, in which case chains and Lewy curves coincide. 
\end{abstract}

\subjclass{Primary: 53A40, 53B15,  53C15; Secondary:  53A55, 34A26, 34A55, 32V99}
\keywords{para-CR geometry, Lewy curves, path geometry, Cartan connection, Cartan reduction}

\maketitle
  
\vspace{-.5 cm}

\setcounter{tocdepth}{2} 
\tableofcontents

\section{Introduction}
\label{sec:introduction}

Many geometric structures on manifolds come equipped with a distinguished class of curves, e.g. geodesics in (pseudo-)Riemannian and projective structures, null geodesics in conformal pseudo-Riemannian structures, and conformal geodesics in conformal (pseudo-)Riemannian structures. In the case of (para-)CR structures the most well-known class of distinguished curves, originally defined in \cite{CM-CR}, are called chains. In this article we focus on another class of curves in (para-)CR structures, referred to as \emph{Lewy curves}.

Lewy curves were first studied  in \cite{Faran-Lewy} for analytic CR structures in relation to CR chains, in which the author attributes their definition  to H. Lewy and, therefore, names them after him. It turns out that through each point of an analytic CR 3-manifold and along  an open set of    transverse directions to the contact distribution at that point, there passes a unique Lewy curve. As a result, as in the case of chains, Lewy curves define a \emph{(generalized) path geometry} on a CR 3-manifold. 

The aim of this paper is four-fold. Firstly, we define   Lewy curves for (real or complex) para-CR structures and show that in dimension three they determine para-CR structures up to the signature of the almost para-complex structure. Secondly, we characterize path geometries that are defined by Lewy curves of a para-CR 3-manifold.   Thirdly, we highlight the difference between chains and Lewy curves using our characterization here and the one in \cite{KM-chains} which  shows that they only coincide when the para-CR 3-manifold is flat.  Lastly, in higher dimensions, it turns out that in   a non-flat para-CR structure Lewy curves define a \emph{higher  path geometry} which corresponds to an ODE system of order greater than two.  We show that in higher dimensional para-CR structures Lewy curves define a (generalized) path geometry if and only if the para-CR structure is flat, in which case Lewy curves coincide with chains. As will be discussed,  this is a stronger version of the para-CR analogue of  Faran's result in the CR setting \cite{Faran-Lewy} and are obtained using path geometric and Cartan geometric techniques. 
 
\subsection{Lewy curves in (para-)CR geometry: a first look}\label{sec:first-look}   
Let us start with the case of 3-dimensional CR structures. Following \cite{Faran-Lewy},  let 
\[
N=\{z\in\C^2\ |\ r(z,\bar z)=0\}\subset\C^2,
\]
be a real analytic 3-dimensional CR-manifold for an analytic function $r\colon\C^2\to\R$ where $\exd r\neq 0$. Function $r$ extends to a complex-valued function  on $\C^2\times\C^2$ (or an open subset of $\C^2\times\C^2$), allowing one to define the hypersurface 
\begin{equation}\label{eq:Nc-complexparaCR}
  \cN^\C=\{(z,\zeta)\in\C^2\times\C^2\ |\ r(z,\zeta)=0\}.
  \end{equation}
The hypersurface $\cN^\C$ is referred to as the \emph{Segre family} of the CR-manifold $N,$  first introduced by Segre in \cite{Segre1,Segre2}, further studied by Cartan \cite{CartanSegre}, and later by many others including \cite{Chern,Webster,Faran-Segre}. The Segre family   defines a \emph{complex para-CR structure} on $\cN^{\C}$ (see Definition \ref{def:para-cr-structures}).  Fixing $\zeta\in\C^2,$ define the complex curve $Q_\zeta\subset\C^2$  as
\begin{equation}\label{eq:Q-zeta}
  Q^\C_\zeta=\{z\in\C^2\ |\ r(z,\zeta)=0\}.
\end{equation} 
 The  family of real  curves  $\gamma_\zeta\subset N$ defined as
\[
\gamma_\zeta=Q^\C_\zeta\cap N
\] 
depend on 4 real parameters and are referred to as Lewy curves. It turns out that the family of Lewy  curves define a  generalized path geometry  canonically associated to the CR structure on $N.$

Our main focus is to define Lewy curves in (complex or real) para-CR geometry and characterize path geometries defined by them. A para-CR structure on a 3-dimensional manifold $N$ is defined as a  contact distribution $\scC\subset TN$ equipped with a splitting $\scC=\scD_1\oplus\scD_2$ for line bundles $\scD_1,\scD_2\subset TN.$

Considering real para-CR 3-manifolds, it turns out (see \ref{sec:paraCR}) that $N$ can be locally realized as a hypersurface in a 4-dimensional product of two surfaces $M_1$ and $M_2$, i.e. 
\[
N=\{(x,y)\in M_1\times M_2\ |\ \Phi(x,y)=0\}\subset M_1\times M_2,
\]
for a differentiable function $\Phi\colon M_1\times M_2\to \R$, where $M_1$ and $M_2$ are the local leaf spaces of the integral curves of $\scD_1$ and $\scD_2$, respectively. Note that locally $M_1\times M_2\cong \RR^2\times (\RR^2)^*$  with its standard para-CR structure, which replaces the complex structure on $\CC^2$ in the CR setting.
 The para-CR analogue  of the Segre family is the real 6-manifold
\[\cN:=\{(x,y,\hat x,\hat y)\in(M_1\times M_2)\times(M_1\times M_2)\ |\ \Phi(x,\hat y)=\Phi(\hat x,y)=0\}.\]
Subsequently, the 4-parameter family of  surfaces $\{Q_{\hat x\hat y}\}$ where
\[
Q_{\hat x\hat y}:=\{(x,y)\in M_1\times M_2 \ |\ \Phi(x,\hat y)=\Phi(\hat x,y)=0\},
\] 
can be used to define Lewy curves as curves in $N$ obtained via the intersection of $Q_{\hat x\hat y}$ and $N$ for all $(\hat x,\hat y)\in M_1\times M_2\setminus N.$  In other words, Lewy curves on $N$ are  parameterized by points $(\hat x,\hat y)\in M_1\times M_2\setminus N$ and are defined as
\begin{equation}\label{eq:gammaxy-first-look}
\gamma_{\hat x\hat y}=\{(x,y)\in N\ |\  \Phi(x,\hat y)=\Phi(\hat x,y)=0 \},
\end{equation}
which endow  $N$ with a generalized path geometry. We give a coordinate-free  description of the definition above  in  \ref{sec:intrinsic}. We study certain properties of these curves and characterize them among 3-dimensional path geometries.

The (complex) Segre family in  CR geometry and the real Segre structure in the  para-CR setting can be treated identically and give rise to the set of Lewy curves. The only difference, as will be discussed  (see Remark \ref{rmk:para-cr-structures-2D-path}), is that in para-CR 3-manifolds assuming smoothness is sufficient for the definition of Lewy curves, unlike the case of CR structures and   higher dimensional para-CR manifolds where analyticity is needed.

Lastly,  note that  the  complex para-CR structure on  complex 3-manifold $\cN^{\C}$ in \eqref{eq:Nc-complexparaCR} is also equipped with a complex path geometry defined by complex Lewy curves. More precisely, replacing the real-valued function $\Phi\colon\RR^2\times\RR^2\to\RR$ in \eqref{eq:gammaxy-first-look} with the complex-valued function $r\colon\C^2\times \C^2\to\C$ used in  \eqref{eq:Nc-complexparaCR}, one obtains a 4 complex-parameter family of complex Lewy curves $\gamma^\C_{\hat x\hat y}\subset\cN^\C,(\hat x,\hat y)\in\CC^2\times \CC^2\setminus\cN^{\C}$  which define a (generalized) path geometry on $\cN^\C.$ Similarly, they can be defined via the intersection of $\cN^\C$ with 4-parameter family of complex surfaces $Q^\C_{\hat x\hat y}.$ As will be mentioned, our discussion on path geometry of Lewy curves in para-CR structures holds valid in both the real and complex setting.

\subsection{Outline of the article and main results}

In \ref{sec:path-geom-review} we review some well-known facts about path geometries and para-CR structures and their Cartan geometric description. In \ref{sec:path-geom-defin} we define a (generalized) path geometry on an $(n+1)$-dimensional manifold $M$  as a triple $(Q,\scV,\scX)$ where $\scX,\scV\in\Gamma(TQ)$  are distributions that intersect trivially and have  ranks 1 and $n-1,$ respectively, $\scV$ is integrable with the property that $[\scX,\scV]=TQ$ and $Q$ is a $(2n+1)$-dimensional manifold such that locally $Q\slash \scV=M$. It follows that $Q$ is locally  (an open subset of) the projectivized tangent bundle $\PP TM\to M.$ Locally, path geometries correspond to point equivalence classes of systems of 2nd order ODEs.

In \ref{sec:paraCR}  para-CR structures are defined as $(2n+1)$-dimensional contact manifolds $N$  whose contact distribution $\scC\subset TN$ is equipped  with a splitting   by a pair of integrable Lagrangian subspaces $\scC=\scD_1\oplus\scD_2.$ Locally, one can consider the leaf spaces of the integral manifolds of $\scD_1$ and $\scD_2,$ denoted by $M_2$ and $M_1,$ respectively. 
For $n\geq 2$, there is a local correspondence between (analytic) para-CR structures and point equivalence classes of certain overdetermined systems of 2nd order PDEs. By definition, 3-dimensional para-CR structures coincide with path geometries on surfaces, which locally correspond to point equivalence classes of scalar 2nd order ODEs.

In \ref{sec:solut-equiv-probl} we give a solution of the equivalence problem for path geometries and para-CR structures in the form of certain types of Cartan geometries. The fundamental invariants of path geometries on manifolds of dimension $n\geq 3,$ are given by a torsion $\bT$ and a curvature $\bC$. In dimension 3, $\bT$ and $\bC,$ as $\mathrm{GL}(2)$-modules, can be represented by a binary quadric and a binary quartic, respectively. 

In \ref{sec:dancing-construction},  we start our study of Lewy curves in para-CR geometry by defining them in terms of a para-CR defining function in \ref{sec:LewyDef}. In \ref{sec:intrinsic} we describe para-CR Lewy curves in terms of intersections of the Legendrian leaves of $\scD_1$ and $\scD_2.$ 
In \ref{sec:examples} we describe systems of 2nd order ODEs that correspond to  path geometries defined by Lewy curves.

Our first main result is in \ref{sec:characterization-dancing}, where we characterize path geometries on 3-dimensional manifolds defined by Lewy curves of a 3-dimensional (real or complex) para-CR structure. We prove the following.
\theoremstyle{plain}
\newtheorem*{thmA}{\bf Theorem \ref{thm2}}
\begin{thmA}
Let $(Q,\scV,\scX)$ be a 3-dimensional path geometry on $N=Q/\scV$. Then $(Q,\scV,\scX)$ is defined by Lewy curves of a para-CR structure if and only if the vertical bundle $\scV$ has a splitting $\scV=\scV_1\oplus \scV_2$ with the following properties:
\begin{itemize}
\item[(a)] The rank-3 distributions
$\scB_1= [\scX, \scV_1]$ and $ \scB_2= [\scX,\scV_2]$ spanned by Lie brackets of sections of $\scX$ and $\scV_i$, respectively, are integrable.
\item[(b)]  $\scB_1$ and $\scB_2$ have rank-2 sub-distributions $\tilde\scK_1$ and $\tilde\scK_2$ containing $\scV_1$ and $\scV_2$, respectively, such that rank 3 distributions 
 $\scK_1=\tilde \scK_1\oplus\scV_2$ and $\scK_2=\tilde \scK_2\oplus\scV_1$ are integrable.
\item[(c)] The projections of $\scK_1$ and $\scK_2$ to $N$ span a contact distribution.
\item[(d)] The three para-CR 3-manifolds in Proposition \ref{prop:freestyle} are equivalent.
\end{itemize}
\end{thmA}
When the 3-dimensional path geometry has non-zero torsion, the splitting $\scV=\scV_1\oplus\scV_2$ in the theorem above is encoded in the torsion $\bT$  by viewing it as a linear map $\bT\colon\scV\to\scV$ whose  eigenspaces are  $\scV_1$ and $\scV_2$. Consequently, conditions (a)-(c) can be explicitly verified for any given path geometry. On the other hand, condition (d) in Theorem \ref{thm2} is in general difficult  to verify.

In Theorem \ref{thm4} it is shown that  Lewy curves on para-CR manifolds of dimension $\geq 5$ in general define higher path geometries i.e. they correspond to ODE systems of order greater than two. More precisely, we prove the following.
\theoremstyle{plain}
\newtheorem*{thmB}{\bf Theorem \ref{thm4}} 
\begin{thmB}
Given a para-CR structure $(N,\scD_1,\scD_2)$ of dimension $2n+1\geq 5$, its Lewy curves define  a path geometry on $N$ if and only if  it is flat, in which case they coincide with chains.
\end{thmB}
The Theorem above can be taken as a stronger version of  Faran's observation in \cite{Faran-Lewy} where it is  shown that if in a CR structure  Lewy curves coincide with chains then the CR structure is flat.

In \ref{sec:reduct-princ-bundle} we use the Cartan geometric description of para-CR structures and path geometries in dimension three to translate our characterization in Theorem \ref{thm2} in terms of the Cartan connection and the fundamental invariants of a path geometry.  Subsequently, one is lead to consider the two cases of zero or nonzero torsion. The  torsion-free case is considered in \ref{sec:path-geometry-lewy-torsionfree}.  This case allows us to understand the relation between chains and Lewy curves on para-CR 3-manifolds. 
\theoremstyle{plain} 
\newtheorem*{corC}{\bf Theorem \ref{cor:3d-path-geometries-dancing-chains}} 
\begin{corC}
A path geometry defined by Lewy curves of a para-CR structure arises via chains of some (para-)CR structure if and only if the para-CR structure is flat.
\end{corC}
The case of nonzero torsion is treated in \ref{sec:non-trivial-torsion} wherein  Theorem   \ref{thm2} is expressed Cartan geometrically, resulting in constraints on the algebraic type of $\bT$ and $\bC$ as a binary quadric and quartic. Below is an excerpt of Corollary \ref{cor:3d-path-geometries-dancing} about such constraints in the case of \emph{real} para-CR 3-manifolds. 
\theoremstyle{plain}
\newtheorem*{corB}{\bf Corollary \ref{cor:3d-path-geometries-dancing}}
\begin{corB}
If a 3-dimensional path geometry with non-zero torsion, i.e. $\bT\neq 0,$ is defined by para-CR Lewy curves then the binary quadric representation of $\bT$ has two distinct real roots and the binary quartic representation of $\bC$ has at least two distinct real roots and no repeated roots.
\end{corB}
In the case of complex para-CR 3-manifolds,  the  conditions in Corollary \ref{cor:3d-path-geometries-dancing}  on the multiplicity of roots remain valid, unlike the reality conditions. 
Comparing the algebraic type of $\bC$ in Corollary \ref{cor:3d-path-geometries-dancing} and the algebraic type of $\bC$ for path geometries defined by para-CR chains in \cite[Theorem 1.1(1)]{KM-chains} shows a clear difference between such path geometries.
Subsequently, we show that the path geometry of Lewy curves allows one to recover the underlying para-CR structure, as has been shown for (para-)CR chains \cite{cheng,CZ-CR}.
\theoremstyle{plain} 
\newtheorem*{corD}{\bf Corollary \ref{cor:determination}} 
\begin{corD}
The path geometry of para-CR Lewy curves in dimension three determine the underlying para-CR structure up to the sign of the almost para-complex structure. 
\end{corD}

 In Appendix, structure equations  \eqref{eq:freestyling-streqs}   obtained by following the  proof of Proposition \ref{prop:torsion-curv-type-freestyling-reduction} is given.

\subsection{Conventions}\label{sec:conventions}

Our consideration in this paper will be over smooth and real or complex manifolds. We will work in the real category. The complex counterpart of our statements will be clear in most cases, and if there is a difference  we will clarify, which will only be limited to \ref{sec:path-geom-aris} when we discuss reality conditions for the roots of the torsion and curvature. As will be mentioned in the text, in the case of para-CR structures of dimension $2n+1$,$n\geq 2$, in order to define Lewy curves we need to assume analyticity. Throughout the article we always consider path geometry in the generalized sense, introduced in \ref{sec:path-geom-defin}, and therefore will not use the term ``generalized'' when talking about path geometries. When defining the leaf space of a foliation we always restrict to sufficiently small open sets where the leaf space is smooth and Hausdorff. Consequently, given an integrable distribution $\scD$ on a manifold $N$, by abuse of notation, we denote the leaf space of its induced foliation by $N/\scD$.

 When dealing with a $k$-dimensional subspace of a vector space, e.g.  $K\subset V,$ the corresponding element in the Grassmannian of $k$-planes is denoted by $[K]\in\mathrm{Gr}_{k}(V).$  

We will use the summation convention over  repeated upper and lower indices. Given  elements  $v_1,\dots,v_k$ of a vector space, their span is denoted by $\langle v_1,\dots,v_k\rangle.$ When dealing with differential forms, the algebraic ideal generated by 1-forms $\alpha^1,\dots,\alpha^k$ is denoted as $\{\alpha^1,\dots,\alpha^k\}.$  Given a set of 1-forms $\alpha^0,\dots,\alpha^n,\beta^1,\dots,\beta^n,$ the corresponding dual frame is denoted as $\frac{\partial}{\partial\alpha^0},\cdots \frac{\partial}{\partial\beta^n}.$ On a principal bundle $\cG\to Q$ with respect to which the 1-forms $\alpha^0,\dots,\alpha^n,\beta^1,\dots,\beta^n$ give a basis for semi-basic 1-forms,   we define the \emph{coframe derivatives} of a function $f\colon\cG\to \RR$  as 
\begin{equation*}
  f_{;i}=\tfrac{\partial}{\partial\alpha^i}\im\exd f,\   f_{;ij}=\tfrac{\partial}{\partial\alpha^j}\im\exd f_{;i},\ f_{;\underline{a}}=\tfrac{\partial}{\partial\beta^a}\im\exd f,\  f_{;\underline{ab}}=\tfrac{\partial}{\partial\beta^b}\im\exd f_{;\underline a},\  f_{;\underline{a}i}=\tfrac{\partial}{\partial\alpha^i}\im\exd f_{;\underline a},\ f_{;i\underline{a}}=\tfrac{\partial}{\partial\beta^a}\im\exd f_{;i} 
\end{equation*}
and similarly for higher orders, where $0\leq i,j\leq n$ and $1\leq a,b\leq n.$ Note that in case we reduce the structure bundle of a geometric structure to a proper principal sub-bundle, by abuse of notation, we suppress the pull-back and use the same notation as above for the coframe derivatives on the reduced bundle. 

Lastly, given two distributions $\scD_1$ and $\scD_2,$ we denote by $[\scD_1,\scD_2]$ the distribution whose sheaf of sections is  $\Gamma([\scD_1,\scD_2])=\Gamma(\scD_1)+\Gamma(\scD_2)+[\Gamma(\scD_1),\Gamma(\scD_2)].$ The sheaf of sections of  $\bigwedge^k(T^* M)$ is denoted by $\Omega^k(M).$

\section{Preliminaries on  path geometries and para-CR structures}\label{sec:path-geom-review}

In this section, we  overview some well-known facts about path geometries and para-CR structures. In particular, after defining these structures, we  show how they correspond to a class of differential equations. We point out that 3-dimensional para-CR structures coincide with path geometries on surfaces. Lastly, we give the solution of the equivalence problem of these structure using a Cartan connection. All the results and constructions will be stated in the real setting from which the complex analogues can be easily found. 

\subsection{Path geometries and  systems of 2nd order ODEs} \label{sec:path-geom-defin} 
A path geometry  on an $(n+1)$-dimensional manifold $M$ is classically defined as a $2n$-parameter family of paths with the property that along each direction at any point of $M,$ there passes a unique path in that family. Consequently, the natural lift of the paths of a path geometry to the projectivized tangent bundle, $\PP TM,$ results in a foliation by curves which are transverse to the fibers $\PP T_xM.$

Path geometries are modeled by the flag variety of pairs of incident lines and 2-planes i.e. 
\begin{equation}\label{eq:flat-path-geom}
  \FF_{1,2}(\RR^{n+2}):=\{(\ell,P)\ \vline\ \ell\subset P, [\ell]\in\PP^{n+1},[P]\in\mathrm{Gr_2(\RR^{n+2})}\}\cong G\slash P_{1,2}
  \end{equation}
where $G=\mathrm{PGL}(n+2,\RR),$ $[\ell]$ and $[P]$ denote the line $\ell$ and the 2-plane $P$ as an element of  the projective space and the Grassmannian of 2-planes, respectively,  and $P_{1,2}$ is the parabolic subgroup which is the stabilizer of a chosen origin. Alternatively, $\FF_{1,2}(\RR^{n+2})$ can be identified with $\PP T\PP^{n+1},$  i.e.  the projectivized tangent bundle of the projective space $\PP^{n+1},$ equipped with the foliation given by the canonical lift of projective lines.

It is well-known that a path geometry on an $(n+1)$-dimensional manifold can be locally defined in terms of  a system of $n$ second order ODEs 
\begin{equation}\label{systemODE}
   (z^i)''=F^i(t,z,z'),\quad t\in\R,\ \ z=(z^1,\dots,z^n),\quad 1\leq i\leq n,
\end{equation}
defined up to \emph{point transformations},  i.e.  
\[t\mapsto \tilde t=\tilde t(t,z^1,\ldots,z^{n}),\quad z^i\mapsto \tilde z^i=\tilde z^i(t,z^1,\ldots,z^{n}),\quad 1\leq i\leq n.\]
 Given a path geometry, its straightforward to see how it defines a  system of $n$ 2nd order ODEs: let $(z^0,\dots,z^n)$ be local coordinates on $V\subset M$. In a sufficiently small open set $U\subset \PP TM,$ the family of paths can be parametrized as $s\mapsto \gamma(s)=(z^0(s),\dots,z^n(s))$ for $s\in (a,b)\subset\RR.$ Each path is determined by   $\gamma(s)$ and $\gamma'(s)$ at $s=s_0\in U.$ Thus, taking another derivative, a system of $n+1$ 2nd order ODEs of the form $\gamma''= G(\gamma,\gamma')$ is obtained for a function   $G\colon \RR^{2n+2}\to\RR^{n+1}.$ If $U\subset \PP TM$ is sufficiently small, without loss of generality one can assume $\frac{\exd z^0}{\exd s}\neq 0$ in $U.$  Since the paths are given up to reparametrization, one is able to eliminate $s$ from the system $\gamma''=G(\gamma,\gamma').$ As a result, one arrives at the system of ODEs \eqref{systemODE} where $t:=z^0.$

Conversely, starting with an  equivalence class of system of $n$ 2nd order  ODEs under point transformations  \eqref{systemODE}, the system defines  a codimension $n$ submanifold $\cE\subset J^2(\RR,\RR^n)$ of the 2-jet space of  $n$ functions of 1 variable.  Pulling-back  the natural contact system on $J^2(\RR,\RR^n)$ to $\cE,$ one can identify $\cE$ with $J^1(\RR,\RR^n)$ which is additionally foliated by contact curves, i.e. the solution curves of the ODE system.  Locally, $J^1(\RR,\RR^n)\to J^0(\RR,\RR^n)\cong\RR^{n+1}$ can be identified as an open subset of the projectivized tangent bundle $\PP TM\to M$ for an $(n+1)$-dimensional manifold $M$ and the solution curves project to a $2n$-parameter family of paths on $M.$ However, for arbitrary ODE systems it may happen that  no path is tangent to some directions at some points of $M$. 

One can easily find  path geometries that do not fit the classical description since the paths are only defined along  an open set of directions. Thus, in order to study  path geometries one is led to work with a generalized notion of such structures as defined below, which is sometimes referred to as \emph{generalized path geometry}. However, in this article since path geometries for us are always defined in this generalized sense we will not use the term generalized.
\begin{definition}
\label{def:generalized-path-geom}
An $(n+1)$-dimensional path geometry  is given by a triple $(Q,\scV,\scX)$   where $Q$ is a $(2n+1)$-dimensional manifold  equipped with a  pair of distributions $(\scX,\scV)$ of rank 1 and $n,$ respectively,  which intersect trivially and satisfy $[\scV,\scV]=\scV$ and $[\scX,\scV]=T Q$. The $(n+1)$-dimensional local leaf space of the foliation induced by $\scV,$ denoted as $Q\slash \scV,$ is said to be equipped with the path geometry $(Q,\scV,\scX)$. 
\end{definition}
Two path geometries $(Q_i,\scX_i,\scV_i),$ $i=1,2$ are called (locally) equivalent if there exists  a (local)  diffeomorphism $f\colon Q_1\to Q_2$ such that $f_*(\scX_1)=\scX_2$ and $f_*(\scV_1)=\scV_2.$  

The rank $(n+1)$  distribution spanned by $\scX$ and $\scV$  induces a    \emph{multi-contact} structure on $Q,$ i.e. one can write  $\scX\oplus\scV=\ker\{\alpha^1,\ldots,\alpha^n\}$ for some 1-forms $\alpha^0,\alpha^a,\beta^b,$ 
such that  $\scX=\ker\{\alpha^a,\beta^b\}_{a,b=1}^n,$ $\scV=\ker\{\alpha^0,\ldots,\alpha^n\}$  and 
\[\exd \alpha^i\equiv \alpha^0\w\beta^i\ \ \mathrm{mod\ \ }\{\alpha^1,\ldots,\alpha^n\},\]
for all $1\leq i\leq n$. 

One can easily check that Definition \ref{def:generalized-path-geom} is satisfied for a classical path geometry on an $(n+1)$-dimensional manifold by letting $Q=\PP TM$ and $\scV$ be the vertical tangent space of the fibration $\PP TM\to M$  and $\scX$ be the line field tangent to the canonical lift of  paths on $M$ to $\PP TM.$  In terms of the point equivalence class of a system of second order ODEs \eqref{systemODE} one has
\begin{equation}
  \label{eq:TotalDerivative}
  \scX=\langle X_F\rangle\quad \text{where}\quad X_F:=\partial_t+ p^i\partial_{z^i}+ F^i\partial_{p^i},
\end{equation}
and
\[
  \scV=\langle\partial_{p^1},\ldots,\partial_{p^n}\rangle,
\]
where $p^i$'s are fiber coordinates for an affine chart of  $\PP TM\to M$. The vector field $X_F$ spanning $\scX$ is often called the total derivative vector field.

As was mentioned before, not all geometries arising from Definition \ref{def:generalized-path-geom} are classical path geometries. However, restricting to  sufficiently small open sets   $U\subset Q$ in Definition \ref{def:generalized-path-geom},  $U$ can be realized as an open set of $\PP TM$ for the $(n+1)$-dimensional manifold $M$ which is the leaf space of $\scV.$ Consequently,  $\scX$ foliates $U\subset\PP TM$ by curves that are transverse to the fibers of $\PP TM\rightarrow M.$ We refer the reader to \cite[Section 2]{Bryant-ProjFlat} for more about path geometry on surfaces and \cite[Section 4.4.2,4.4.4]{CS-Parabolic}  in higher dimensions.

\subsection{Para-CR structures and their double fibration}\label{sec:paraCR} 

In this section we review the basics of para-CR structures and their double fibration. We refer the reader to  \cite[Chapter X]{BookSegre}  for a review of the topic in the  CR setting  and to  \cite[Section 2]{DoubrovThe1} in the para-CR setting wherein the para-CR structures are referred to as  integrable Legendrian contact (ILC) structures. We start by the definition of a nondegenerate para-CR structure of hypersurface type, which we will refer to simply as a para-CR structure.    

\begin{definition}\label{def:para-cr-structures}
  A  para-CR structure defined on a $(2n+1)$-dimensional manifold   $N$  with a contact distribution $\scC\subset TN,$ is given by a splitting $\scC=\scD_1\oplus\scD_2$ where $\scD_1$ and $\scD_2$ are integrable and everywhere Lagrangian with respect to the induced conformal symplectic structure on $\scC.$ Such a para-CR structure will be denoted by $(N,\scD_1,\scD_2).$
\end{definition}
The definition above suggests a natural notion of duality between para-CR structures, as given in Definition \ref{def:duality-para-CR}. Two para-CR structures $(N,\scD_1,\scD_2)$ and $(\tilde N,\tilde\scD_1,\tilde\scD_2)$ are (locally) equivalent if there exists  a (local) diffeomorphism $f\colon N\to\tilde N$ such that $f_*(\scD_a)=\tilde\scD_a$ for $a=1,2.$

Note that the integral manifolds of $\scD_1$ and $\scD_2$ are Legendrian submanifolds of $N.$ The terminology para-CR structure has to with the fact that, in the real setting, $\scD_1$ and $\scD_2$ are  integrable $\pm 1$-eigenspaces of an \emph{almost para-complex structure} $K\colon \scC\to \scC,$ i.e. $K^2=\mathrm{Id}.$ Furthermore,  one has the compatibility condition $\cL(KX,KY)=-\cL(X,Y)$ , for all $X,Y\in\Gamma(\scC),$ where $\cL\colon\bigwedge^2\scC\to TN\slash\scC$ is the Levi bracket of $\scC.$ As a result, para-CR structures are the para-complex analogues of nondegenerate CR structures of hypersurface type.

Para-CR geometries are modeled by the flag variety of pairs of incident lines $\ell$ and hyperplanes $H,$ i.e. 
\begin{equation}\label{eq:flat-para-CR}
 G\slash P_{1,n+1}  \cong \FF_{1,n+1}(\RR^{n+2}):=\{(\ell,H)\ \vline\ \ell\subset H\}\subset  \PP^{n+1}\times(\PP^{n+1})^*,
  \end{equation}
where   $G=\mathrm{PGL}(n+2,\RR)$ and $P_{1,n+1}$ is the parabolic subgroup which is the stabilizer of a chosen origin. Alternatively, $\FF_{1,n+1}(\RR^{n+2})$ can be identified with $\PP T^*\PP^{n+1},$  i.e.  the projectivized cotangent bundle of the projective space $\PP^{n+1},$ equipped with its canonical contact structure and  Legendrian foliations given by the fibers of the  natural fibrations  $\PP T^*\PP^{n+1}\to\PP^{n+1}$ and $\PP T^*\PP^{n+1}\to(\PP^{n+1})^*$.

If a para-CR structure is not flat then one can locally define $M_1$ and $M_2$ as  the $(n+1)$-dimensional (local) leaf space of the integral manifolds of $\scD_2$ and $\scD_1,$ respectively. As a result, one has the double fibration
\begin{equation}\label{eq:double-fibration}
  N\slash\scD_2 =:M_1\overset{\pi_1}\longleftarrow N\overset{\pi_2}\longrightarrow M_2:=N\slash\scD_1,
\end{equation}
where $(\pi_1)_*(\scD_2)=0$ and $(\pi_2)_*(\scD_1)=0.$
\begin{proposition}\label{prop:jet-realization-para-cr-structures}
  Given a $(2n+1)$-dimensional para-CR manifold $(N,\scD_1,\scD_2),$  there is a local equivalence  $N\cong J^1(\RR^{n},\RR)$ together with  an embedding $J^1(\RR^{n},\RR)\hookrightarrow J^2(\RR^{n},\RR)$ given by the point equivalence class of a PDE system
  \begin{equation} \label{eq:pde-system-para-CR}
    \frac{\partial^2 z}{\partial t^i\partial t^j}=f_{ij}(t,z,\partial z),  
  \end{equation}
  for functions $f_{ij}=f_{ji}$ satisfying the compatibility condition
  \[\tfrac{\mathrm{D}}{\exd t^i}f_{jk}(t,z,p)=\tfrac{\mathrm{D}}{\exd t^j}f_{ik}(t,z,p),\]
  such that the fibers of $N\to M_1$ and $J^1(\RR^n,\RR)\to J^0(\RR^n,\RR)$ are identified     where variable $(t,z,p)=(t^1,\ldots,t^{n},z,p_{1},\ldots,p_{n})$ are  Darboux coordinates on $J^1(\RR^{n},\RR)$ with respect to which the contact form is given  by $\exd z-p_i\exd t^i,$    and  $\tfrac{\mathrm{D}}{\exd t^i}=\partial_{t^i}+p_i\partial_z+f_{ij}\partial_{p_j}$.
\end{proposition}
\begin{proof}
  Locally, as contact manifolds, one has a local diffeomorphism $\phi\colon N\to J^1(\RR^n,\RR)$ such that $\phi_*(\scD_1)=\scV$ where $\scV:=\langle\partial_{p^1},\ldots,\partial_{p^n}\rangle$ are vertical with respect to $J^1(\RR^n,\RR)\to J^0(\RR^n,\RR),$ and the contact form is given by $\exd z-p_i\exd t^i.$  Since $\scC=\scD_1\oplus\scD_2,$ one has $\phi_*(\scD_2)=\scE,$ where
  \begin{equation}\label{eq:E-distribution}
    \scE=\langle\partial_{t^i}+p_i\partial_z+f_{ij}\partial_{p_j}\rangle_{i=1}^n,
  \end{equation}
  for some functions $f_{ij}=f_{ji}$ on $J^1(\RR^n,\RR).$ It is straightforward to show that the integrability of $\scD_2$ is equivalent to $\tfrac{\mathrm{D}}{\exd t^i}f_{jk}=\tfrac{\mathrm{D}}{\exd t^j}f_{ik}.$ The existence of the functions $f_{ij}$ are equivalent to a section of $ J^2(\RR^n,\RR)\to J^1(\RR^n,\RR)$ given by the PDE system \eqref{eq:pde-system-para-CR}. 
\end{proof}
\begin{remark}\label{rmk:N-open-set} 
  For the purposes of this article, we will most often  restrict ourselves to sufficiently small open sets of  $N$ where a local coordinate system given in Proposition \ref{prop:jet-realization-para-cr-structures} exists. 
\end{remark}

Notice that the double fibration \eqref{eq:double-fibration} establishes a correspondence between points in $M_1$ and hypersurfaces in $M_2$, and conversely, points in $M_2$ and hypersurfaces in $M_1$. For a given $x\in M_1$ the corresponding hypersurface $H_x\subset M_2$ is defined as the  projection to $M_2$ of the integral leaf of $\scD_2$ corresponding to $x$, i.e. $H_x=\pi_2(\pi_1^{-1}(x))$. Similarly, for $y\in M_2$, we denote $H_y=\pi_1(\pi_2^{-1}(y))\subset M_1$.  
\begin{definition}\label{def:incidence}
Given a para-CR structure with double fibration \eqref{eq:double-fibration}, a pair $(x,y)\in M_1\times M_2$ is called incident if $x\in H_y$, or equivalently, $y\in H_x$. Otherwise, the pair is called non-incident.
\end{definition}

Proposition \ref{prop:jet-realization-para-cr-structures} describes the correspondence between points and hypersurfaces in terms of the PDE system \eqref{eq:pde-system-para-CR}. Assuming analyticity, the solutions $z=z(t)$ of system \eqref{eq:pde-system-para-CR}, where $t=(t^1,\ldots,t^n),$ are determined by the value of their 1st jet at a point. Thus each solution defines a hypersurface in $M_1$ and a Legendrian submanifolds of $N\cong J^1(\RR^n,\RR).$  Consequently, such a PDE system is of finite type whose solutions depend on $n+1$ parameters.  Its general solution can be written as a function
\begin{equation}\label{eq:Phi-definiting-function}
  \Phi(z,t^i,a_i,b)=0,
\end{equation}
where $(a^1,\ldots,a^n,b)$ are local coordinates for the solution space of the PDE system \eqref{eq:pde-system-para-CR}. Since by our previous discussion each solution  \eqref{eq:pde-system-para-CR} corresponds to an integral leaf of $\scD_1,$ the solution space can be identified as the leaf space of $\scD_1,$ which in \eqref{eq:double-fibration}, was denoted by $M_2.$

Using the general solution  \eqref{eq:Phi-definiting-function}, ones obtains a defining function for  $N$ viewed as a hypersurface in $M_1\times M_2$ given by
\begin{equation}\label{eq:defining-function-for-N}
  N=\{(x, y)\in M_1\times M_2 \ \vline\ \Phi(x,y)=0 \},
  \end{equation}
where $x=(t^i,z)$  and $y=(a^i,b)$ are local coordinate systems for $M_1$ and $M_2,$ respectively. In terms of a defining function
\[
H_x=\{y\in M_2\ | \ \Phi(x,y)=0\}\quad  \mathrm{and}\quad H_y=\{x\in M_1\ |\ \Phi(x,y)=0\}.
\]
The set of non-incident pairs $(x,y)$ is thus given as $M_1\times M_2\setminus N$.

\begin{remark}\label{rmk:para-cr-structures-2D-path}
  By Proposition \ref{prop:jet-realization-para-cr-structures}, $N$ can be locally identified as an open subset of the  projectivized cotangent bundle of $M_1$ or $M_2,$ denoted as $\PP T^* M_1$ and $\PP T^*M_2$.
Moreover, comparing Definitions \ref{def:generalized-path-geom} and \ref{def:para-cr-structures}, it is easy to see that para-CR 3-manifolds in which the contact distribution has a splitting into two line fields, is equivalent to 2-dimensional path geometries.     As a result, in dimension three $N$ can be locally identified as   $\PP TM_1$ or $\PP TM_2.$

Note that in the case of 3-dimensional para-CR structure the system of PDEs \eqref{eq:pde-system-para-CR} become a scalar 2nd order ODE
\begin{equation}\label{eq:sode}
\frac{\exd^2z}{\exd t^2}=f\left(t,z,\frac{\exd z}{\exd t}\right).
\end{equation}
Otherwise, the PDE system is overdetermined. By the local existence of solutions for ODEs, one does not need to assume analyticity in dimension 3 in order to have $N\subset M_1\times M_2$ in \eqref{eq:defining-function-for-N}.
\end{remark}

Lastly, we define a notion of duality among para-CR structures arising from  the jet realization of Proposition \ref{prop:jet-realization-para-cr-structures}. It is the para-CR analogue of \emph{complex conjugation} in the CR setting. 
\begin{definition}\label{def:duality-para-CR}
Given a para-CR structure $(N,\scD_1,\scD_2),$ its dual is the para-CR structure $(N,\scD_2,\scD_1).$  A para-CR manifold is called self-dual if its equivalent to its dual.
\end{definition} 
As a result, one can define (local) \emph{anti-automorphisms} of a para-CR structure as (local) diffeomorphisms $f\colon N\to N$ such that $f_*(\scD_1)=\scD_2$ and $f_*(\scD_2)=\scD_1.$

The study of duality among 3-dimensional para-CR structures goes back to Cartan and can be formulated in terms of duality among scalar second order ODEs. From Proposition \ref{prop:jet-realization-para-cr-structures} it immediately follows that the dual para-CR structure  corresponds to the point equivalence class of a PDE system
\begin{equation} \label{eq:dual-pde-system-para-CR}
\frac{\partial^2 b}{\partial a^i\partial a^j}=\tilde f_{ij}(a,b,\partial b),  
\end{equation}
for some functions $\tilde f_{ij}=\tilde f_{ji}$ satisfying the appropriate integrability condition.  
In the 3-dimensional case one obtains  the scalar dual ODE
\begin{equation}\label{eq:dual}
\frac{\exd^2 b}{\exd a^2}=\tilde f\left(a,b,\frac{\exd b}{\exd a}\right),
\end{equation}

\begin{example}\label{exa:flat-para-CR}
  In a flat CR structure one can set $f_{ij}=0$ in \eqref{eq:pde-system-para-CR} which results in the PDE system $z_{ij}=0$ whose solutions are given by
  \begin{equation}\label{eq:flat-para-CR-PDE}
    z=a_it^i+b,
      \end{equation}
  where  $(a_i,b)$ parameterize solutions. Hence    a defining function \eqref{eq:Phi-definiting-function} is given by
  \begin{equation}\label{exa:flat-para-CR-function}
  \Phi(z,t^i,a_i,b)=z-a_it^i-b.
  \end{equation}
  Note that in \eqref{eq:flat-para-CR-PDE}  one can treat $b$ as a function of independent variables $a_i.$ Differentiating twice and eliminating $z$ and $t^i$'s, one obtains the PDE system $b_{ij}=0,$ which shows that the flat para-CR structure is self-dual. 
\end{example}
\begin{example}\label{exa:self-dual-non-flat}
Since  3-dimensional para-CR structures correspond to point equivalence classes of   scalar 2nd order ODEs, consider the ODE   $z''=\sqrt{z'}.$ Solving explicitly, one   obtains a defining function to be 
\begin{equation}\label{exa:self-dual-non-flat-function}
\Phi(z,t,a,b)= z-(\tfrac{1}{12}t^3+\tfrac 14 b t^2+\tfrac 14b^2t+a).
\end{equation}
\end{example}

\subsection{Cartan connection for path geometries and para-CR structures}\label{sec:solut-equiv-probl}  
In this section we present a solution to the equivalence problem of path geometries that is due to Cartan \cite{Cartan-Proj} for surfaces and due to Grossman and Fels  \cite{Grossman-Thesis,Fels} in higher dimensions. Similarly, for para-CR structures there is a canonical Cartan connection, defined via the parabolic geometric machinery, which has been studied in \cite{DoubrovThe1}. First we recall  the definition of a Cartan geometry. 
\begin{definition}
  Let $G$ be a Lie group and $P\subset G$ a Lie subgroup  with Lie algebras $\fg$ and $\fp\subset\fg,$ respectively.  A Cartan geometry of type $(G,P)$ on $Q,$ denoted as $(\cG\to Q,\psi),$ is a right principal $P$-bundle $\tau\colon\cG\to Q$ equipped with a Cartan connection $\psi\in\Omega^1(\cG,\fg),$ i.e. a $\fg$-valued 1-form on $\cG$ satisfying
  \begin{enumerate}
  \item $\psi$ is $P$-equivariant, i.e. $r_g^*\psi=\mathrm{Ad}_{g^{-1}}\circ\psi$ for all $g\in P.$
  \item  $\psi_z\colon T_z\cG\to \fg$ is  a linear isomorphism for all $z\in \cG.$
  \item $\psi$ maps fundamental vector fields to their generators, i.e. $\psi(\zeta_X)=X$ for any $X\in\fp$ where $\zeta_X(z):=\frac{\exd}{\exd t}\,\vline_{\,t=0}r_{exp(tX)}(z).$
  \end{enumerate} 
 The 2-form $\Psi\in\Omega^2(\cG,\fg)$ defined as
    \[\Psi(u,v)=\exd\psi(u,v)+[\psi(X),\psi(Y)]\quad \text{for\ \ }X,Y\in \Gamma(T\cG),\]
is called the Cartan curvature and is $P$-equivariant and semi-basic with respect to the fibration $\cG\to Q.$
\end{definition} 

Now we give a solution to the equivalence problem of path geometries in dimensions larger than 2.
\begin{theorem}\label{thm:path-geom-cartan-conn}
Every path geometry $(Q^{2n+1},\scX,\scV),$  as in Definition \ref{def:generalized-path-geom},   defines Cartan geometry  $(\cG\to Q,\psi)$ of type $(\mathrm{PGL}(n+2,\RR),P_{1,2})$ where $P_{1,2}\subset \mathrm{PGL}(n+2,\RR)$ is the parabolic subgroup preserving the flag of a line and a 2-plane in $\RR^{n+2}.$ Assume that  the distributions $\scX$ and $\scV$  are the projection of the distributions  $\ker\{\alpha^i,\beta^i\}_{i=1}^n$ and $\ker\{\alpha^i\}_{i=0}^n,$ respectively.   When $n\geq 2$, the Cartan connection and its curvature can be  expressed as
  \begin{equation}
  \label{eq:path-geom-cartan-conn-3D}  
  \psi=
  \def\arraystretch{1.3}
\begin{pmatrix}    
 -\psi^i_i-\psi^0_0 &\mu_0& \mu_j\\
\alpha^0 &\psi^0_0&\nu_j\\
\alpha^i&  \beta^i& \psi^i_j\\
\end{pmatrix}    
\quad \mathrm{and}\quad \exd\psi+\psi\w\psi=
  \def\arraystretch{1.3}
\begin{pmatrix}    
0 &M_0& M_j\\
0 &\Psi^0_0&V_j\\
0 &   B^i& \Psi^i_j\\
\end{pmatrix},    
    \end{equation}
where $1\leq i,j\leq n$ and 
\[B^i=T^i_j\alpha^0\w\alpha^j+\half T^i_{jk}\alpha^j\w\alpha^k,\qquad \mathrm{and}\quad \Psi^i_j=C^i_{jkl}\alpha^k\w\beta^l+\half T^i_{jkl}\alpha^k\w\alpha^l +T^i_{j0k}\alpha^0\w\alpha^k. \]
The fundamental invariants of a path geometry  are the \emph{torsion}, $\mathbf{T}=(T^i_j)_{1\leq i,j\leq n},$ and the \emph{curvature}, $\mathbf{C}=(C^i_{jkl})_{1\leq i,j,k,l\leq n}$, satisfying
\begin{equation}
  \label{eq:fels-invariants-path}
  T^i_i=0,\qquad C^i_{jkl}=C^i_{(jkl)},\qquad C^i_{ijk}=0.
\end{equation}
\end{theorem}

If $(t,z^i)$ and $(t,z^i,p^i)$ are local jet coordinates on $\RR\times \RR^{n}=J^0(\RR,\RR^n)$ and $J^1(\RR,\RR^n),$ respectively, then in terms of the system of ODEs \eqref{systemODE}, one obtains 
\begin{equation}
  \label{eq:fels-invariants-explicit}
  T^i_j=F^i_j-\textstyle{\frac{1}{2}\delta^i_jF^k_k},\qquad\qquad C^i_{jkl}=F^i_{jkl}-\textstyle{\frac{3}{4}}F^r_{r(jk}\delta^i_{l)}
\end{equation}
where $1\leq i,j,k,l\leq n$ and 
\begin{equation}
  \label{eq:fels-torsion-explicit-components}
  F^i_{j}=\textstyle{- F^i_{z^j}+\frac{1}{2}\Dt (F^i_{p^j})-\frac{1}{4} F^i_{p^k} F^k_{p^j},}\qquad F^i_{jkl}=F^i_{p^jp^kp^l},
\end{equation}
and, as before, $X_F=\partial_t+p^i\partial_{z^i}+F^i\partial_{p^i}$ is the total derivative vector field of the system.

Before stating Cartan's solution of the equivalence problem of path geometries on surfaces, which by Remark \ref{rmk:para-cr-structures-2D-path}, coincides with the one for para-CR 3-manifolds, we state the solution of equivalence problem for para-CR structures, which, as a class of parabolic geometries, were treated in \cite{CapCrelle,DoubrovThe1}.
\begin{theorem}\label{thm:paraCR-cartan-conn}
Every para-CR structure $(N^{2n+1},\scD_1,\scD_2),$  as in Definition \ref{def:para-cr-structures},   defines Cartan geometry  $(\cP\to N,\phi)$ of type $(\mathrm{PGL}(n+2,\RR),P_{1,n+1})$ where $P_{1,n+1}\subset \mathrm{PGL}(n+2,\RR)$ is the parabolic subgroup preserving the flag of a line in a hyperplane in $\RR^{n+2}.$ Assume that  the distributions $\scD_2,\scD_1\subset TN$  are the projection of the distributions  $\ker\{\omega^i\}_{i=0}^n$  and $\ker\{\omega^0,\omega^{n+a}\}_{a=1}^n,$ respectively.   When $n\geq 2$, the Cartan connection and its curvature can be  expressed as
  \begin{equation}
  \label{eq:paraCR-geom-cartan-conn-3D}  
  \phi=
  \def\arraystretch{1.3}
\begin{pmatrix}    
 -\phi^i_i-\phi^{n+1}_{n+1} &\theta_{j+n}& \theta_0\\
\omega^i &\phi^i_j&\theta^i\\
\omega^0&  \omega_{j+n}& \phi^{n+1}_{n+1}\\
\end{pmatrix}    
\quad \mathrm{and}\quad \exd\phi+\phi\w\phi=
  \def\arraystretch{1.3}
\begin{pmatrix}    
0 &\Theta_{j+n}& \Theta_0\\
0 &\Phi^i_j&\Theta^i\\
0 &   0& 0 \\
\end{pmatrix},    
    \end{equation}
where $1\leq i,j\leq n$ and 
\[\Phi^i_j\equiv K^{il}_{jk}\omega^k\w\omega_{l+n}\quad \mod\quad \{\omega^0\}. \]
The fundamental invariant of a para-CR is given by the structure functions $(K^{il}_{jk})_{1\leq i,j,k,l\leq n}$ which are symmetric in $(i,j)$ and $(k,l)$ and trace-free in $(i,j).$ 
\end{theorem}

In terms of the corresponding PDE system of a para-CR structure, as described in Proposition \ref{prop:jet-realization-para-cr-structures}, in \cite{DoubrovThe1} its was shown that  the fundamental invariant  can be expressed parametrically as
\[K^{il}_{jk}=S^{il}_{jk}-\tfrac{4}{n+2}\delta^{(i}_{(j}S^{l)}_{k)}+\tfrac{1}{(n+1)(n+2)}(\delta^i_j\delta^l_k+\delta^l_j\delta^i_k)S,\]
where
\[S^{il}_{jk}=-\tfrac{\partial f_{jk}}{\partial p_{i}\partial p_l},\quad S^i_j=S^{ik}_{jk},\quad S=S^{ij}_{ij}.\]

Now we state Cartan's solution of the equivalence problem for the special case of 3-dimensional para-CR structures which coincides with path geometries on surfaces. 
\begin{theorem}\label{thm:2D-path-geome}
 Given a 3-dimensional para-CR structure  $(N,\scD_1,\scD_2),$ it defines a Cartan geometry $(\cP \to N,\phi)$ of type $(\mathrm{PGL}(3,\RR),P_{1,2})$ where $P_{1,2}$ is the subgroup of upper diagonal matrices and
  \begin{equation}
  \label{eq:2D-path-geom-cartan-conn}  
  \phi=
  \def\arraystretch{1.3}
\begin{pmatrix}    
 -\phi^1_1-\phi^2_2 & \theta_2 &\theta_0\\
 \omega^1 &\phi^1_1&\theta_1\\
\omega^0&  \omega^2& \phi^2_2\\
\end{pmatrix},
\quad \exd\phi+\phi\w\phi=  \begin{pmatrix}
    0 & P_1 \omega^0 \w\omega^1 & P_2\omega^0 \w\omega^1+Q_2\omega^0\w\omega^2\\
    0 & 0 & Q_1\omega^0\w\omega^2\\
0 & 0 & 0
  \end{pmatrix}
\end{equation}
for some functions $P_1,P_2,Q_1,Q_2$ on $\cP.$ Structure functions $P_1$ and $Q_1$ are the \emph{fundamental invariants}.
\end{theorem}
In comparison to  the matrix forms of $\phi$ in \eqref{eq:path-geom-cartan-conn-3D}, we have decided  in \eqref{eq:2D-path-geom-cartan-conn} for all $\omega$'s to have superscript and all $\theta$'s to have subscript.

As was mentioned in Remark \ref{rmk:para-cr-structures-2D-path}, para-CR 3-manifolds correspond to point equivalence classes of scalar 2nd order ODEs. In terms of  coordinates $(t,z,p)$ on $J^1(\R,\R)$ one obtains
\[
P_1=X_F^2(F_{pp})-4X_F(F_{zp})-F_{p}X_F(F_{pp})+4F_{p}F_{zp}-3F_zF_{zpp}+6F_{zz},\quad Q_1=F_{pppp},
\]
where $F$ defines an ODE $z''=F(t,z,z')$ in the point equivalence class, and $X_F=\partial_t+p\partial_{z}+F\partial_{p}.$

\begin{remark}\label{rmk:1-form-defined-on-G-or-Q}
Unlike in \ref{sec:path-geom-defin}, the 1-forms in Theorems \ref{thm:path-geom-cartan-conn}, \ref{thm:paraCR-cartan-conn}, and \ref{thm:2D-path-geome} are defined on the respective principal bundles $\cG$ and $\cP$ rather than the manifold $Q$ or $N.$   Whenever it is clear from the context that one needs to take a section $s\colon Q\to\cG$ and consider $s^*\psi,$ we will not make a distinction between 1-forms defined on the principal bundle or the underlying manifold. 
\end{remark}

A substantial part of the paper will involve path geometries on 3-dimensional manifolds, i.e.  $Q$ in Definition \ref{def:generalized-path-geom} is 5-dimensional. By the discussion in \ref{sec:path-geom-defin} such geometries correspond to point equivalence classes of pairs of 2nd order ODEs. We have the following proposition.
\begin{proposition}\label{prop:3D-path-geom}
In three-dimensional path geometries the fundamental invariants $\bT$ and $\bC$ in Theorem \ref{thm:path-geom-cartan-conn}, as $\mathrm{GL}_2(\RR)$-modules  can be represented as a  binary  quadric and a binary quartic, respectively, given by
\begin{equation}
  \label{eq:quadric-quartic}
  \begin{aligned}
 \bT&= s^*(A_0(\beta^1)^2+2A_1\beta^1\beta^2+A_2(\beta^2)^2)\otimes V\otimes X^{-2},\\
\bC&=s^*(W_0(\beta^1)^4+4W_1(\beta^1)^3(\beta^2)+6W_2(\beta^1)^2(\beta^2)^2+4W_3(\beta^1)(\beta^2)^3+W_4(\beta^2)^4)\otimes V\otimes X^{-1} \\
\end{aligned}
\end{equation}
where $s\colon Q\to\cG$ is a section, $\bT\in \Gamma(\mathrm{Sym}^2(\scV^*)\otimes\bigwedge^2\scV\otimes\scX^{-2}),$ $\bC\in \Gamma(\mathrm{Sym}^4(\scV^*)\otimes\bigwedge^2\scV\otimes\scX^{-1}),$ $X:=\frac{\partial}{\partial s^*\alpha^0}\in\Gamma(\scX),$  $V:=\frac{\partial}{\partial s^*\beta^1}\w\frac{\partial}{\partial s^*\beta^2}\in\Gamma(\bigwedge^2\scV),$ 
\[
\begin{gathered}
A_0=T^2_1,\quad A_1=T^2_2,\quad A_2=-T^1_2,\\
  W_0=C^2_{111},\quad W_1=C^2_{211},\quad W_2=C^2_{221},\quad W_3=C^2_{222},\quad W_4=-C^1_{222},
\end{gathered}
\]
and $(\tfrac{\partial}{\partial s^*\alpha^0},\tfrac{\partial}{\partial s^*\alpha^1},\tfrac{\partial}{\partial s^*\alpha^2},\tfrac{\partial}{\partial s^*\beta^1},\tfrac{\partial}{\partial s^*\beta^2})$ denote the vector fields dual to the coframe $s^*(\alpha^0,\alpha^1,\alpha^2,\beta^1,\beta^2)$ in the Cartan connection \eqref{eq:path-geom-cartan-conn-3D}.
Moreover, $Q$ is equipped with a degenerate conformal structure given by  $[s^*h]\subset\mathrm{Sym^2}(T^*Q)$ where
\begin{equation}
  \label{eq:degenerate-conformal-str}
  h:=\alpha^1\beta^2-\alpha^2\beta^1\in\Gamma(\mathrm{Sym}^2(T^*\cG)).
  \end{equation}
\end{proposition}
In 
 \ref{sec:path-geom-aris} we will need the Bianchi identities among the entries of the $\bC$ and $\bT$ given by
\begin{equation}
  \label{eq:W-A-curvature-torsion-Bianchies}
  \begin{aligned}
    \exd W_0 \equiv& 4 W_0 \psi^1_1 + 4 W_1 \psi^2_1,\quad     &&\exd W_1 \equiv W_0 \psi^1_2 + (3  \psi_1^1 +  \psi^2_2)W_1 + 3 W_2 \psi^2_1\\ 
    \exd W_2 \equiv& 2 W_1 \psi^1_2 + (2 \psi^1_1 + 2  \psi^2_2)W_2 + 2 W_3 \psi^2_1,\quad    && \exd W_3 \equiv 3 W_2 \psi^1_2 + (\psi^1_1 + 3 \psi^2_2)W_3 + W_4 \psi^2_1\\
    \exd W_4 \equiv& 4 W_3 \psi^1_2 + 4 W_4 \psi^2_2,\quad     &&\exd A_0 \equiv (4 \psi^0_0 + 3 \psi^1_1 + \psi^2_2)A_0 + 2 A_1 \psi^2_1\\
    \exd A_1 \equiv& A_0 \psi^1_2 + (4 \psi^0_0 + 2 \psi^1_1 + 2 \psi^2_2)A_1 + A_2 \psi^2_1,\quad 
    &&\exd A_2 \equiv 2 A_1 \psi^1_2 + (4 \psi^0_0 + \psi^1_1 + 3  \psi^2_2)A_2.
  \end{aligned}
\end{equation}
modulo $\{\alpha^0,\alpha^1,\alpha^2,\beta^1,\beta^2\}.$

\begin{remark}\label{rmk:wuenschman}
The torsion $\bT$ is also referred to as the  (generalized) Wilczy\'nski or the W\"unschman invariant of the ODE system. It can be interpreted as a vector bundle endomorphism $\bT=(T^i_j)_{i,j=1}^n\colon\scV\to\scV$, as will be explained in Remark \ref{remark:torsion}.
\end{remark}

Using  Theorems \ref{thm:path-geom-cartan-conn} and \ref{thm:2D-path-geome}, it follows that two path geometries $(Q_i,\scX_i,\scV_i)$, $i=1,2$,  are locally equivalent if and only if for their respective Cartan geometries $(\cG_i\to Q_i,\psi_i)$ there is a diffeomorphism $f\colon\cG_1\to\cG_2$ such that $f^*\psi_2=\psi_1.$ If the fundamental invariants vanish, i.e. $\bC=\bT=0,$ then the Cartan curvature is zero and the Cartan connection coincides with the Maurer-Cartan forms on $\mathrm{PGL}(n+2,\RR)$. Consequently, the path geometry is locally equivalent to the  canonical one on $\RR\PP^{n+1}$ whose paths are projective lines, as discussed in \eqref{eq:flat-path-geom}. Similarly, by Theorems \ref{thm:paraCR-cartan-conn} and \ref{thm:2D-path-geome}, one has a Cartan geometric description of local equivalence between two para-CR structures. The vanishing of the fundamental invariant implies that the para-CR structure is locally equivalent to the flat model \eqref{eq:flat-para-CR}.

Using the notion of duality in Definition \ref{def:duality-para-CR} and double fibration \eqref{eq:double-fibration} we make the following definition that will be used in \ref{sec:path-geom-aris}.
\begin{definition}\label{def:co-proj-surf}
  A path geometry on  $M_1$ is called a \emph{projective structure} if its paths coincide with the geodesics of a linear connection $\nabla$ on $M_1$ as unparametrized curves. A 2-dimensional path geometry on $M_1$ is called a \emph{co-projective structure} if its dual path geometry on $M_2$  is a projective structure. 
  \end{definition}
 
  In terms of Cartan curvature \eqref{eq:2D-path-geom-cartan-conn}, projective structures on  $M_1=N\slash\scD_2$ are characterized by  $Q_1=0.$  The vanishing condition $P_1=0$  implies that the path geometry on $M_1$ is co-projective.

\section{Para-CR Lewy curves and their characterization}\label{sec:dancing-construction}    
In this section we define  Lewy curves in a para-CR structure $(N,\scD_1,\scD_2)$ in terms of an associated defining function, extending the 3-dimensional case discussed briefly in \ref{sec:first-look}. In \ref{sec:intrinsic}, we define them via the integral manifolds of $\scD_1$ and $\scD_2$. Moreover, we show that Lewy curves of a para-CR 3-manifold define a path geometry on $N$. In \ref{sec:examples}  the corresponding  ODE systems of such path geometries are described along with some examples. In \ref{sec:characterization-dancing}  we provide necessary and sufficient conditions that characterize path geometries defined by para-CR Lewy curves among 3-dimensional path geometries. These results serve as the starting point for \ref{sec:path-geom-aris}, where Lewy curves are described Cartan geometrically. Finally in \ref{sec:higher-dimensions} we characterize  path geometries defined  by Lewy curves of higher dimensional para-CR structures.
\subsection{Lewy curves via a defining function}\label{sec:LewyDef}   
Let $\Phi\colon M_1\times M_2\to \R$ be a defining function of a $(2n+1)$-dimensional para-CR structure $(N,\scD_1,\scD_2)$ realized as a hypersurface $N\subset M_1\times M_2$ where
\[
N=\{(x,y)\in M_1\times M_2\ |\ \Phi(x,y)=0)\}.
\]
Note that this may require working in  sufficiently small open sets of $N$ as explained in Remark \ref{rmk:N-open-set}. We recall from \ref{sec:paraCR}  that any point $y\in M_2$ corresponds to a  hypersurface $H_y\subset M_1$ and any point $x\in M_1$ corresponds to a  hypersurface $H_x\subset M_2$.  We can make the following definition which is automatically satisfied when $k=2.$
\begin{definition}\label{def:general-position}
  Given a para-CR structure with double fibration \eqref{eq:double-fibration}, a set of  $k$  points $x_1,\ldots,x_k\in M_1$, respectively $y_1,\ldots,y_k\in M_2$, are  in general position if the intersection of the corresponding hypersurfaces, i.e. $\bigcap_{i=1}^k H_{x_i}\subset M_2$, respectively $\bigcap_{i=1}^kH_{y_i}\subset M_1$,  is an $(n+1-k)$-dimensional submanifold. 
\end{definition}
\begin{remark}
Note that the assumption that points $\{x_i\}_{i=1}^k$, or $\{y_i\}_{i=1}^k$, are in the general position is an open condition. The hypersurfaces $\{H_{ x_i}\}_{i=1}^k$ or  $\{H_{y_i}\}_{i=1}^k$ intersect along a $(n+1-k)$-dimensional submanifold if and only if their tangent spaces are in general position at intersection points.  Note that Definition \ref{def:general-position} can be also expressed in terms of transversality of submanifolds, which will be exploited throughout the paper. In general we say that two submanifolds $N_1,N_2$ of a manifold $M$ \emph{intersect transversely} if for any $p\in N_1\cap N_2$,  $T_pN_1 +T_pN_2$ has maximal rank as a linear subspace of $T_pM$. It is a standard result  that the intersection of two transverse submanifolds with the property that $\dim N_1+\dim N_2\geq \dim M$ is a submanifold of dimension $\dim N_1+\dim N_2-\dim M$. 
In   Definition \ref{def:general-position}, for $k=2$, two points $x_1$ and $x_2$ (resp. $y_1$ and $y_2$) are in  general position if and only if $H_{x_1}$ and $H_{x_2}$ (resp. $H_{y_1}$ and $H_{y_2}$) intersect transversely.  For $k>2$, one can proceed iteratively by defining points $\{x_i\}_{i=1}^k$ to be in  general position if $\{x_i\}_{i=1}^{k-1}$ are in  general position, and $\bigcap_{i=1}^{k-1} H_{x_i}$ and $H_{x_k}$ intersect transversely. It turns out that the definition only depends on the set $\{x_i\}_{i=1}^k.$ Lastly, when defining intersection of transverse submanifolds, we only take connected components and restrict ourselves to sufficiently small open sets, as mentioned in Remark \ref{rmk:N-open-set}, where the intersection is smooth.
\end{remark} 

The \emph{Segre family} are submanifolds in $M_1\times M_2$ defined as
\[
Q_{\hat x\hat y}=\{(x,y)\in M_1\times M_2\ |\ \Phi(x,\hat y)=0,\ \Phi(\hat x,y)=0)\}.
\]
This family is parameterized by \emph{non-incident} points $(\hat x,\hat y)\in M_1\times M_2\setminus N,$ as described in Definition \ref{def:incidence}. The Segre family gives rise to a family of curves on $N$  defined as follows.
\begin{definition}\label{def:Lewy}
Let $(N,\scD_1,\scD_2)$ be a $(2n+1)$-dimensional para-CR structure with double fibration \eqref{eq:double-fibration}. For any $(\hat x_i,\hat y_i)\in M_1\times M_2\setminus N$, $i=1,\ldots,n$, such that $(\hat x_i)_{i=1}^n$ and $(\hat y_i)_{i=1}^n$ are in general positions in $M_1$ and $M_2$, the curve
\[
\gamma_{\hat x_1\ldots \hat x_n\hat y_1\ldots\hat y_n}=N\cap Q_{\hat x_1\hat y_1}\cap\ldots\cap Q_{\hat x_n\hat y_n}
\]
is called a \emph{Lewy curve}.
\end{definition}

The fact that $\gamma_{\hat x_1\ldots \hat x_n\hat y_1\ldots\hat y_n}$ is a curve follows immediately from the assumption that $(\hat x_i)_{i=1}^n$ and $(\hat y_i)_{i=1}^n$ are in general positions. Indeed, this assumption implies that the intersections
\begin{equation}\label{eq:L-curves-M1-M2}
  \begin{aligned}
    L_{\hat x_1\ldots\hat x_n}=&H_{\hat x_1}\cap\ldots\cap H_{\hat x_n}=\{y\in M_2\ |\ \Phi(\hat x_1,y)=0,\ldots, \Phi(\hat x_n,y)=0\}\\
L_{\hat y_1\ldots\hat y_n}=&H_{\hat y_1}\cap\ldots\cap H_{\hat y_n}=\{x\in M_1\ |\ \Phi(x,\hat y_1)=0,\ldots, \Phi(x,\hat y_n)=0\}
  \end{aligned}
\end{equation}
are curves. Consequently, $\Gamma_{\hat x_1\ldots\hat x_n}=\pi_2^{-1}(L_{\hat x_1\ldots\hat x_n})$ and $\Gamma_{\hat y_1\ldots\hat y_n}=\pi_1^{-1}(L_{\hat y_1\ldots\hat y_n})$ are $(n+1)$-dimensional submanifolds of $N$. Their intersection is a curve in $N$ provided that they are transverse. This is implied by the condition that $(\hat x_i)_{i=1}^n$ and $(\hat y_i)_{i=1}^n$ can be grouped into non-incident pairs as in Definition \ref{def:Lewy}. In fact, for any pair $(\hat x, \hat y)\in M_1\times M_2$ the hypersurfaces $\pi_1^{-1}(H_{\hat y}),\pi_2^{-1}(H_{\hat x})\subset N$ do not intersect transversely if and only if $(\hat x,\hat y)$ is an incident pair which immediately follows from Proposition \ref{prop:jet-realization-para-cr-structures} and the interpretation of $\Phi$ as a general solution function. If the pair $(\hat x,\hat y)$ is incident then $(\hat x,\hat y)\in N$ and at this point the tangent planes of $\pi_1^{-1}(H_{\hat y})$ and $\pi_2^{-1}(H_{\hat x})$ coincide with~$\scC$.

A priori, we have a family of curves depending on $2n(n+1)$ parameters given by set of $n$ points $x_i$'s in $M_1$ and $y_i$'s in  $M_2.$ In dimension $3$, which corresponds to $n = 1$, the space of paths has dimension $2(n+1)n = 4$, which agrees with that of a 3-dimensional path geometry. In higher  dimensions, in general there are too many paths to define a path geometry. Regarding flat para-CR structures, we have the following.
\begin{proposition}\label{prop:flatcase}
Given a flat para-CR structure $(N,\scD_1,\scD_2),$ its Lewy curves   define a path geometry on $N.$
\end{proposition}
\begin{proof}
Using the definition of Lewy curves, one can explicitly find the corresponding ODE systems for such path geometries as done in Example \ref{exa:corr-syst-odes-flat-paraCR}. Alternatively, one can proceed as follows.  If $N$ is $(2n+1)$-dimensional then  $M_1$ and $M_2$ are identified as $\RR\PP^{n+1}$ and $(\RR\PP^{n+1})^*$, respectively. Furthermore, all hypersurfaces $H_x$ and $H_y$ are projective hyperplanes described by equation \eqref{eq:flat-para-CR-PDE}, for a point $x\in M_1$, or $y\in M_2$. Consequently, all curves $L_{\hat x_1\ldots\hat x_n}$ and $L_{\hat y_1\ldots\hat y_n}$ are straight lines. Note that two lines $L_{\hat x_1\ldots\hat x_n}$ and $L_{\tilde x_1\ldots\tilde x_n}$ coincide if and only if $(\hat x_i)_{i=1}^n$ and $(\tilde x_i)_{i=1}^n$ lie on the same   2-codimensional projective plane in $M_1$. Similarly, two lines $L_{\hat y_1\ldots\hat y_n}$ and $L_{\tilde y_1\ldots\tilde y_n}$ coincide if and only if $(\hat y_i)_{i=1}^n$ and $(\tilde y_i)_{i=1}^n$ lie on the same   2-codimensional projective plane in $M_2$.  As a result, Lewy curves are effectively parametrized by a pair of  2-codimensional projective planes in $M_1$ and $M_2,$ which implies that they form a  $4n$-parameter family of curves. Furthermore, the two lines $L_{\hat x_1\ldots\hat x_n}$ and $L_{\hat y_1\ldots\hat y_n}$ can be picked arbitrarily in $M_1$ and $M_2$, which implies that the directions of the corresponding Lewy curves in $N$ form an open subset in the projective tangent bundle of $N$. 
\end{proof}

 On $M_1$ the curves $\gamma_{\hat x_1\ldots \hat x_n\hat y_1\ldots\hat y_n}$  are represented by incident pairs $(q,K)$, i.e. $q\in K\subset M_1$, constructed as follows: $K$ is an arbitrary hyperplane in $M_1$ containing a  2-codimensional projective plane containing the points $(\hat x_1,\ldots,\hat x_n)$, and $q\in L_{\hat y_1\ldots\hat y_n}$ is the intersection point of $L_{\hat y_1\ldots\hat y_n}$ and $K$. Figures \ref{pic:2D} and \ref{pic:3D} illustrate the construction of Lewy curves in flat  para-CR structures of dimensions 3 and 5 at as viewed in   $M_1.$ Explicit formula for the corresponding system of ODEs is given in Example \ref{exa:corr-syst-odes-flat-paraCR}.

\begin{center}
\begin{figure}[htbp]
\includegraphics[width=.3\linewidth]{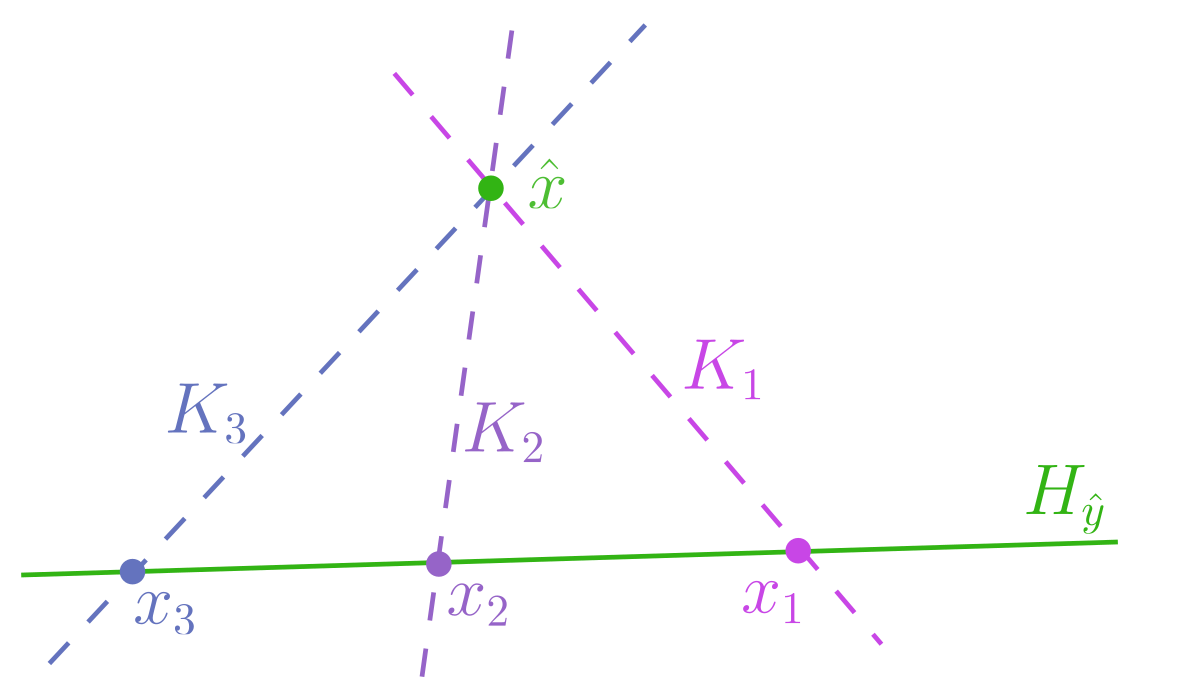}  
\caption{The case $n=1$ corresponds to flat para-CR 3-manifolds $N$ where $\dim M_1=\dim M_2=2$. Every non-incident pair $(\hat x,H_{\hat y})$ where $\hat x\notin H_{\hat y},$ determines a curve on $N$ given by incident pairs such as $(x_i,K_i)$ where $x_i\in H_{\hat y}\subset M_1, \hat x\in K_i,$ and $x_i\in K_i.$ }
\label{pic:2D}
\end{figure}
\end{center}

\begin{center}
\begin{figure}[htbp]
\includegraphics[width=.4\linewidth]{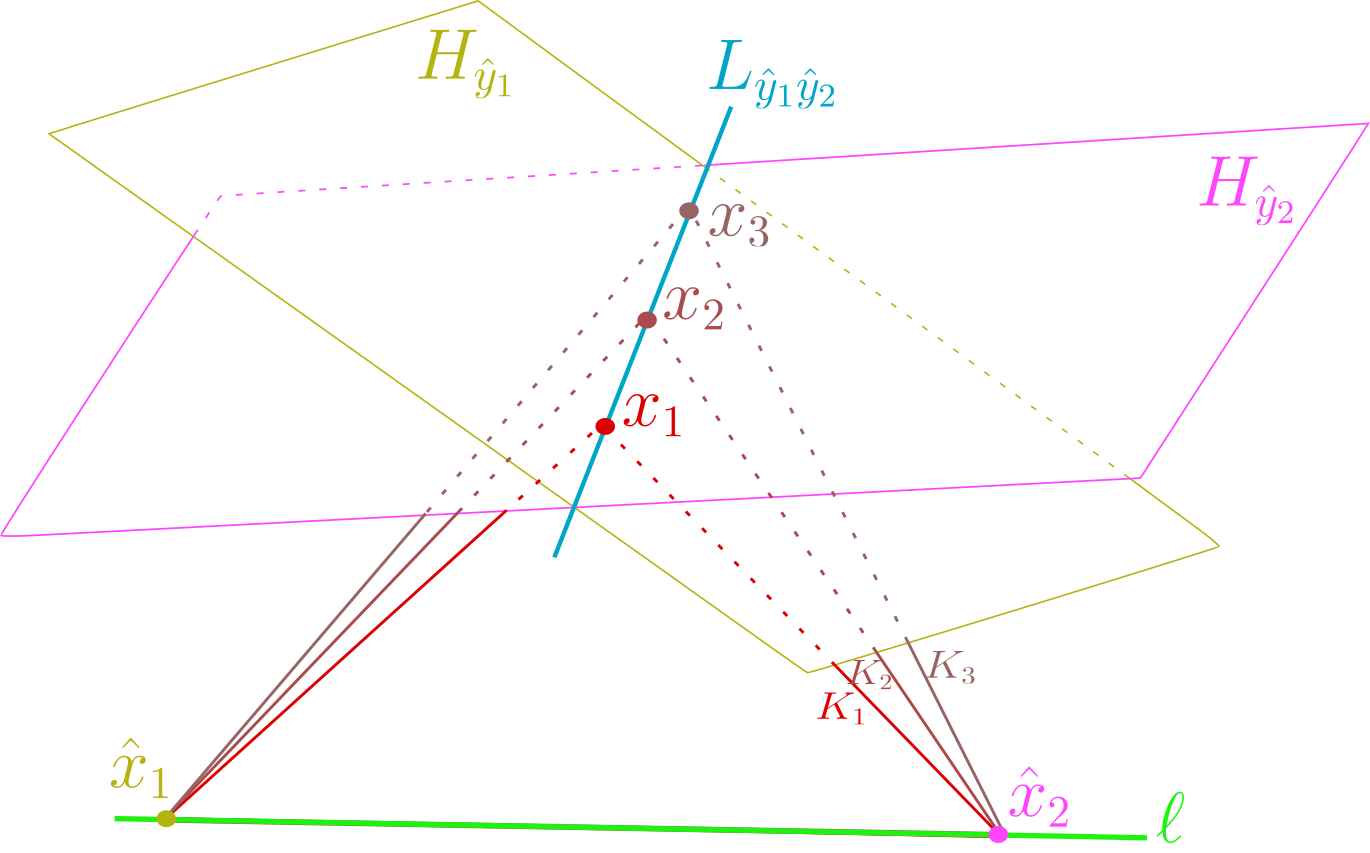}  
\caption{The case $n=2$ corresponds to flat  para-CR 5-manifolds $N$ where $\dim M_1=\dim M_2=3$. Every two non-incident pairs $(\hat x_1,H_{\hat y_1}),(\hat x_2, H_{\hat y_2}),$ determine a curve on $N$ given by  incident pairs such as $(x_i,K_i)$ where $x_i\in H_{\hat y_a}\subset M_1, x_i\in K_i,$ and $\hat x_a\in K_i$ for $a=1,2$ and all $i$'s, where $K_i$ is the plane spanned by $x_i$ and the unique projective line $\ell$ between $\hat x_1$ and $\hat x_2.$ The resulting curve in $N$ is determined by  lines $\ell$ and $L_{\hat y_1\hat y_2}$ in $M_1$, which results in  an 8-parameter family of curves on $N.$}
\label{pic:3D}
\end{figure}
\end{center}
 
\subsection{An intrinsic description}\label{sec:intrinsic}   
\label{sec:an-overview-dancing}
Given   a para-CR structure $(N,\scD_1,\scD_2)$, we now define Lewy curves on $N$ via the  integral manifolds of $(\scD_1,\scD_2)$, hence avoiding the use of a defining function. 
\vskip 1ex

Firstly, we define  hypersurfaces $\Sigma^1_x,\Sigma_y^2\subset N$ as
\begin{equation}\label{eq:def-Sigma1-Sigma2}
  \Sigma^1_x=\bigsqcup_{y\in \pi_2(\pi_1^{-1}(x))} \pi_2^{-1}(y),\qquad   \Sigma^2_y=\bigsqcup_{x\in \pi_1(\pi_2^{-1}(y))} \pi_1^{-1}(x).
\end{equation}
As a result, every $\Sigma_x^1,x\in M_1$ (resp. $\Sigma_y^2,y\in M_2$) contains an integral manifold of  $\scD_2$ (resp. $\scD_1,$) and is foliated by  integral manifolds of $\scD_1$ (resp.  $\scD_2$)
\begin{remark}\label{rmk:Lewy}
In the framework of  \ref{sec:LewyDef}, the hypersurfaces $\Sigma^1_x$ and $\Sigma^2_y$ are expressed as submanifolds of $M_1\times M_2$ by
\[        
\Sigma^1_{x}=\{(\hat x,\hat y)\in N\ |\ \Phi(x,\hat y)=0\},
\quad \Sigma^2_{y}=\{(\hat x,\hat y)\in N\ |\ \Phi(\hat x, y)=0\}.
\]
where $\Phi$ is a  defining function. Figure 3 shows these hypersurfaces in a para-CR 3-manifold, where we interpret $N$ as an open subset of the projectivized tangent bundle of the surface $M_1$. 
Note that by definition all hypersurfaces $\Sigma^1_x$ and $\Sigma^2_y$ satisfy
\begin{equation}\label{eq:T-Sigma-C}
  T_{q}\Sigma^1_x=\scC_q,\qquad  T_{s}\Sigma^2_y=\scC_s,
  \end{equation}
for all $q\in \pi^{-1}_1(x)\subset \Sigma^1_x $ and all $s\in\pi^{-1}_2(y)\subset \Sigma^2_y,$ where $\scC_q$ is the contact distribution at $q\in N.$
\end{remark}

\begin{center}
\begin{figure}[htbp]
\includegraphics[width=.4\linewidth]{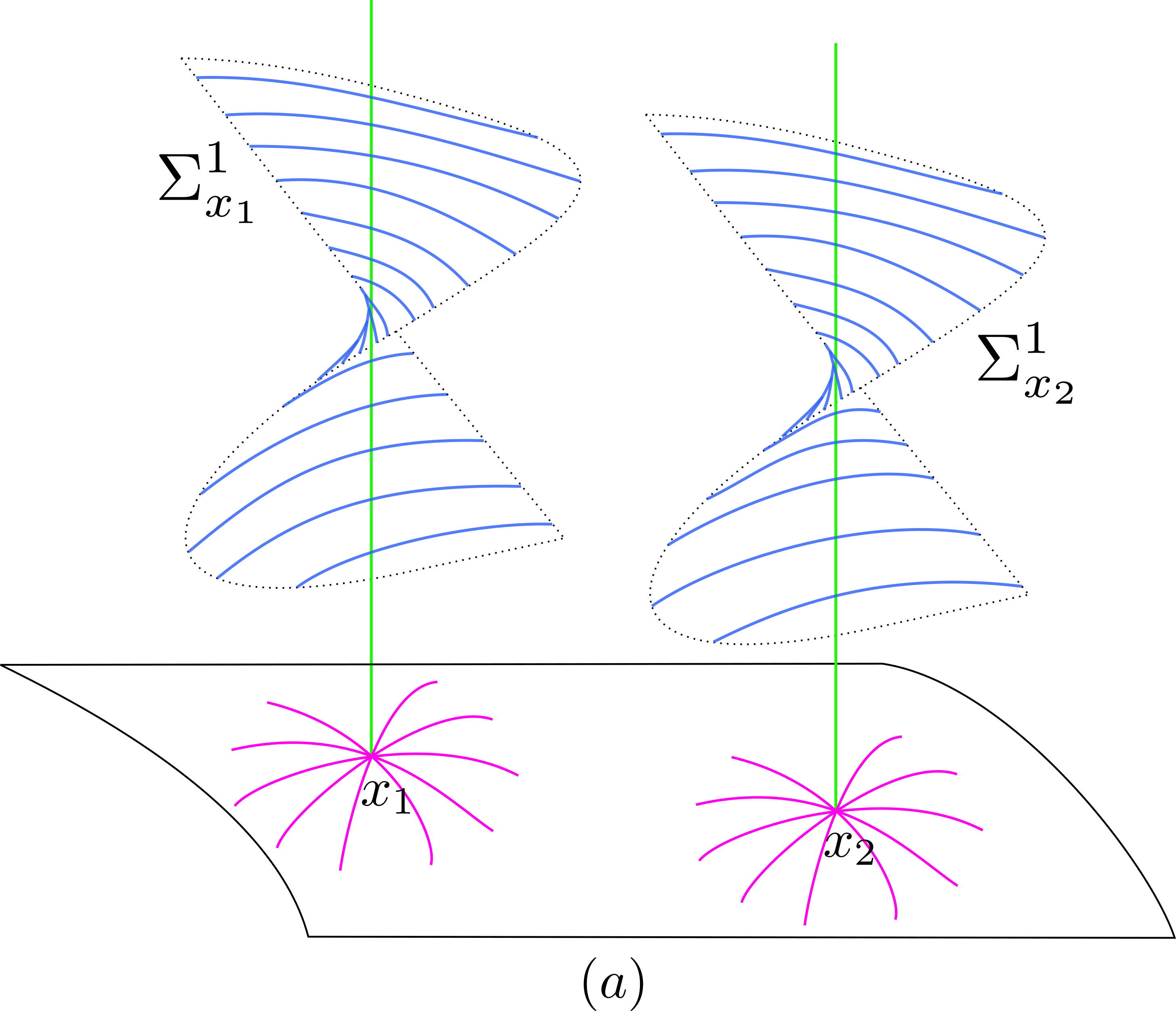} 
\includegraphics[width=.4\linewidth]{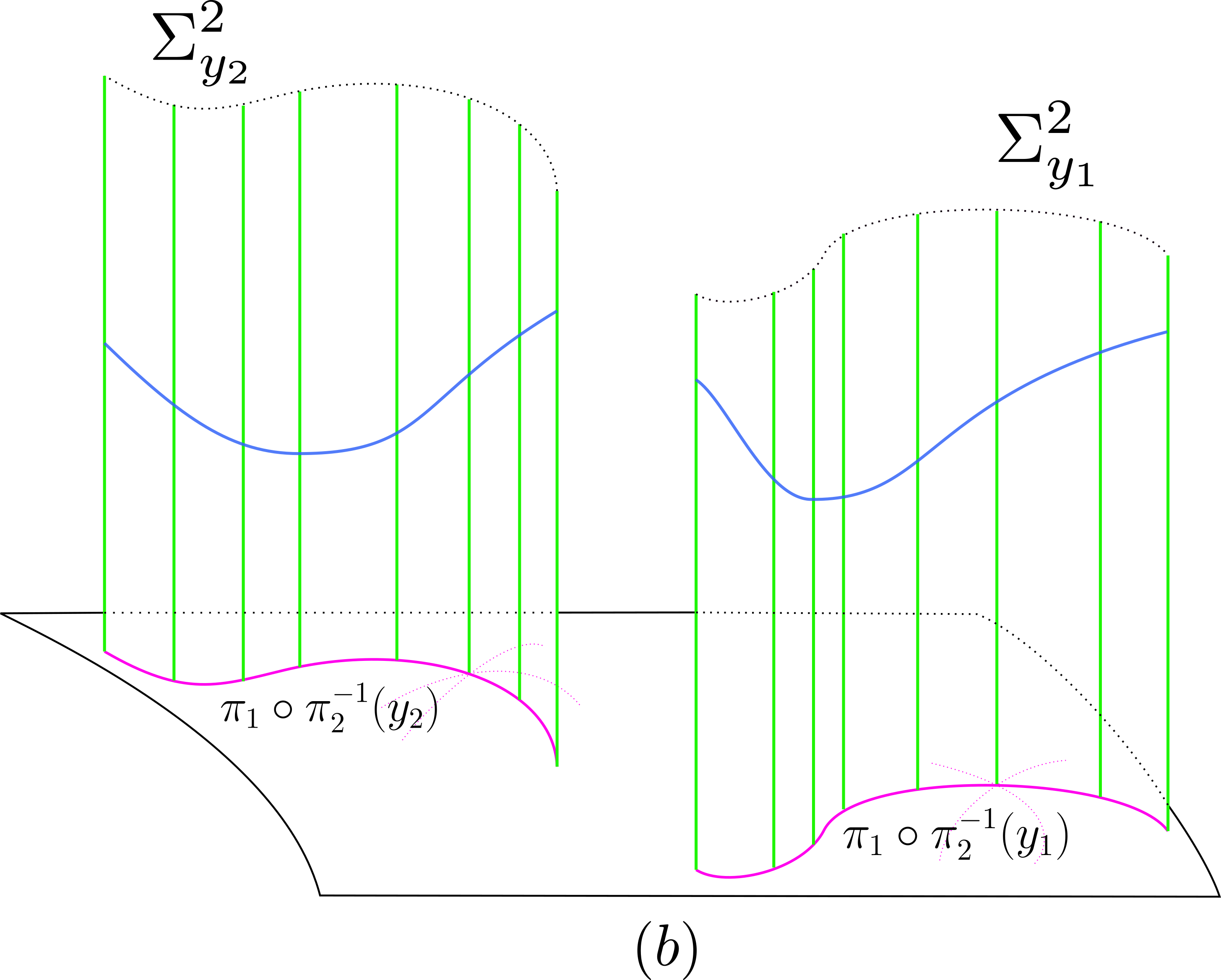}   
\caption{ (a) At $x_1,x_2\in M_1,$ the vertical green lines are open subsets of the fibers $\pi_1^{-1}(x_i)\subset N.$ The blue curves intersecting the fibers $\pi_1^{-1}(x_i)$ are natural lifts of the paths (purple curves) on $M_1$ passing through $x_i.$ The surfaces $\Sigma^1_{x_i}\subset N$  are foliated by the blue curves and  its generating curve is the fiber $\pi_1^{-1}(x_i)$. (b) For $y_1,y_2\in M_2,$ the  purple curves on $M_1$ are the paths $\pi_1\circ\pi^{-1}_{2}(y_i).$ Their lift to $N,$ represented by  blue curves, are  generating curves of  ruled surfaces $\Sigma^2_{y_i}\subset N$ whose ruling lines are vertical green lines which are open subsets of the fibers $\pi_1^{-1}(t)\subset N$ for all $t\in\pi_1\circ\pi^{-1}_{2}(y_i).$}
\end{figure}
\end{center}

Some basic properties of the submanifolds $\Sigma^1_x$ and $\Sigma^2_y$ in $N$ are as follows.

\begin{proposition}\label{prop0}
Let $(N,\scD_1,\scD_2)$ be a para-CR structure of dimension $2n+1$ with double fibration \eqref{eq:double-fibration}. Fix $\hat p\in N$ and denote $\hat x=\pi_1(\hat p)$, $\hat y=\pi_2(\hat p)$. Restricting to a sufficiently small neighborhood $U$ of $\hat p$, one has the following:
\begin{itemize}
\item[(a)] The family of hypersurfaces $\Sigma^1_x$ (resp. $\Sigma^2_y$) satisfying  $\hat p\in\Sigma^1_x$ (resp. $\hat p\in \Sigma^2_y$) is parameterized by points of the fiber $\pi_2^{-1}(\hat y)$ (resp.  $\pi_1^{-1}(\hat x)$) i.e. each such hypersurface corresponds to an  intersection point  $\pi_1^{-1}(x)\cap \pi_2^{-1}(\hat y)$ (resp. $\pi_2^{-1}(y)\cap \pi_1^{-1}(\hat x)$).
\item[(b)] For every point $\hat p\in \Sigma^1_x$ (resp. $\hat p\in\Sigma^2_y$) one has $\scD_1(\hat p)\subset T_{\hat p}\Sigma^1_x$ (resp. $\scD_2(\hat p)\subset T_{\hat p}\Sigma^2_y$).
\item[(c)] Denote by $\mathrm{Gr}^{\scD_i}_{2n}(T_{\hat p}N)$ the set of all hyperplanes in $T_{\hat p}N$ containing  $\scD_i(\hat p)\subset\scC_{\hat p}$, $i=1,2$. There is a neighborhood  $U^{\scD_1}_{\hat p}\subset \mathrm{Gr}_{2n}^{\scD_1}(T_{\hat p}N)$ (resp. $U^{\scD_2}_{\hat p}\subset \mathrm{Gr}_{2n}^{\scD_2}(T_{\hat p}N)$) of $[\scC_{\hat p}]\in \mathrm{Gr}_{2n}^{\scD_i}(T_{\hat p}N)$  such that there is a one-to-one correspondence between points of $U^{\scD_1}_{\hat p}$ (resp. $U^{\scD_2}_{\hat p}$ ) and hypersurfaces $\Sigma^1_x$ (resp. $\Sigma^2_y$) that pass through $\hat p$  where  $x=\pi_1(p)$ and $y=\pi_2(p)$ for all $p\in U$.
\end{itemize}
\end{proposition}
\begin{proof}
Without loss of generality we shall restrict the proof to the family $\Sigma^1_x$. According to Proposition~\ref{prop:jet-realization-para-cr-structures}, $N$ is locally isomorphic to $J^1(\R^n,\R)$, where
\begin{equation}\label{eqD1D2}
\scD_1=\langle\partial_{p_1},\ldots,\partial_{p_n}\rangle,\ \ \  \mathrm{and}\ \ \ \scD_2=\langle\partial_{t^i}+p_i\partial_z+f_{ij}\partial_{p_j}\rangle_{i=1}^n
\end{equation}
Let $U$ be a neighborhood of $\hat p$ where such a local coordinate system is defined with $\hat p$ being the origin.

Part (a) follows directly from the construction. Namely, if $\hat p\in \Sigma^1_x$ then the integral leaf of $\scD_2$ corresponding to $x$, i.e. $\pi^{-1}_1(x)\subset N$, necessarily intersects the integral leaf of $\scD_1$ containing $p$, i.e. $\pi^{-1}_2(\hat y)$. Notice that the intersection point of the two leaves is uniquely determined in $U$.

Part (b) follows from the construction as well, because all $\Sigma^1_x$ are foliated by integral manifolds of $\scD_1$. It follows that $\scD_1(\hat p)\subset T_{\hat p}\Sigma^1_x$.

To show part (c) we only discuss open set $U^{\scD_1}_{\hat p}$. The one-to-one correspondence is given by tangent hyperplanes $T_{\hat p}\Sigma^1_{x}\subset T_{\hat p}N$. Note that it follows from (b) that $\scD_1(\hat p)\subset T_{\hat p}\Sigma^1_{x}$, and hence $T_{\hat p}\Sigma^1_{x}\in \mathrm{Gr}_{2n}^{\scD_1}(T_{\hat p}N)$. Let $q=\pi^{-1}_1(x)\cap\pi^{-1}_2(\hat y)$ and denote the corank 1 preimage of $(\pi_{2})_*(\scC_q)\subset T_{\hat y}M_2$ along the fiber $\pi^{-1}_2(\hat y)$ as  $(\pi_{2})_*^{-1}((\pi_{2})_*(\scC_q))$. By \eqref{eq:def-Sigma1-Sigma2} and property \eqref{eq:T-Sigma-C}, one has 
\begin{equation}\label{eq:T-p-Sigma1x}
  (\pi_{2})_*^{-1}((\pi_{2})_*(\scC_q))(\hat p)=T_{\hat p}\Sigma^1_x.
  \end{equation}
Furthermore, using the definition $M_2=N\slash\scD_1,$ one has the isomorphism
\begin{equation}\label{eq:isom-gr-PTM2}
  \mathrm{Gr}_{2n}^{\scD_1}(T_{\hat p}N)\cong \PP T_{\hat y}^*M_2.
\end{equation}
By definition, at $x=\hat x$ we have $T_{\hat p}\Sigma^1_{\hat x}=\scC_{\hat p}$. By our discussion in  Remark \ref{rmk:para-cr-structures-2D-path} the fiber $\pi^{-1}_2(\hat y)$ can be identified with $\PP T^*_{\hat y}M_2$. As a result, the preimage $(\pi_{2})_*^{-1}((\pi_{2})_*(\scC_q))(\hat p)$ in \eqref{eq:T-p-Sigma1x} for  all $q\in \pi^{-1}_2(\hat y)$ in a sufficiently small neighborhood of $\hat p$ give an open set of $\PP T_{\hat y}^*M_2.$ Using isomorphism \eqref{eq:isom-gr-PTM2}, such hyperplanes define an open set $U^{\scD_1}_{\hat p}$ around $[\scC_{\hat p}]\in\mathrm{Gr}_{2n}^{\scD_1}(T_{\hat p}N).$

\end{proof}

\begin{corollary}\label{cor:tangentplane}
At every point $\hat p\in N$ of a para-CR structure $(N,\scD_1,\scD_2),$ there is a an open set of transverse directions to the contact distribution, $Q_{\hat p}\subset\PP(T_{\hat p} N\backslash\scC_{\hat p}),$ such that for any $v\in Q_{\hat p}$ there is a Lewy curve $\gamma_{x_1\ldots x_ny_1\ldots y_n}$ tangent along $v$. If $\dim N=3$ then the curve is unique and consequently Lewy curves define a  path geometry on $N$.
\end{corollary}
\begin{proof}
Recall that  $\mathrm{Gr}_{2n}^{\scD_i}(T_{\hat p}N)$, $i=1,2$, denotes the set of all hyperplanes in $T_{\hat p}N$ containing $\scD_i(\hat p)$. It is an elementary observation that any transverse direction $\langle v\rangle\subset TN\backslash\scC$ to the contact distribution can be written as
\begin{equation}\label{eq_direction}
\langle v\rangle=\bigcap_{i=1}^nK^{\scD_1}_{i}\cap  \bigcap_{i=1}^nK^{\scD_2}_{i},
\end{equation}
for some hyperplanes $[K^{\scD_1}_i]\in \mathrm{Gr}_{2n}^{\scD_1}(T_{\hat p}N)$ and $[K^{\scD_2}_i]\in \mathrm{Gr}^{\scD_2}_{2n}(T_{\hat p}N)$. Restricting to  open sets $U^{\scD_i}_{\hat p}$ in part (c) of Proposition \ref{prop0}, it follows that there are points $x_i\in M_1$ and $y_i\in M_2$ such that $K^{\scD_1}_i=T_{\hat p}\Sigma^1_{x_i}$ and $K^{\scD_2}_i=T_{\hat p}\Sigma^2_{y_i}$. One obtains that Lewy curve
\begin{equation}\label{eq:lewy-curve-as-intersection}
   \gamma_{x_1\ldots x_ny_1\ldots y_n}:=\Sigma^1_{x_1}\cap\ldots\cap \Sigma^1_{x_n}\cap \Sigma^2_{y_1}\cap\ldots\cap\Sigma^2_{y_n}
\end{equation}
is tangent along $\langle v\rangle$ at $\hat p$. Since, $U^{\scD_i}_{\hat p}$, $i=1,2$, are open, one obtains an open set of transverse directions to the contact distribution that are tangent to Lewy curves. Note that this open set of transverse directions are  given as intersections \eqref{eq_direction} for hyperplanes $[K_i^{\scD_1}],[K_2^{\scD_2}]$ in a neighborhood of $[\scC_{\hat p}]\in \mathrm{Gr}_{2n}(T_{\hat p}N).$

The 3-dimensional case corresponds to $n=1$, for which the choice of hyperplanes in \eqref{eq_direction} is uniquely determined by $[v]\in\PP (T_{\hat p}N\backslash\scC_{\hat p})$. Consequently, there is an open set of transverse direction to the contact distribution $,Q\subset \PP (T N\backslash\scC),$ that is foliated by Lewy curves, which ultimately defines  a path geometry on $N$.
\end{proof}

\begin{remark}\label{rmk:dancing}
  As was explained in \eqref{eq:gammaxy-first-look},  the space $M_1\times M_2\setminus N$ parameterizes Lewy curves in dimension 3. This space appears in the so-called dancing construction of \cite{Bor} and \cite{D} where they show that when the para-CR 3-manifold is flat then $M_1\times M_2\backslash N$  is equipped with an anti-self-dual conformal structure of the split signature. Later we will show that in the flat case, Lewy curves coincide with chains, and therefore, the induced anti-self-dual conformal structure is the standard metric on the homogeneous space $\mathrm{SL}(3,\RR)\slash\mathrm{GL}(2,\RR)$ (e.g. see \cite[Section 3.2.2]{KM-chains}). Moreover, in this case the set of self-dual null surfaces, also known as   $\alpha$-surfaces, of this metric are parameterized by incident pairs $(x,y)\in N$, and are defined as
  \begin{equation}\label{dancing}
    S_{xy}=\left\{(\hat x,\hat y)\in M_1\times M_2\setminus N\ \vline\ \Phi(\hat x, y)= \Phi(x,\hat y)=0\right\},
  \end{equation}
  which, in a sense, is ``dual'' to  $\gamma_{\hat x\hat y},$ as defined in \eqref{eq:gammaxy-first-look}. Any null curve that lie in an $\alpha$-plane is referred to as a \emph{dancing curve} in \cite{Bor}. In \cite{D} the dancing construction is studied via a twistorial construction, where $N$ becomes the twistor space  of the conformal structure, i.e. its space of $\alpha$-surfaces.

  Notice that definition \eqref{dancing} makes sense for any para-CR structure and does not require flatness. In  \cite{D} it is shown that such surfaces arise from a half-flat conformal structure if and only if the para-CR structure is flat, otherwise they correspond to a half-flat \emph{causal structure} \cite{Omid-Thesis}. Proposition \ref{prop:torsionfree-freestyling} gives path geometric proof of this result. 
\end{remark}

\subsection{Corresponding systems of ODEs}\label{sec:examples} 
In this section we focus on 2nd order ODE systems defined by the path geometries of  para-CR Lewy curves. We begin with the general case of para-CR 3-manifolds and then  consider  flat para-CR structures in all dimensions, which will be used in \ref{sec:higher-dimensions}.

\begin{proposition}\label{prop:ODEs-danc-constr-2d}
  Given a para-CR 3-manifold $(N,\scD_1,\scD_2),$ let
  \begin{equation}\label{eq:Sclar-ODE-dancing}
    z''=F(t,z,z'),
  \end{equation}
  be a representative of  its corresponding point equivalence class of scalar 2nd order ODE, with local identification $N\cong J^1(\RR,\RR)$. Let $(t,z)$ and $(t,z,p)$ be  jet coordinates on  $J^0(\RR,\RR)$ and $J^1(\RR,\RR)$ with $t$ as the independent variable. Then its Lewy curves define a 3-dimensional path geometry on $N$ whose corresponding pair of 2nd order ODE is the point equivalence class of
  \begin{equation}\label{eq:pair-ODEs-dancing}
    z''=F(t,z,z'),\quad p''=G(t,z,p,z',p'),
  \end{equation}
  where  $G\colon J^1(\RR,\RR^2)\to \RR$ is determined by the para-CR structure. 
\end{proposition}
\begin{proof}
  Recall that every solution curve of the pairs of ODEs that define Lewy curves lies on a surface $\Sigma^2_y$ for some $y\in M_1$. Moreover, by definition, $\Sigma_y^2$ projects to a solution curve of \eqref{eq:Sclar-ODE-dancing} on $M_1$. Indeed,  $\Sigma^2_y\subset N$ locally can be identified with the bundle  $\PP TM_1$ restricted to the curve $\pi_1\circ\pi^{-1}_{2}(y)\subset M_2$, see Figure~3. Consequently, it follows that  the ODE \eqref{eq:Sclar-ODE-dancing} on $M_1$ has to be in the pair of ODEs  for Lewy curves. 
\end{proof}

Analogously, one can replace $M_1$ in Proposition \ref{prop:ODEs-danc-constr-2d} with $M_2=N\slash\scD_1$ and adopt a dual viewpoint as discussed in \ref{sec:paraCR}. Identifying $M_2$ as $J^0(\RR,\RR)$  with the natural coordinates $(\tilde t,\tilde v)$, let $(\tilde t,\tilde z,\tilde p)$ be the induced jet coordinates on $N\cong J^1(\RR,\RR)$. Similar reasoning to Proposition \ref{prop:ODEs-danc-constr-2d} demonstrates that the dual equation \eqref{eq:dual} on $M_2$ shows up in system \eqref{eq:pair-ODEs-dancing} written in coordinates $(\tilde t,\tilde z,\tilde p)$. From this viewpoint, the construction extends the dual ODE to a system as well. However, the two equations (the original one and the dual) do not appear in the system  simultaneously. Each of them only appear for particular choices of coordinate system. Nevertheless, the construction of Lewy curves can be thought of as a geometric pairing of the original ODE and its dual into one system. 
\vskip 1ex
Lewy curves and their corresponding systems of ODEs can be explicitly expressed using  a defining function $\Phi\colon M_1\times M_2\to\R$. Take local coordinates $(t,z)$ on $M_1$ and $(a, b)$ on $M_2.$ As explained in \ref{sec:paraCR}, the coordinates on $M_2$ can be interpreted as constants of integration for the scalar ODE \eqref{eq:sode} in dimension 3 or the PDE system \eqref{eq:pde-system-para-CR}. Fixing $\hat x=(\hat t,\hat z)\in M_1$ and $\hat y=(\hat a,\hat b)\in M_2,$ by Definition  \ref{def:Lewy}, it follows that  Lewy curves $\gamma_{\hat x\hat y}$  are defined as 
\begin{equation}\label{eq:LewyCurv-eqn}
  \gamma_{\hat x\hat y}=\left\{(z,t,a,b)\in M_1\times M_2\ \vline\ \Phi(z,t,a,b)=\Phi(\hat z,\hat t,a,b)=\Phi(z,t,\hat a,\hat b)=0\right\}.
\end{equation}
One is able to obtain a   pair of second order ODEs that correspond to the path geometry of Lewy curves via elimination and differentiation. More precisely, using $\Phi(\hat z,\hat t,a,b)=0$ and its first derivative with respect to $t$ one obtains
\[
  \Phi(\hat z,\hat t,a,b)=0,\quad \Phi_a(\hat z,\hat t,a,b)a'+\Phi_b(\hat z,\hat t,a,b)b'=0,
\]
which can be solved to express $\hat t$ and $\hat z$ in terms of $a$, $b$, $a'$ and $b'$. Similarly, using  $\Phi(z,t,\hat a,\hat b)=0$ and its first derivative with respect to $t$ one has
\[
  \Phi(z,t,\hat a,\hat b)=0,\quad \Phi_t(z,t,\hat a,\hat b)+\Phi_z(\hat z,\hat t,a,b)z'=0,
\]
which gives $\hat a$ and $\hat b$ in terms of $t$, $z$ and $z'$. The equation $\Phi(z,t,a,b)=0$ gives $a$ as a function of $t$, $z$ and $b$, which is then used in the expressions for $\hat t$ and $\hat z$ such that they become functions of  $t$, $z$, $b$, $z'$ and $b'$. Finally, replacing the expression for $a,$ the second derivative of $\Phi(\hat z,\hat t,a,b)=0$ and $\Phi(z,t,\hat a,\hat b)=0$ with respect to $t$ give $b''$ and $z''$ as expressions in terms of $t$, $z$, $b$, $z'$ and $b'$.

As shown in  Proposition \ref{prop:flatcase},  Lewy curves in  higher dimensional flat para-CR structures define path geometries which can  be explicitly described in terms of systems of second order ODEs, using the defining function \eqref{exa:flat-para-CR-function}. 
\begin{example}[Flat para-CR structures] \label{exa:corr-syst-odes-flat-paraCR}
  Continuing Example \ref{exa:flat-para-CR} for $n=1$, with the defining function \eqref{exa:flat-para-CR-function}, and following the recipe following \eqref{eq:LewyCurv-eqn}, one finds that Lewy curves $\gamma_{xy}$ are solutions to
  \begin{equation}\label{eq:lewy-curves-flat-3D-paraCR} 
    z''=0,\qquad b''=-\frac{2(b')^2}{z'-b}.
  \end{equation}
  As explained in Remark \ref{rmk:path-geometry-lewy-remark1}, the pair of ODEs \eqref{eq:lewy-curves-flat-3D-paraCR} also corresponds to chains of the flat para-CR structure.

  To treat higher dimensions,  recall from our discussion in \ref{sec:paraCR}, that in flat para-CR structures one has $M_1\cong\PP\PP^{n+1}.$   Let $(z,t)$ be affine coordinates on $M_1$, where $z\in \R$ and $t=(t^i)_{i=1}^n\in\R^n.$ Similarly, let $(a,b)$, where $a=(a_i)_{i=1}^n\in\R^n$ and $b\in\R,$ be affine coordinates on $M_2\cong(\PP\RR^{n+1})^*,$  identified as the space of projective hyperplanes in $M_1$. By \eqref{exa:flat-para-CR-function}, a defining function is given by
  \[
    \Phi(z,t,a,b)=z-a\cdot t- b,
  \]
  where  $ a\cdot t:=a_it^i$. In order to view  Lewy curves as solutions to a system of ODEs, we follow the recipe above, however we use $z$ as the independent variable.  The formula \eqref{eq:LewyCurv-eqn} needs to be adjusted for higher dimensions. As we discussed in Definition \ref{def:Lewy}, if $N$ is $(2n+1)$-dimensional then we need two sets of $n$-points $x_{1},\ldots,x_n\in M_1$ and $y_1,\ldots,y_n\in M_2$ in general position. To avoid confusion, we use Greek letters for the points in $M_i$ i.e. $(x_{\alpha})_{\alpha=1}^i$ and $(y_{\alpha})_{\alpha=1}^n.$ Using  coordinates $(z,t^i,a_i,b)$ from Proposition \ref{prop:jet-realization-para-cr-structures}, we have $x_{\alpha}=(z_{\alpha},t^1_{\alpha},\ldots,t^n_{\alpha})$ and $y_{\alpha}=(a_{1\alpha},\ldots a_{n\alpha},b_\alpha)_{i=1}^n.$ As a result, the defining equations for  Lewy curves are given by
  \begin{equation}\label{eqq0}
    \Phi(z,t,a,b)=\Phi(\hat z_\alpha,\hat t_{\alpha},a,b)=\Phi(z, t,\hat a_\alpha,\hat b_\alpha)=0, \qquad \alpha=1\ldots,n,
  \end{equation}
  where all $\hat z_\alpha,\hat b_\alpha\in\R$, $\hat t_\alpha=(\hat t_\alpha^1,\ldots,\hat t_\alpha^n),\hat a_\alpha=(\hat a_{\alpha 1},\ldots,\hat a_{\alpha n}) \in\R^n$ are constants. We additionally assume that the set of  $n$-tuples of vectors $(\hat a_\alpha)_{\alpha=1}^n$ and $(\hat t_\alpha)_{\alpha=1}^n$ are linearly independent which by Definition \ref{def:general-position} means they are in general positions, as required in the Definition~\ref{def:Lewy} for Lewy curves. 

  Consider first $\Phi(z,t,\hat a_\alpha,\hat b_\alpha)=0$, $\alpha=1\ldots,n$. For any values of $\hat a_\alpha$ and $\hat b_\alpha$ as above, it is a linear system of equations for $t=(t^1,\ldots,t^n)$ and $z$, whose solutions are affine lines in $M_1$. They can be parameterized as
  \begin{equation}\label{eqq2}
    t=Az+B,
  \end{equation}
  where $A,B\in\R^n$. Similarly, $\Phi(\hat z_\alpha,\hat t_\alpha,a,b)$, $\alpha=1\ldots,n$, defines a family of affine lines in $M_2$, which can be parameterized as
  \begin{equation}\label{eqq3}
    a=Zb+T
  \end{equation}
  where $Z,T\in\R^n$. Observe that each line described by equation \eqref{eqq2} or \eqref{eqq3} corresponds to many possible parameters  $\hat z_\alpha,\hat b_\alpha\in\R$, $\hat t_\alpha,\hat a_\alpha\in\R^n$ in equation \eqref{eqq0}. This freedom is essentially described in the proof of Proposition \ref{prop:flatcase}. 
 
  Now, take $(z,t,a)$ as local coordinates on $N$ and look for Lewy curves parameterized as $z\mapsto (z,a(z),t(z))$.  Form $\Phi=0$ it follows that $b=z-a\cdot t$. We differentiate \eqref{eqq2} and \eqref{eqq3} twice with respect to $z$ in order to eliminate constants $(A,B)$, and $(Z,T)$. Equation \eqref{eqq2} yields $t''=0$. On the other hand, from equation \eqref{eqq3} we get
  \[
    a'=Z(1- a'\cdot t- a\cdot t'),\quad a''=-Z(a''\cdot t+2 a'\cdot t'),\]
  in which we used the equation $t''=0.$ Eliminating $Z$, and solving for $a''$, we get
  \begin{equation}\label{eqq4}
    t''=0, \qquad a''=\frac{-2  a'\cdot t' }{1- a\cdot t'}a'.
  \end{equation}
  as an ODE system for Lewy curves of flat para-CR structures.
\end{example}
\begin{remark}\label{rmk:corr-point-equiv}
  Note that for $n=1$ the ODE pair \eqref{eqq4} differs form the one given in~\eqref{eqq4} when $n=1$. This is due to the fact that in the derivation of former pair of ODEs  $N$ is treated as $\PP TM_1,$ unlike the latter one  where it is viewed as $\PP TM_1$.  In fact, the transformation $(z,t,a)\mapsto (z,t,b)$, where  $b=a^{-1},$ is a  point transformation between these pairs of ODEs.

\end{remark}

\begin{example}\label{exa:danc-constr-2d-SD-path}
  Continuing Example \ref{exa:self-dual-non-flat} and using $t$ as the independent variable in the recipe that follows \eqref{eq:LewyCurv-eqn}, one computes that Lewy curves $\gamma_{xy}$ are solutions to the following pair of ODEs
  \[
    \begin{gathered}    
      z''=\sqrt{z'},\\
      b''=\tfrac{1}{-b'(4z'-(t+b)^2)}\left(-2(b')^2(b'+1)^2(t+b)+4\sqrt{z'}(b')^2 +2(b')^3\sqrt{(t+b)^2b'(b'+1)-4z' b'}\right).
    \end{gathered}
  \]
\end{example}

\subsection{Geometric characterization in dimension 3}\label{sec:characterization-dancing} 
In this section we present a characterization of 3-dimensional path geometries defined by para-CR Lewy curves. We begin by describing the fundamental properties of the surfaces $\Sigma^1_x$ and $\Sigma^2_y$. We shall use the following notion which is a path geometric extension of totally geodesic submanifolds in Riemannian geometry.
\begin{definition}\label{def:totallyaligned}
  Given a family of curves on a manifold $N$, a submanifold $\Sigma\subset N$ is \emph{totally aligned} with  this family if any curve from the family that is tangent to $\Sigma$ at some point   remains on $\Sigma$.
\end{definition}
\begin{proposition}\label{prop1}
  Let $(N,\scD_1,\scD_2)$ be a para-CR 3-manifold with double fibration \eqref{eq:double-fibration}. Then  surfaces $\Sigma^1_x,\Sigma^2_y\subset N$, for any $x\in M_1$ and $y\in M_2,$ are totally aligned with  Lewy curves. 
\end{proposition}
\begin{proof}
  This is a direct consequence of the constructions as all Lewy curves that are  tangent to $\Sigma^1_x$ at some point are of the form $\Sigma^1_x\cap \Sigma^2_y$, for some $y\in M_2$. Similarly, one can argue for surfaces $\Sigma^2_y.$
\end{proof}
Recall from Remark \ref{rmk:para-cr-structures-2D-path}  that a para-CR 3-manifold $N$ with double fibration \eqref{eq:double-fibration} is equivalent to  a path geometry  on the surface $M_1,$ or alternatively,  a  path geometry on $M_2$ which  is dual to the path geometry defined on $M_1.$ It follows from Proposition \ref{prop1} that the surfaces $\Sigma^1_x$ and $\Sigma^2_y$ for all $x\in M_1$ and $y\in M_2$ are equipped with their own  path geometries. This leads to the following proposition.

\begin{proposition}\label{prop2}
  Let $(N,\scD_1,\scD_2)$ be  a para-CR 3-manifold with double fibration \eqref{eq:double-fibration}.  For any $x \in M_1$ (resp. $y\in M_2$) the restriction of the projection $\pi_1 \colon N \to M_1$ to $\Sigma^1_x$ (resp. $\pi_2 \colon N \to M_2$ to $\Sigma^2_y$)   establishes an equivalence between path geometries on $\Sigma^1_x$ and $M_1$ (resp. $\Sigma^2_y$ and $M_2$). 
\end{proposition}
\begin{proof}
 Let $x\in M_1$ with corresponding surface $\Sigma^1_x\subset N$. The path geometry on $\Sigma^1_x$ is defined by intersections of $\Sigma^1_x$ with all surfaces $\Sigma^2_y\subset N$, $y\in M_2$. The image of $\Sigma^2_y$ under the projection $\pi_1$ is a path on $M_1$ corresponding to $y\in M_2$. Therefore Lewy curve $\gamma_{xy}$ is projected to a path for the original path geometry on $M_1$. Similarly, one can argue for surfaces  $M_2$ and $\Sigma^2_y.$
\end{proof}

Now we are prepared to state the following characterization of path geometries of Lewy curves.

\begin{theorem}\label{thm1}
  Let $(Q,\scV,\scX)$  be a 3-dimensional path geometry on $N=Q/\scV$. Suppose that there are two 2-parameter families $\{\Sigma^1_x\}_{x\in \tilde M_1}$ and $\{\Sigma^2_y\}_{y\in \tilde M_2}$ of surfaces in $N$ that are totally aligned with this path geometry and have the following properties:
  \begin{itemize}
  \item[(a)]  Each family of surfaces $\{\Sigma^1_x\}_{x\in \tilde M_1}$ and $\{\Sigma^2_y\}_{y\in \tilde M_2}$ has exactly a 1-parameter sub-family passing through any point of $N$. 
  \item[(b)] At every $p\in N,$  tangent planes of each of the 1-parameter family  $\{\Sigma^1_{x}\}$ and $\{\Sigma^1_{y}\}$ that pass through $p$  intersect transversely in  $T_pN$ and have a common line, denoted as $\scD_1$ and  $\scD_2,$ respectively. Any tangent direction to a path in the path geometry  passing through $p\in N$ is uniquely realized as intersection of tangent planes at $p$ of certain $\Sigma^1_x$ and $\Sigma^2_y$.
  \item[(c)] The rank 2 distribution $\scC:=\scD_1\oplus \scD_2\subset T N$ is contact.
  \item[(d)] Path geometries on all surfaces $\Sigma^1_x$, $x\in \tilde M_1$, (resp. $\Sigma^2_y$, $y\in \tilde M_2$), are equivalent via the natural projection $\pi_1\colon N\to N\slash\scD_2$ (resp. $\pi_2\colon N\to N\slash\scD_1$) with the path geometry on $N\slash\scD_2$ (resp. $N\slash\scD_1$) whose paths are the projection of integral curves of $\scD_1$ (resp. $\scD_2$)
  \end{itemize}
  Then $(Q,\scV,\scX)$ is defined by para-CR Lewy curves. Conversely, any 3-dimensional path geometry defined by para-CR Lewy curves satisfies (a)-(d) in sufficiently small open sets of $\PP(TN\backslash \cC)$.
\end{theorem}
\begin{proof} 
  The fact that path geometry of para-CR Lewy curves satisfy conditions (a)-(c) follows from Proposition \ref{prop0} when $n=1,$ setting $\tilde M_2= M_2$ and $\tilde M_1= M_1$. Furthermore, the surfaces $\Sigma^1_x$ and $\Sigma^2_y$ are totally aligned with Lewy curves by Proposition \ref{prop1} and condition (d) is satisfied by Proposition \ref{prop2}. 

  On the other hand, if conditions (a)-(c) are satisfied then the pair $(\scD_1,\scD_2)$ is a para-CR structure on $N$. The claim follows if we prove that surfaces $\Sigma^1_x$ and $\Sigma^2_y$ are necessarily obtained via the construction in \ref{sec:intrinsic} applied to $(\scD_1,\scD_2)$, because then, by the second part of (b), any path of $(Q,\scV,\scX)$ is a Lewy curve of $(\scD_1,\scD_2)$. First we prove that $\Sigma^1_x\cap \Sigma^2_y$ is a path of the original 3-dimensional path geometry on $N$ for all $x$ and $y$, provided that $\Sigma^1_x$ and $\Sigma^2_y$ intersect transversely. Indeed, since both $\Sigma^1_x$ and $\Sigma^2_y$ are totally aligned, any path tangent to $\Sigma^1_x$ and $\Sigma^2_y$ at some point has to stay on both $\Sigma^1_x$ and $\Sigma^2_y$ at all time. Hence, the path is given by $\Sigma^1_x\cap \Sigma^2_y$.

  Furthermore, by the first part of condition (b) any surface $\Sigma^1_x$ is foliated by integral curves of $\scD_1$. Hence, the projection $\gamma:=\pi_2(\Sigma^1_x)$ is a curve in $N/\scD_1.$ Thus, $\gamma$ is also  the projection of the curve $\Sigma^1_x\cap\Sigma^2_y$ for some choice of $y$ (guaranteed by the second part of (b)). We already know that  $\Sigma^1_x\cap\Sigma^2_y$ is a path of the original path geometry on $N$. Hence, by (d), $\gamma$ is necessarily the projection of an integral curve of $\scD_2$, which proves that any $\Sigma^1_x$ is of the form \eqref{eq:def-Sigma1-Sigma2} and consequently we identify $\tilde M_1\cong N/\scD_2$. Similarly, we prove that $\tilde M_2$ can be identified with $N/\scD_1$ and that any surface $\Sigma^2_y$ is of the form \eqref{eq:def-Sigma1-Sigma2} as well. 
\end{proof}

\vskip 1ex
From Theorem \ref{thm1}, dropping  condition (d)  one obtains a  class of path geometries that includes para-CR Lewy curves as a proper sub-class. We have the following.

\begin{proposition}\label{prop:freestyle}
  Any path geometry on a 3-dimensional  manifold $N$ satisfying conditions (a)-(c) of Theorem \ref{thm1} uniquely determines, in sufficiently small open sets, a triple of para-CR structures $(N,\scD_1,\scD_2)$, $(N,\tilde\scD_1,\scD_2)$, $(N,\scD_1, \tilde\scD_2)$ with the same underlying contact structure $\scC$, i.e. $\scC=\scD_1\oplus\scD_2=\tilde\scD_1\oplus\scD_2=\scD_1\oplus\tilde\scD_2$.  
\end{proposition}
\begin{proof}
  By condition (c), the path geometry defines a para-CR geometry $(N,\scD_1, \scD_2)$. Furthermore, $N$ is equipped with two additional families of surfaces $\{\Sigma^1_x\}_{x\in \tilde M_1}$ and $\{\Sigma^2_y\}_{y\in \tilde M_2}$. In general $\tilde M_1$ and $\tilde M_2$ cannot be identified with $M_1=N/\scD_2$ and $M_2=N/\scD_1$, respectively. However, by conditions (a) and (b), for any $x\in \tilde M_1$, $\Sigma^1_x$ is foliated by integral curves of $\scD_2$, and for any $y\in \tilde M_2$, $\Sigma^2_y$ is foliated by integral curves of $\scD_1$. It follows that $\pi_2(\Sigma^1_x)$ and $\pi_1(\Sigma^2_y)$ are well-defined curves in the quotient spaces $M_1=N/\scD_2$ and $M_2=N/\scD_2,$ respectively. The 2-parameter family of curves $\{\pi_2(\Sigma^1_x)\}_{x\in \tilde M_1}$ defines a supplementary path geometry on $M_2$ and the family $\{\pi_1(\Sigma^2_y)\}_{y\in \tilde M_2}$ defines a supplementary path geometry on $M_1$. Note that in this regard $\tilde M_2$ and $\tilde M_1$ become the  space of paths for the two additional path geometries on $M_1$ and $M_2$, respectively. Finally, the two path geometries give rise to para-CR structures $(N,\tilde\scD_1,\scD_2)$ and $(N,\scD_1, \tilde\scD_2)$, respectively. 

  To show the last part of the proposition, it is sufficient to notice that locally $N$ can be identified as $\PP TM_1,$ where $M_1=N\slash\scD_2$, with the canonical contact distribution $\scC$, which from this perspective splits as $\scC=\scD_1\oplus\scD_2$ or $\scC=\tilde\scD_1\oplus\scD_2$. Similarly, $N$ can be identified as  $\PP T(N\slash\scD_1)$, and then the canonical contact distribution splits as $\scC=\scD_2\oplus\scD_1$ or $\scC=\tilde\scD_2\oplus\scD_1$. 
\end{proof}

Note that Theorem \ref{thm1} describes Lewy curves on the 3-dimensional manifold $N$. However, for the corresponding path geometry it is more natural to work on the 5-dimensional manifold $Q\subset \PP(TN\backslash\cC)$. The following result provides an alternative viewpoint to Theorem \ref{thm1}.
\begin{theorem}\label{thm2}
  Let $(Q,\scV,\scX)$ be a 3-dimensional  path geometry on $N=Q/\scV$. Then $(Q,\scV,\scX)$ satisfies parts (a), (b), (c) of Theorem \ref{thm1} if and only if the vertical bundle $\scV$ has a splitting $\scV=\scV_1\oplus \scV_2$ with the following properties:
  \begin{itemize}
  \item[(a)] The rank-3 distributions
    \[
      \scB_1= [\scX,\scV_1],\qquad \scB_2=[\scX,\scV_2]
    \]
    are integrable.
  \item[(b)] $\scB_1$ and $\scB_2$ have rank-2 sub-distributions $\tilde\scK_1$ and $\tilde\scK_2$ containing $\scV_1$ and $\scV_2$, respectively, such that rank 3 distributions 
    \[\scK_1=\tilde \scK_1\oplus\scV_2 \qquad \scK_2=\tilde \scK_2\oplus\scV_1\]
    are integrable.
  \item[(c)] The projections of $\scK_1$ and $\scK_2$ to $N$ span a contact distribution.
  \end{itemize}
  Moreover, the path geometry satisfies part (d) of Theorem \ref{thm1}, i.e. it is defined by Lewy curves of a para-CR structure, when:
  \begin{itemize}
  \item[(d)] The three para-CR geometries defined in Proposition \ref{prop:freestyle} coincide.
  \end{itemize} 
\end{theorem}  
\begin{proof}
Part (d) above is exactly part (d) in Theorem \ref{thm1}. 
Thus, we only need to discuss parts (a)-(c). Let us show parts (a)-(c)  in Theorem \ref{thm2}   are equivalent to those in Theorem \ref{thm1}. Let $\tilde M_1$ and $\tilde M_2$ be leaf spaces  $Q\slash \scB_1$ and $Q\slash\scB_2$. Define surfaces $\Sigma^1_x\subset N$, for $x\in\tilde M_1$, and $\Sigma^2_y\subset N$, for $y\in\tilde M_2,$ as  projections to $N$ of the corresponding leaves of $\scB_1$ and $\scB_2$. These are well-defined surfaces, since $\scV_i$ are vertical with respect to projection $\varpi\colon Q\to N$, while $\scX$ and $[\scV_i,\scX]$ are horizontal. The surfaces $\Sigma^1_x$ and $\Sigma^2_y$ that pass through a fixed point $p\in N$ are parameterized by 1-dimensional fibers of the quotient bundles $\scV/\scV_i$, $i=1,2$, respectively. Moreover, clearly, $\varpi_*(\scX)$ is uniquely given as the intersection of $\varpi_*(\scB_1)$ and $\varpi_*(\scB_2)$, because of $\scX=\scB_1\cap\scB_2$, proving that all paths on $N$ are obtained as intersections of $\Sigma^1_x$ and $\Sigma^2_y$, for appropriate $x\in\tilde M_1$ and $y\in \tilde M_2$. Additionally, by (b), distributions $\scK_1$ and $\scK_2$ project to well-defined line fields $\scD_1$ and $\scD_2$ on $N$, which are common tangent lines of tangent bundles to $\Sigma^1_x$ and $\Sigma^2_y$, respectively, and span a contact distribution due to (c).

  The only remaining key point to note here is that condition (a) implies that $\Sigma^1_x$ and $\Sigma^2_y$ are totally aligned surfaces in $N$. This follows from the fact that the paths of the 3-dimensional path geometry are defined as projections of integral curves of $\scX$ to $N$. The paths emerging in directions tangent to $\Sigma^1_x$ and $\Sigma^2_y$ are projections of $\scX$ in the appropriate leaves of $\scB_1$ and $\scB_2$. Since $\scB_1$ and $\scB_2$ are integrable, the paths that are tangent to the surfaces at a point stay in the surfaces.

Now we show that conditions (a)-(c) in Theorem \ref{thm1} imply conditions (a)-(c) in Theorem \ref{thm2}. We note that by (b) in Theorem \ref{thm1} any direction tangent to a path in $N$ is uniquely realized as an intersection of tangent planes to the appropriate surfaces $\Sigma^1_x$ and $\Sigma^2_y$. Thus, there is a natural splitting $\scV=\scV_1\oplus\scV_2$ along integral curve of $\scX$ being a lift of a path in $N$. Indeed, points in the integral lines of $\scV_i$, $i=1,2$, parameterize directions tangent to $\Sigma^1_x$ and $\Sigma^2_y$, respectively. Because the surfaces $\Sigma^1_x$ and $\Sigma^2_y$ are totally aligned, composing the flows of $\scX$ and $\scV_i$ necessarily projects to the corresponding surfaces. This proves that distributions $\scB_i$  defined in (a) are integrable. Then, (b) and (c) are automatically satisfied for $\scK_i=\varpi^{-1}_*\scD_i$, and $\tilde \scK_i=\scK_i\cap\scB_i$.

\end{proof}

\begin{remark}\label{remark:torsion}
  Conditions (a)-(c) stated in Theorem \ref{thm2} can be directly verified for a given 3-dimensional path geometry. The key element in the theorem is the splitting  $\scV=\scV_1\oplus\scV_2$. Initially, $\scV_1$ and $\scV_2$ are unknown, but they can be easily identified using the approach presented in \cite[Section 4.2]{KM-Cayley}. This process requires determining the torsion of the structure, which is expressed by an explicit formula \eqref{eq:fels-invariants-explicit} in terms of the corresponding system of ODEs. Then, one immediately deduces that point (a) implies that $\scV_i$ is spanned by an eigenvector of the torsion $\bT$ treated as an endomorphism of $\scV$ (see Remark \ref{rmk:wuenschman}). Indeed, one can introduce $\bT$ as a trace-free part of matrix $(F^j_i)_{i,j=1,2}$ such that
  \[
    [X,[X,V_i]]=-F_i^jV_j\mod\scX
  \]
  for certain appropriately chosen sections $V_1,V_2$ of $\scV$, and $X$ of $\scX$. Consequently, it follows that, if $\scV_i$ is spanned by $V_i$ then condition $(a)$ implies that $(F^j_i)_{i,j=1,2}$ is diagonal and $V_i$'s are eigenvectors of $\bT$. Furthermore, having $\scB_i$ one can check the existence of $\scK_i$ by solving a linear system of algebraic equations. This will be expressed in terms of the curvature $\bC$ in  \ref{sec:path-geom-aris}.
\end{remark}

\subsection{Path geometry of Lewy curves in higher dimensions}\label{sec:higher-dimensions}
In Proposition \ref{prop:flatcase}, we showed that Lewy curves define a path geometry on flat para-CR structures. This path geometry is explicitly given in  Example \ref{exa:corr-syst-odes-flat-paraCR} as a system of ODEs. Now we prove the converse. 
\begin{theorem}\label{thm4}
  Given a para-CR structure $(N,\scD_1,\scD_2)$ of dimension $2n+1$, $n\geq 2$, its Lewy curves define  a path geometry on $N$ if and only if  it is flat, in which case they coincide with chains.
\end{theorem}
\begin{proof}
  If  Lewy curves define a path geometry on $N$, then any Lewy curve $\gamma_{x_1\ldots x_ny_1\ldots y_n}$ is uniquely determined by a point and a tangent direction. However, unlike the case $n=1,$ when  $n\geq 2$ a transverse tangent direction to $\scC$ does not correspond to a unique set of  intersecting $2n$ hyperplanes, as explained in Corollary \ref{cor:tangentplane}.     The proof of this theorem relies on exploiting this non-uniqueness property of the set of such intersecting hyperplanes for any tangent direction transverse to $\scC$. Take a  point $p\in N$ and a transverse  direction to the contact distribution, $\langle v\rangle\subset T_pN\backslash\scC_p,$ such that
  \begin{equation}\label{eq:two-directions}
    \langle v\rangle=\bigcap_{i=1}^nK^{\scD_1}_{i}\cap  \bigcap_{i=1}^nK^{\scD_2}_{i}=\bigcap_{i=1}^n\tilde K^{\scD_1}_{i}\cap  \bigcap_{i=1}^n\tilde K^{\scD_2}_{i}.
  \end{equation}
  Assuming that $[v]\in\PP(T_pN\backslash\scC_p)$ is tangent to a Lewy curve, as guaranteed by Corollary \ref{cor:tangentplane}, one has
  \begin{equation}\label{e1}
    \langle v\rangle=\bigcap_{i=1}^nT_p\Sigma^1_{x_i}\cap \bigcap_{i=1}^nT_p\Sigma^2_{y_i}=\bigcap_{i=1}^nT_p\Sigma^1_{\tilde x_i}\cap \bigcap_{i=1}^nT_p\Sigma^2_{\tilde y_i}
  \end{equation}
  for sets of  $n$ points $(x_i)_{i=1}^n$, $(\tilde x_i)_{i=1}^n$ in $M_1$ and $(y_i)_{i=1}^n$, $(\tilde y_i)_{i=1}^n$ in $M_2$. It follows that the corresponding Lewy curve can be obtained in two different ways
  \begin{equation}\label{e2}
    \gamma_{x_1\ldots x_ny_1\ldots y_n}=\Sigma^1_{x_1}\cap\ldots\cap \Sigma^1_{x_n}\cap \Sigma^2_{y_1}\cap\ldots\cap\Sigma^2_{y_n}=\Sigma^1_{\tilde x_1}\cap\ldots\cap \Sigma^1_{\tilde x_n}\cap \Sigma^2_{\tilde y_1}\cap\ldots\cap\Sigma^2_{\tilde y_n}.
  \end{equation}
  The observation above can be extended from directions in $T_pN\backslash\scC_p$ to  any linear subspace $R\subset T_pN\backslash\scC_p$  spanned by open set of transverse  directions $[v]\in\PP(T_pN\backslash\scC_p)$ that are tangent to Lewy curves. As a result, for any $l,k\leq n,$
  one obtain an open set of linear subspaces  $R\subset T_pN\backslash\scC_p$ that can be written as
  \begin{equation}\label{eq:L-linearSubspace}
    R=\bigcap_{i=1}^kT_p\Sigma^1_{x_i}\cap \bigcap_{i=1}^lT_p\Sigma^2_{y_i}.
  \end{equation}
  It follows that such a subspace in $T_pN\backslash\scC_p$ is tangent to a unique submanifold obtained as an intersection of the form 
  \begin{equation}\label{eq:lambda-submanifold}
    \Gamma_{x_1\ldots x_ky_1\ldots y_l}=\Sigma^1_{x_1}\cap\ldots\cap \Sigma^1_{x_k}\cap \Sigma^2_{y_1}\cap\ldots\cap\Sigma^2_{y_l}.
  \end{equation}
  It is straightforward to see that setting $l=0$ and $k=n,$ one obtains
  \begin{equation}\label{eq:lambda-submfld-M2}
    \Gamma_{x_1\ldots x_n}=\Sigma^1_{x_1}\cap\ldots\cap \Sigma^1_{x_n}\Rightarrow \pi_2(\Gamma_{x_1\ldots x_n})=H_{x_1}\cap\ldots\cap H_{x_n}=L_{x_1\ldots x_n}
  \end{equation}
  where the curves  $L_{x_1\ldots x_n}\subset M_2$ are defined in \eqref{eq:L-curves-M1-M2}. The key observation is that the family of curves of the form  $L_{x_1\ldots x_n}$ define a path geometry on $M_2.$ In fact, by our discussion after Definition \ref{def:Lewy}, the $(n+1)$-dimensional submanifolds  $\Gamma_{x_1\ldots x_n}= \pi^{-1}_2(L_{x_1\ldots x_n})\subset N$   project to such curves on $M_2$ and are uniquely determined by a tangent direction $\langle u\rangle\subset T_{y}M_2.$ More precisely, taking any $p\in N$ that is a preimage of $y$ i.e. $\pi_2(p)=y,$ then the rank $n+1$ subspace  $R=(\pi_{2})_*^{-1}(\langle u\rangle)\subset T_pN$ uniquely determines the submanifold $\Gamma_{x_1\ldots x_n}$ which implies that the direction $\langle u\rangle\subset T_yM_2$ uniquely determines the curve $L_{x_1\ldots x_n}.$  Analogously,  $M_1$ is equipped with a path geometry. Note that all Lewy curves $\gamma_{x_1\ldots x_n y_1\ldots y_n}=\Gamma_{x_1\ldots x_n}\cap \Gamma_{y_1\ldots y_n}$ project to  curves of  path geometries on $M_1$ and $M_2.$

  Now we prove that the induced path geometries on $M_1$ and $M_2$ are flat. It suffices to show for $M_1$ and the same argument works for $M_2.$  First we show that the path geometry on $M_1$ is torsion-free. We shall denote by $(\scX_{M_1},\scV_{M_1})$ the associated pair on $Q_{M_1}\subset \PP TM_1$ . As was shown before all $H_y$ are totally aligned (see Definition \ref{def:totallyaligned}) with respect to curves of the form $L_{y_1\ldots y_n}$, since $L_{y_1\ldots y_n}=\cap_{i=1}^nH_{y_i}$. Similarly, any $(n+1-k)$-dimensional manifold
  \[
    H_{y_1\ldots y_k}=H_{y_1}\cap\ldots\cap H_{y_k}\subset M_1,
  \]
  where $(y_i)_{i=1}^k$ is in general position, is totally aligned with respect to the path geometry on $M_1$.

Given a point $x\in M_1,$ take a curve $\gamma=L_{y_1\ldots y_n}$ passing through $x$ and a direction $[v]\in \PP T_xM_1$ along which the path geometry on $M_1$ is defined such that $[v]$ does not coincide with $[\dot\gamma(x)]\in\PP T_xM_1$. From the discussion above, there is a unique surface $\Pi_R$ tangent to the plane $R=\langle v,\dot\gamma\rangle$ at $x$, which is totally aligned with the path geometry on $M_1$. Consequently, similar to the proof of part (a) in Theorem \ref{thm2}, we get that $\Pi_R\subset M_1$ gives rise to a 3-dimensional submanifold $\tilde \Pi_R\subset Q_{M_1}.$ Intuitively, one has $\tilde\Pi_R\subset \PP T\Pi_R$ as an open subset. If $\scV_R\subset \scV$ is the  fiber of $\tilde\Pi_R\to \Pi_R$, since $\tilde\Pi_R$ is totally aligned with respect to the path geometry on $M_1,$   Definition \ref{def:generalized-path-geom} implies
  \begin{equation}\label{e5}
    [\scV_R,\scX_{M_1}]=T\tilde\Pi_R.
  \end{equation}
 Using the interpretation of torsion given in Remark \ref{remark:torsion}, it follows that $\scV_R$ is an eigenspace for the torsion $\bT_{M_1}$ of the path geometry $(Q_{M_1},\scX_{M_1},\scV_{M_1})$. As a result, if $n\geq 2,$ any 2-plane  in an open set of  2-planes defined similarly to $R$ at each tangent space of $M_1$ is an eigenspace of $\bT_{M_1}.$  Since by \eqref{eq:fels-invariants-path} the torsion of a path  geometry is trace-free,  it follows that $\bT_{M_1}=0.$ 

 To prove flatness of  path geometries on $M_1$ and $M_2,$ by Theorem \ref{thm:path-geom-cartan-conn} it remains to show  that  curvatures $\bC_{M_1}$ and $\bC_{M_2}$ are zero. Let us only treat  $M_1.$ Having $\bT_{M_1}=0$ for the path geometry on $M_1$ implies that  the $2n$-dimensional space of paths $S_{M_1}:=Q_{M_1}/\scX_{M_1}$ has an \emph{almost Grassmannian structure}, also known as a $(n,2)$-Segre structure, which means that $TM_1\cong E\otimes F$ for  vector bundles $E\to M_1$ and $F\to M_1$ of rank $2$ and $n,$ respectively.  In \cite{Grossman-Thesis}  it was shown that  such structures are defined on the  space of paths in torsion-free path geometries and they are additionally \emph{semi-integrable}, also referred to as \emph{$\alpha$-integrable}, which, using Cartan connection \eqref{eq:path-geom-cartan-conn-3D}, means that   the Pfaffian system $\cI_{\alpha}=\{\alpha^0,\alpha^1,\cdots,\alpha^n\}$ is integrable. The integral manifolds of $\cI_{\alpha}$  are the projection of the $n$-dimensional vertical fibers of $Q_{M_1}\to M_1$ to $S_{M_1}$, whose lift foliate the $\PP^1$-bundle $\PP E\to M_1.$ Additionally, it turns out that the path geometry of para-CR Lewy curves when $n\geq 2$ is also \emph{$\beta$-integrable}, i.e. with respect to the Cartan connection \eqref{eq:path-geom-cartan-conn-3D} the Pfaffian system $\cI_{\beta}=\{\alpha^2,\cdots,\alpha^n,\beta^2,\cdots,\beta^n,\psi^2_1,\cdots,\psi^n_1\}$ is integrable. Note that the integral leaves of $\cI_{\beta}$ are surfaces that foliate the $\PP^{n-1}$-bundle $\PP F\to M_1.$  Their $\beta$-integrability   is due to the fact that at each point $q\in S_{M_1}$ the submanifolds  $\Pi_{R}$ defined previously are tangent to rank 2 subspaces of the form $\zeta\otimes\R^2$, for all $\zeta\in \R^n$, using the isomorphism $T_qS_{M_1}\cong \R^n\otimes\R^2.$ Such surfaces are also  referred  to as  \emph{$\beta$-surfaces}. As was shown in \cite{Grossman-Thesis},  Segre structures that are both $\alpha$-integrable and $\beta$-integrable are flat. This fact can be verified by checking the Frobenius condition for $\cI_{\beta}$ using structure equations \eqref{eq:path-geom-cartan-conn-3D} for torsion-free path geometries,  which are found to be equivalent to  $C^i_{jkl}=0.$

  Having shown that the induced path geometries on $M_1$ and $M_2$ are flat, we now  show that the initial para-CR structure on $N$ is flat. Since $M_1$ has a flat path geometry,  it is locally   isomorphic to $\RR\PP^{n+1}$  with  its canonical path geometry defined by projective lines. By our discussion above, the hyperplanes $H_x\subset M_1$ have to be totally aligned with the path geometry and, therefore, are the projective hyperplanes. By the discussion in \ref{sec:paraCR}, it follows that $N$ is  locally isomorphic to  the flat model in  para-CR geometry.

  The last part of the theorem involving chains follows from  Proposition \ref{prop:lewy-curves-chains}.
\end{proof}
\begin{proposition}\label{prop:lewy-curves-chains}
  Path geometries defined by chains of flat para-CR structures  correspond to point equivalence classes of  ODE systems \eqref{eqq4}.
\end{proposition}
\begin{proof}
  Although  the case $n=1$ follows from  Theorem \ref{cor:3d-path-geometries-dancing-chains}, one can also directly see that the pair of ODEs \eqref{eq:lewy-curves-flat-3D-paraCR} for Lewy curves of flat para-CR structures is identical to the pair of ODEs obtained for chains of flat para-CR 3-manifolds in \cite[Equation (3.12)]{KM-chains}.

  For  $n\geq 2,$  following the derivation method in \cite[Section 3.2.1]{KM-chains},  the system of ODEs for para-CR chains can be obtained, which turn out to be identical to those for Lewy curves in \eqref{eqq4}.

  To do so we first recall  the Cartan geometric description of chains in para-CR structure following \cite{CZ-CR}. By Theorem \ref{thm:paraCR-cartan-conn}, para-CR structures are encoded by Cartan geometries $(\cP\to N,\phi)$ of type $(\mathrm{PGL}(n+2,\RR),P_{1,n+1})$ where the Cartan connection $\phi$ is a $\mathfrak{sl}(n+2,\RR)$-valued 1-form. The grading of $\mathfrak{sl}(n+2,\RR)$ that corresponds to $P_{1,n+1}$ is a contact grading
  \[\mathfrak{sl}(n+2,\RR)=\fg_{-2}\oplus\fg_{-1}\oplus\fg_{0}\oplus\fg_{1}\oplus \fg_2.\]  
  As unparametrized curves, chains are defined on $N$ and  are the projection of the  integral curves of the the vector field $\phi^{-1}(X)\in\Gamma(T\cP)$ from $\cP$ to $N,$ where $X$ is a nonzero element of $\fg_{-2}.$ In particular, it is clear from the definition that chains are transverse to the contact distribution on $N.$   More precisely, let us describe the set of transverse directions to the contact distribution. Let $S\subset P$ be the stabilizer of the line $l\subset\mathfrak{sl}(n+2,\RR)$  spanned by an element of $\fg_{-2}.$ It follows that  $\fs=\fg_{0}\oplus\fg_2\subset \fp=\fg_0\oplus\fg_1\oplus\fg_2$ where $\fs$ and $\fp$ are the  Lie algebras of $S$ and $P,$ respectively.  Moreover, it is straightforward to see that the  $P$-orbit of $l,$ denoted as $P\cdot l\subset\mathfrak{sl}(n+2,\RR)\slash\fp,$ is the set of all lines in $\mathfrak{sl}(n+2,\RR)\slash\fp$ that are not contained in $\fg^{-1}\slash\fp$ where $\fg^{-1}:=\fg_{-1}\oplus\fg_0\oplus\fg_1\oplus\fg_2$ and $\cP\times_P\fg^{-1}\slash\fp$ is the  contact distribution $\scC\subset TN.$ Thus, the bundle of directions that are transverse to the contact distribution, i.e.  $Q:=\PP(TN\backslash\scC),$ can be identified with $\cP\slash S:=\cP\times_P
  P\slash S.$ Viewing the Cartan connection $\phi\in\Omega^1(\cP,\mathfrak{sl}(n+2,\RR))$ as a connection on $Q,$ it gives 
  \[TQ\cong\cP\times_S (\mathfrak{sl}(n+2,\RR)\slash \fs).\] 
  Using the above identification, it follows that the induced path geometry  $(Q,\scV,\scX)$ is give by $\scX=\nu_*(\phi^{-1}(\fg_{-2}))$ and $\scV=\nu_*(\phi^{-1}(\fg_{1}))$ where $\nu\colon\cP\to Q$ is the projection.

  To relate the Cartan connections of para-CR structures, given by $(\cP\to N,\phi)$ where $\phi$ is as in \eqref{eq:path-geom-cartan-conn-3D}, and that of the path geometry $(\cG\to Q,\psi)$ defined by its chains, where $\psi$ is given by \eqref{eq:paraCR-geom-cartan-conn-3D},  using the extension functor \cite[Section 3.1]{CZ-CR}, consider the natural bundle map $\iota\colon \cP\to \cG:=\cP\times_S P_{1,2},$ determined by Lie group homomorphism $\iota\colon S\to P_{1,2}$  where $P_{1,2}\subset \mathrm{SL}(2n+2,\RR)$ is the parabolic defined in \ref{sec:path-geom-defin}. Then one obtains
  \begin{equation}\label{eq:inclusion-para-path}
    \def\arraystretch{1.3}\iota^*\psi=
    \begin{pmatrix}
      -\phi^{n+1}_{n+1}-\lambda & \theta_0 & \half \theta_{n+i} & \half\theta_i\\
      \omega^0 & \phi^{n+1}_{n+1}+\lambda & \half\omega_{n+i} & -\half\omega_i\\
      \omega^i & \theta^i & \phi^i_j+\lambda\delta^i_j & 0\\
      \omega^{i+n} & -\theta^{n+i} & 0 & -\tilde\phi^i_j-\lambda\delta^i_j
    \end{pmatrix}
  \end{equation}
  where $\lambda =\tfrac{1}{2}\phi^i_i,$ $\tilde\phi^i_j=\phi^j_i$ and we have used $\delta^i_j$ to raise and lower indices in $\omega^i,\theta^i$ and $\omega_{n+i}$ and $\theta_{n+i}.$

  Lastly, similar to the observation made in \cite[Section 6]{MS-cone} and \cite[Proposition 3.2]{KM-chains}, chains on $Q$ are characteristic curves of the 2-form
  \begin{equation}\label{eq:rho-quasi-symp}
    \rho:=s^*(\omega^i\w\theta_{n+i}-\omega_{i+n}\w\theta^i)\in\Omega^2(Q),
  \end{equation}
  for any section $s\colon Q\to \cP.$ More precisely, using the Cartan connection $\iota^*\psi$  \eqref{eq:inclusion-para-path}, the tangent field of chains are given by $\langle v\rangle$ where $v=\tfrac{\partial}{\partial s^*\omega^0}$ which is the characteristic direction of $\rho,$ i.e.
  \begin{equation}\label{eq:char-direction}
    \rho(v)=0,\quad v\im\exd\rho=0.
  \end{equation}
  To compute the 2-form $\rho$ on $Q$ we start with a choice of  coordinate system for a flat para-CR structure on $N$. Recall from Proposition \ref{prop:jet-realization-para-cr-structures} that by the local isomorphism $N\cong J^1(\RR^n,\RR),$ one has the following coframe on $N$
  \[\zeta^0=\exd z - p_i\exd t^i,\quad \zeta^i=\exd t^i,\quad \zeta_{n+i}=\exd p_i-f_{ij}\exd t^j,\]
  wherein  $f_{ij}$ can be set to zero in the case of flat para-CR structures. To relate the coframe above to the entries of $\psi$ in \eqref{eq:paraCR-geom-cartan-conn-3D}, it follows that the only non-zero entries are given by
  \begin{equation}\label{eq:0-adapted-coframe-zeta}
    s^* \omega^0=\zeta^0,\quad s^*\omega^i=\zeta^i,\quad s^*\omega_{n+i}=\zeta_{n+i}
  \end{equation}
  for a section $s\colon J^1(\RR^{n},\RR)\to\cP.$

  To introduce a coordinate system on $Q:=\PP(TN\backslash\scC)$, which we identified with $\cP\slash S$ above, we shall first parameterize the fibers of $\cP\to N$  which is the parabolic subgroup $P_{1,n+1}\subset\mathrm{SL}(n+2,\RR)$. Writing $P_{1,n+1}=P_0\rtimes P_+,$ where $P_0$ is the
  reductive subgroup, referred to as the structure group, and $P_+$ is the nilpotent normal subgroup, one
  has  $P_0=\mathrm{GL}(n,\RR)\times \RR^+$ and
  \begin{equation}\label{eq:P+}
    P_+=\left\{ B\in \mathrm{SL}(n+2,\RR)\ \vline\  B= \begin{pmatrix}
        1 & q_{i+n} & \half q^iq_{i+n} +q_0\\
        0 & \mathrm{Id} & q^i\\
        0 & 0 & 1
      \end{pmatrix}\right\}.
  \end{equation}
  By the above description of $S$ and that $Q=\cP\slash S,$ it follows that $Q$ can be identified with a section $t\colon Q\to \cP$ given by  $P_0=\mathrm{Id}$ and $q_0=0$ in \eqref{eq:P+}, as a result of which the parameters $(q^i,q_{n+i})$ can be taken as the fiber coordinates for the fibration $Q\to N.$ Now we have a coordinate system $(z,t^i,p_i,q^i,q_{n+i})$ on $Q.$

  The next step is to obtain an  adapted coframe on $Q$ from the coframe \eqref{eq:0-adapted-coframe-zeta} on $N.$ Recall that the transformation of the Cartan connection along the fibers of  $\cP\to N$ by an action of $g\in P_{1,n+1}$ is given by
  \begin{equation}
    \label{eq:gauge-transformation-Cartan-conn}
    \phi(w)\to \phi(r_gw)=g^{-1}\phi(w)g+g^{-1}\exd g.
  \end{equation}
  As a result, using the gauge transformation \eqref{eq:gauge-transformation-Cartan-conn} and restricting  $g$ to the section $t\colon Q\to \cP$ which we identified with $Q$, i.e.  $P_0=\mathrm{Id}$ and $q_0=0,$  one obtains a lift of the adapted coframe \eqref{eq:0-adapted-coframe-zeta} to $Q$ given by
  \[
    \begin{gathered}
      t^*\omega^0=\zeta^0,\quad t^*\omega^i=\zeta^i-q^i\zeta^0,\quad t^*\omega_{n+i}=\zeta_{n+i}+q_{n+i}\zeta^0\\
      t^*\theta^i=\exd q^i+\half q^jq_{n+j}(\zeta^i-q^i\zeta^0)-q^i(q^j\zeta_{j+n})\\
      t^*\theta_{n+i}=\exd q_{n+i}+\half q^jq_{n+j}(\zeta_{n+i}+q_{n+i}\zeta^0)-q_{n+i}(q_{n+j}\zeta^{j})\\
    \end{gathered} 
  \]
  Using this adapted coframe, the close 2-form $\rho$ in \eqref{eq:rho-quasi-symp} is given by
  \begin{equation}\label{eq:quasi-symp}
    \begin{aligned}
      \rho=&(q_{n+i}\exd q^i+q^i\exd q_{n+i} )\w \exd z + q_{i+n}q^i\exd t^j\w\exd p_j-\exd p_i\w\exd q^i\\
      &+\exd t^i\w\exd q_{n+i}+p_j\exd t^j\w(q_{n+i}\exd q^i+q^i\exd q_{i+n})
    \end{aligned}
  \end{equation}
  As was mentioned above, chains are the characteristic directions of $\rho$ and since they  are  transverse to the contact distribution, it follows that they can be expressed as $\langle v \rangle$ where
  \begin{equation}\label{eq:v-char-direc}
    v=\partial_z+\tfrac{q^i}{p_j q^j+1}\partial_{t^i}+\tfrac{q_{i+n}}{p_j q^j+1}\partial_{p_i} +\tfrac{q^iq_{n+i}}{p_jq^j+1}(q_{n+k}\partial_{q_{n+k}}-q^k\partial_{q^k} )
  \end{equation}
  Using the  change of variables
  \[T^i=\tfrac{q^i}{p_j q^j+1},\quad P_i=\tfrac{q_{i+n}}{p_j q^j+1}\]
  one obtains
  \[v=\partial_z+T^i\partial_{t^i}+P_{i}\partial_{p_i}+\tfrac{2P_jT^j}{T^kp_k-1}P_i\partial_{P_i}\]
  Comparing $v$ to the line field $X_F$ in \eqref{eq:TotalDerivative} defined by a path geometry, it follows that chains are the solution curves of the point equivalence class of  ODE system
  \begin{equation}\label{eq:ODEs-chains}
    (t^i)''=0,\quad (p_i)''=\tfrac{2(p_j)'(t^j)'}{p_k(t^k)'-1}(p_i)'.
  \end{equation}
\end{proof}
 
\begin{remark}\label{rmk:path-geometry-lewy-remark1}
  The point equivalence class of pair of ODEs that corresponds to  chains of a flat para-CR 3-manifold, as given in \cite[Equation (3.12)]{KM-chains}, coincides with the pair of ODEs \eqref{eq:lewy-curves-flat-3D-paraCR}, which by Remark \ref{rmk:corr-point-equiv} is point equivalent to the ODE system \eqref{eq:ODEs-chains} when $n=1.$ We point out that for an arbitrary analytic para-CR structure, starting from the coframe \eqref{eq:0-adapted-coframe-zeta} for a set of compatible functions $f_{ij},$ after obtaining a fully adapted coframe, as in \cite[Equation (3.7)]{KM-chains} in the 3-dimensional case, one can follow the same steps as in the proof of Proposition \ref{prop:lewy-curves-chains} to find the systems of ODEs for  its chains. 
  Lastly, we point out that, as we have seen so far,  many notions in CR geometry, e.g. chains, Segre family, and Lewy curves, naturally extend to the para-CR setting.  It would be interesting to develop a para-CR analogue of the so-called  \emph{stationary discs} and their relation to the Segre family (see \cite{Florian2} for such a relation in the CR case).

\end{remark}

\begin{remark}
  In the proof of Theorem \ref{thm4}, the relation between a flat projective structure on $M_1$ and flat para-CR structure on $N$ is the special case of the correspondence space construction in \cite{Takeuchi}  where projective structures on a manifold $M$ correspond to certain classes of \emph{partially integrable almost para-CR structure} on $\PP T^*M$; also see \cite{CapCrelle}. When the projective structure is flat, the corresponding almost para-CR structure is  the flat para-CR structure as in Theorem \ref{thm4}. 
\end{remark}

\section{Cartan geometric description  in dimension 3}\label{sec:path-geom-aris}   

In this section we give a Cartan geometric description of  path geometries defined by Lewy curves of a 3-dimensional para-CR structure. Throughout the section $(Q,\scV,\scX)$ is a 3-dimensional path geometry on $N=Q/\scV$. We carry out a natural reduction of the principal bundle for such path geometries in \ref{sec:reduct-princ-bundle},     investigate the cases of zero and nonzero torsion in \ref{sec:path-geometry-lewy-torsionfree} and \ref{sec:non-trivial-torsion}.

\subsection{Reduction of the principal bundle}\label{sec:reduct-princ-bundle}
In this section we  describe a reduction of the principal bundle in Theorem \ref{thm:path-geom-cartan-conn} which is canonically defined for the path geometry of Lewy curves.

\begin{proposition}\label{prop:torsion-curv-type-freestyling-reduction}
  Given a 3-dimensional path geometry with Cartan data $(\cG\to Q,\psi),$ if it satisfies conditions (a)-(c) in Theorem \ref{thm2} then it has a canonical reduction to  a principal $B$-bundle $\iota\colon\cG_{red}\hookrightarrow\cG,$  where $B\subset\mathrm{GL}(2,\RR)$ is the Borel subgroup, over which the following invariant conditions hold
  \[
    \iota^* A_0=\iota^*A_2=0,\quad \iota^*W_0=\iota^*W_4=0,\quad  \iota^*W_2=\pm 1,
  \]
  and the components of $\iota^*\psi$, as given in \eqref{eq:path-geom-cartan-conn-3D}, satisfy
  \[
    \begin{aligned}
      \iota^*\psi^2_1=\iota^*\psi^1_2=& 0,&&\\
      \iota^*\psi^2_2\equiv& -\iota^*\psi^1_1&&\mod \{\iota^*\alpha^0,\iota^*\alpha^1,\iota^*\alpha^2\}\\
      \iota^*\nu_1,\iota^*\nu_2\equiv & 0&&\mod  \{\iota^*\alpha^1,\iota^*\alpha^2\}\\
      \iota^*\mu_1,\iota^*\mu_2\equiv & 0&&\mod  \{\iota^*\alpha^0,\iota^*\alpha^1,\iota^*\alpha^2,\iota^*\beta^1,\iota^*\beta^2\}
    \end{aligned}
  \]
  As a result, the  curvature $\bC$ of  such path geometries, as a binary quartic \eqref{eq:quadric-quartic},  has at least two distinct real roots whose multiplicities are at most two and its torsion $\bT$ is either zero or has two distinct real roots. 
\end{proposition}
\begin{proof}
  The proof is done via a standard application of Cartan's reduction procedure. Recall that a 3D path geometry is a Cartan geometry of type $(\mathrm{PGL}(4,\RR),P_{1,2})$ where $P_{1,2}=P_0\ltimes P_+$ is the stabilizer of a flag of a line inside a plane in $\RR^4$,  $P_0=\RR^*\times\mathrm{GL}(2,\RR)$ is the reductive subgroup of $P_{1,2},$ referred to as the structure group, and $P_+$ is the nilpotent normal subgroup of $P_{1,2}$. 
  Parametrically,  the structure group $P_0$ is a  block-diagonal matrix expressed as
  \begin{equation}
    \label{eq:P-0-3Dpath}
    P_0=\left\{A\in\mathrm{SL}(4,\RR)\ \vline \ A=\mathrm{diag}(\tfrac{1}{a_{00}\det(\bH)},a_{00},\bH),\bH=
      \begin{pmatrix}
        a_{11} & a_{12}\\
        a_{21} & a_{22}
      \end{pmatrix}
    \right\}
  \end{equation}
  and one has 
  \begin{equation}
    \label{eq:P-p-3Dpath}
    P_+=\left\{B\in\mathrm{SL}(4,\RR)\ \vline\  B=
      \begin{pmatrix}
        \hspace{-.3cm}1 & \hspace{-.3cm} p_0 & p_1+\half p_0q_1 & p_2+\half p_0q_2 & \\
        \hspace{-.3cm}0 & \hspace{-.3cm} 1 & q_1 & q_2 &\\
        \ \ 0_{2\times 1} & \ \ 0_{2\times 1} & &\hspace{-1.9cm} \mathrm{Id}_{2\times 2}&\\
      \end{pmatrix}
    \right\}.\end{equation}

  By Theorem \ref{thm2}, the path geometry gives a splitting  $\scV=\scV_{1}\oplus\scV_{2}$. Since the structure group of a 3-dimensional path geometry is $P_0=\RR^*\times\mathrm{GL}(2,\RR),$ having a splitting reduces $P_0$ to $(\RR^*)^3$ in order for the splitting to be preserved. More precisely, one can choose 1-forms $\beta^1$ and $\beta^2$ in \eqref{eq:path-geom-cartan-conn-3D}  to have the additional coframe adaptation  $\scV_a=\langle\tfrac{\partial}{\partial s^*\beta^a}\rangle$ and $\scX=\langle\tfrac{\partial}{\partial s^*\alpha^0}\rangle$ for any section $s\colon Q\to\cG.$   Such adaptation reduces the group parameters $a_{12}$ and $a_{21}$ in \eqref{eq:P-0-3Dpath} and results in an inclusion $\iota_1\colon\cG^{(1)}\to \cG,$ where $\cG^{(1)}$ is a principal $P^{(1)}$-bundle with $P^{(1)}:=(\RR^*)^3\ltimes P_+.$ Since the group parameters $a_{12}$ and $a_{21}$ correspond to the entries $\psi^1_2$ and $\psi^2_1$ in the Cartan connection \eqref{eq:path-geom-cartan-conn-3D}, one obtains that $\iota_1*\psi^1_2\equiv\iota^*_1\psi^2_1\equiv 0$ modulo $\{\iota^*_1\alpha^0,\iota^*_1\alpha^1,\iota^*_1\alpha^2,\iota^*_1\beta^1,\iota^*_1\beta^2\},$ i.e.  
  \begin{equation}\label{eq:psi12-red-freestyling}
    \iota_1^*\psi^1_2=A^1_{2i}\alpha^i+B^1_{2a}\beta^a,\quad     \iota_1^*\psi^2_1=A^2_{1i}\alpha^i+B^2_{1a}\beta^a 
  \end{equation}
  for some functions $A^a_{bi}$ and $B^a_{bc}$ on $\cG^{(1)}.$

  Furthermore,  this adaptation  implies $[\scX,\scV_a]=\langle\tfrac{\partial}{\partial s^*\alpha^a}\rangle$.  As a result, for the rank 3 integrable distributions in Theorem \ref{thm2} one has $\scB_a=\ker \cI_a$  where $\cI_1:=\{\alpha^2,\beta^2\}$ and  $\cI_2:=\{\alpha^1,\beta^1\}.$ Using the structure equations, one immediately obtains that the integrability of $\iota_1^*\cI_{1}$ and $\iota_1^*\cI_{2}$  implies that with respect to this coframe adaptation the following torsion entries in \eqref{eq:quadric-quartic} vanish:
  \begin{equation}\label{eq:A02-zero-freestyling}  
    \iota_1^*A_0=0,\quad \iota_1^*A_2=0.
  \end{equation}
  Moreover, as a result of  differential relations arising from the integrability of $\iota_1^*\cI_{1}$ and $\iota_1^*\cI_{2},$ one obtains
  \begin{equation}\label{eq:W04-zero-freestyling}
    \iota_1^*W_0=0,\quad \iota^*_1W_4=0,\quad \iota_1^*\psi^1_2=A^1_{21}\alpha^1+B^1_{21}\beta^1,\quad \iota_1^*\psi^2_1=A^2_{12}\alpha^2+B^2_{12}\beta^2.
  \end{equation}
  By part (b) of Theorem \ref{thm2} the integrability of $\scK_1$ and $\scK_2$ are required for path geometry of Lewy curves. In order to impose these integrability conditions we first carry out a canonical reduction of $\cG^{(1)}\to Q$ to a principal $P^{(2)}$-bundle $\iota_2\colon\cG^{(2)}\hookrightarrow\cG^{(1)}$ where $P^{(2)}=(\RR^*)^3\times\RR.$ To do so we first 
  find the action of $P^{(1)}$ on the quantities $A^a_{bk}$ and $B^a_{bc}$ in \eqref{eq:W04-zero-freestyling}  defined on $\cG^{(1)}.$ In particular, using the parametrization in \eqref{eq:P-p-3Dpath} and assuming $a_{12}=a_{21}=0,$ acting by $g\in P^{(1)}$ gives
  \begin{equation}\label{eq:AB-group-action-chains}
    \begin{aligned}
      B^1_{21}(g^{-1}u)=&\tfrac{a_{22}}{a_{00}}B^1_{21}(u)+q_2\\
      B^2_{12}(g^{-1}u)=&\tfrac{a_{11}}{a_{00}}B^2_{12}(u)+q_1\\
      A^1_{21}(g^{-1}u)=&{a_{00}a_{11}a_{22}^2A^1_{21}(u)- \tfrac{a_{22}}{a_{00}}p_0B^{1}_{21}(u) -\half p_0q_2+p_2}\\
      A^2_{12}(g^{-1}u)=& a_{00}a_{11}^2a_{22}A^2_{12}(u)- \tfrac{a_{11}}{a_{00}}p_0B^2_{12}(u)-\half p_0q_1+p_1
    \end{aligned}
  \end{equation}
  Infinitesimally, these actions correspond to Bianchi identities 
  \begin{equation}
    \label{eq:AB-inf-action-Bianchies-chains}
    \begin{aligned}
      \exd B^1_{21}\equiv& (\psi^2_2-\psi^0_0)B^1_{21}+\nu_2\\
      \exd B^2_{12}\equiv& (\psi^1_1-\psi^0_0)B^2_{12}+\nu_1\\
      \exd A^1_{21}\equiv&  (\psi^0_0+\psi^1_1+2\psi^2_2)A^1_{21}-B^1_{21}\mu_0+\mu_2\\
      \exd A^2_{12}\equiv&  (\psi^0_0+2\psi^1_1+\psi^2_2)A^1_{21}-B^2_{12}\mu_0+\mu_1
    \end{aligned}
  \end{equation}
  modulo $\{\iota_1^*\alpha^0,\iota_1^*\alpha^1,\iota_1^*\alpha^2,\iota_1^*\beta^1,\iota_1^*\beta^2\}.$ As a result, the sub-bundle $\iota_2\colon\cG^{(2)}\hookrightarrow \cG^{(1)}$ given by
  \begin{equation}
    \label{eq:G3-reduced-Pp-chains}
    \cG^{(2)}=\{u\in\cG^{(1)}\ \vline\ B^1_{21}(u)=B^2_{12}(u)=A^1_{21}(u)=A^2_{12}(u)=0\}
  \end{equation}
  is well-defined as a principal $P^{(2)}$-bundle where $P^{(2)}\cong(R^*)^3\ltimes\RR$. In terms of parametrizations \eqref{eq:P-0-3Dpath} and \eqref{eq:P-p-3Dpath} for $P_{1,2}=P_0\ltimes P_+,$ one can express $P^{(2)}\subset P_{1,2}$ as  $a_{12}=a_{21}=p_1=p_2=q_1=q_2=0.$


  As a result, the integrability of Pfaffian systems
  \[\cI_3=\mathrm{Ann}\scK_1=\mathrm{Ker}\{\iota^*_2\circ\iota^*_1\alpha^0,\iota^*_2\circ\iota^*_1\alpha^1\},\quad \cI_4=\mathrm{Ann}\scK_2=\mathrm{Ker}\{\iota^*_2\circ\iota^*_1\alpha^0,\iota^*_2\circ\iota^*_1\alpha^2\}\] on $\cG^{(2)}$ is well-defined. It is matter of straightforward computation to show that if $\cI_3$ and $\cI_4$ are integrable  on $\cG^{(2)}$ then
  \begin{equation}\label{eq:W2-sign}
    \exd\alpha^0\equiv W_2\alpha^1\w\alpha^2\mod\{\alpha^0\},
      \end{equation}
  wherein we have suppressed  the pull-backs. Consequently, condition (c) in Theorem \ref{thm2} implies that on $\cG^{(2)}$ one has
  \begin{equation}\label{eq:W2-zero-freestyling}
    W_2\neq 0
  \end{equation}
  which combined with  \eqref{eq:W04-zero-freestyling}  shows that the quartic $\bC$ has at least two real roots whose multiplicity is at most two.

  Lastly, using \eqref{eq:W2-zero-freestyling}, there is a third reduction of the structure bundle to a principal $P^{(3)}$-bundle $\iota_3\colon\cG^{(3)}\hookrightarrow\cG^{(2)}$   where $P^{(3)}= (\RR^*)^2\ltimes\RR\cong B,$ and $B\subset\mathrm{GL}(2,\RR)$ is the Borel subgroup, defined as
  \begin{equation}
    \label{eq:P3-3rd-reduction-freestyling}
    \cG^{(3)}=\{u\in\cG^{(2)}\ \vline\ W_2(u)=\pm 1\}.
  \end{equation}


  The action of $P^{(2)}$ on $W_2$ is given by
  \begin{equation}
    \label{eq:W2-group-action}
    W_2(g^{-1}u)=a_{11}^2a_{22}^2W_2(u).
  \end{equation}
  Thus, depending on the sign of $W_2,$ one can normalize it to $\pm 1.$   From now on we assume   $W_2> 0$ since the case $W_2<0$ can be treated identically.   See Remark \ref{rmk:sign-of-W_2-chains} for the difference of outcome in these two cases which has to do with the sign of the induced almost para-complex structure on $\scC\subset TN.$. 
  It follows that $\cG^{(3)}\to Q$ is a principal $P^{(3)}$-bundle where $P^{(3)}=(\RR^*)^2\ltimes \RR\cong B$ and $B\subset \mathrm{GL}(2,\RR)$ is the Borel subgroup. In terms of the parametrization \eqref{eq:P-0-3Dpath}, $B\subset P$ is  given by $a_{12}=a_{21}=q_1=q_2=p_1=p_2=0$ and $a_{22}=1/a_{11}.$ 
  
  On $\cG^{(3)}$ one obtains that $\iota^*_3(\psi^2_2+\psi^1_1)$ is semi-basic with respect to the fibration $\cG^{(3)}\to Q.$  This finishes the proof where  $\cG_{red}:=\cG^{(3)}$ and $B=(\RR^*)^2\ltimes\RR$ and $\iota:=\iota_3\circ\iota_2\circ\iota_1.$ The full structure equations for the $\{e\}$-structures $\iota^*\psi$ on $\cG_{red}$ is given in \eqref{eq:freestyling-streqs} for some quantities $x_i$'s on $\cG_{red}$ where
  \[\iota^*\psi^2_2=-\iota^*\psi^1_1+4\iota^*x_{15}\alpha^0+4\iota^*x_{11}\alpha^1-4\iota^*x_{12}\alpha^2.\]

\end{proof}

\begin{remark}\label{rmk:sign-of-W_2-chains}
Note that in the case of complex para-CR structures, the reality condition for the roots of $\bC$ in Proposition \ref{prop:torsion-curv-type-freestyling-reduction} will not longer hold. Moreover,   the sign of $W_2$ determines the sign of the corresponding almost para-complex structure on  the contact distribution  $\scC=\scD_1\oplus\scD_2$ via the splitting $\scV=\scV_1\oplus\scV_2$ in 3D path geometry $(Q,\scV,\scX)$ equipped with reduction $\cG^{(2)}\to Q,$ where $\scD_i=\tau_*\langle\tfrac{\partial}{\partial\alpha^i}\rangle=\tau_*\langle [\scX,\scV_i]\rangle$.  More precisely, by \eqref{eq:W2-sign},  $W_2>0$ implies that $\scD_1$ and $\scD_2$ are $+1$ and $-1$ eigenspaces of the almost para-complex structure, respectively, and that the orientation of the solution curves of the path geometry coincides with the coorientation of the contact distribution. Changing the orientation of $\scV$, one obtains $W_2<0$ and $\scD_1$ and $\scD_2$ become $-1$ and $+1$ eigenspaces of the almost para-complex structure, respectively. As a result, locally,  the initial path geometry determines the almost para-complex structure up to a sign. See also Corollary \ref{cor:determination}
\end{remark}

\begin{remark}\label{rmk:bundle-G}
  We point out that the principal $B$-bundle $(\cG_{red}\to Q,\iota^*\psi)$ does not define a Cartan geometry, as can be seen from the non-horizontal 2-forms in the structure equation for $\exd\mu_0$ in \eqref{eq:freestyling-streqs}. However the principal $\RR^*\times B$-bundle $(\cG^{(2)}\to Q,(\iota_2\circ\iota_1)^*\psi)$ does define a Cartan geometry. In other words, the third reduction that  involves the normalization of $W_2$ to $\pm 1$ breaks the equivariancy. It is a matter of straightforward computation to show that if one requires $(\cG_{red}\to Q,\iota^*\psi)$ to be a Cartan geometry, i.e. $x_{11}=x_{12}=x_{15}=0,$ then the path structure is torsion-free, which is the case treated in Proposition \ref{prop:torsionfree-freestyling}.
\end{remark}

\begin{remark}  \label{rmk:dancing-Not-Variational}
 In the language of \cite{MS-cone}, $(\cG^{(2)}\to Q,(\iota_2\circ\iota_1)^*\psi)$ defines an \emph{orthopath geometry} since the conformal class of the 2-form $\rho=\alpha^1\w\beta^2+\alpha^2\w\beta^1\in\Omega^2(\cG^{(2)})$ is well-defined  and satisfies
  \[
    \exd\rho=-2(\psi^1_1+\psi^2_2)\w\rho+ A_1\alpha^0\w\alpha^1\w\alpha^2.
  \]
  As a result, such structures define a \emph{variational} orthopath geometry, i.e. $\exd\rho\equiv 0$ modulo $\{\rho\}$ if and only if they are torsion-free which is treated in the next section.
\end{remark}

\subsection{Path geometry of Lewy curves: torsion-free case}\label{sec:path-geometry-lewy-torsionfree}
As already mentioned in Remarks \ref{rmk:bundle-G} and \ref{rmk:dancing-Not-Variational}, torsion-free path geometries defined by Lewy curves are very special. In fact, we have the following characterization.
\begin{proposition}\label{prop:torsionfree-freestyling}
  Let $(Q,\scV,\scX)$  be a path geometry on  $N=Q/\scV$ satisfying conditions (a)-(c) of Theorem \ref{thm2}. Then the following  are equivalent:
  \begin{enumerate}
  \item The torsion of $(Q,\scV,\scX)$ is zero.
  \item The curvature of $(Q,\scV,\scX)$ has 2 distinct real roots of multiplicity 2.    
  \item $(Q,\scV,\scX)$ is locally equivalent to chains of a flat para-CR structure.
  \end{enumerate}
\end{proposition}
\begin{proof}
  Let us show $(1)\to (2).$  Using the  $\{e\}$-structure obtained in Proposition \ref{prop:torsion-curv-type-freestyling-reduction} and its structure equations \eqref{eq:freestyling-streqs}, it follows that there is only one fundamental invariant which is the structure function $A_1.$ Moreover, using our convention \ref{sec:conventions} for coframe differentiation, one obtains  differential relations
  \begin{equation}
    \label{eq:W13-A1}
    W_3=-\tfrac{1}{6}A_{1;\underline{2221}},\quad W_1=\tfrac{1}{6}A_{1;\underline{1112}}.
  \end{equation}
  Thus,  $A_1=0$ implies $W_1=W_3=0,$ which combined with $W_0=W_4=0$ from \eqref{eq:W04-zero-freestyling}, gives (2).

  To show $(2)\to (1)$ we note that the conditions $W_1=W_3=0$ are invariantly defined on $\cG_{red}.$ Assuming  $W_1=W_3=0$, it is a matter of tedious but straightforward calculations to show that the identities arising from $\exd^2=0$ and their resulting differential consequences  (in fact three iterations are needed) imply $A_1=0.$  

  Showing $(1)\leftrightarrow (3)$ is done using elementary computations and the characterization of chains in  \cite[Theorem 1.1]{KM-chains}. Starting from (1), one knows from (2) that  the curvature has two distinct real roots of multiplicity 2 and   can check that all necessary and sufficient conditions in \cite[Theorem 1.1]{KM-chains} are satisfied, e.g. the 2-form $\rho$ in Remark \ref{rmk:dancing-Not-Variational} on $\cG^{(3)}\to N$ is closed. Furthermore, by \cite{CZ-CR,KM-chains}, chains of a para-CR 3-manifold are torsion-free if and only if the para-CR structure is flat.  Conversely, starting from (3), by Example \ref{exa:corr-syst-odes-flat-paraCR} and Proposition \ref{prop:lewy-curves-chains}, the point equivalence class of pair of ODEs for chains and Lewy curves in flat para-CR 3-manifolds coincide, which, as mentioned before, define torsion-free path geometries.
\end{proof}

\begin{theorem}\label{cor:3d-path-geometries-dancing-chains}
A path geometry defined by Lewy curves of a 3-dimensional para-CR structure  arises as chains of some (para-)CR 3-manifold if and only if  it is flat, in which case its Lewy curves and chains coincide. 
\end{theorem}
\begin{proof}
This follows from Proposition \ref{prop:torsionfree-freestyling} and the fact that the curvature of 3-dimensional path geometry of chains of a (para-)CR structure always has two distinct real roots of multiplicity two (see \cite[Theorem 1.1]{KM-chains}). Since the roots are real, Lewy curves can only coincide with chains of a flat para-CR 3-manifold. 
\end{proof}

\subsection{Path geometry of Lewy curves: non-zero torsion case}\label{sec:non-trivial-torsion}
Having treated the case of torsion-free path geometries, we proceed to the case of non-zero torsion. Note that in the statements in the section, the reality conditions on the roots of torsion and curvature of 3-dimensional path geometries only hold in the real setting.
\begin{proposition}\label{thm3}
A 3-dimensional path geometry with non-zero torsion satisfies conditions (a)-(c) in Theorem \ref{thm2} if and only if 
  \begin{enumerate}
  \item The torsion $\bT$  has two distinct real roots and the quartic $\bC$ has at least two distinct real roots and at most one real root of multiplicity two.
  \item  There is a reduction $\iota\colon\cG_{red}\to\cG$ characterized  as in Proposition \ref{prop:torsion-curv-type-freestyling-reduction} such that the Pfaffian systems $\iota^*\cI_1,\iota^*\cI_2,\iota^*\cI_3,\iota^*\cI_4$ are integrable where
    \begin{equation}
      \label{eq:I_1234-freestyling}
      \cI_1=\{\alpha^1,\beta^1 \},\quad \cI_2=\{\alpha^2,\beta^2\},\quad \cI_3=\{\alpha^1,\alpha^0 \},\quad \cI_4=\{\alpha^2,\alpha^0\}.
    \end{equation}
  \end{enumerate}
Defining $\tilde M_2,\tilde M_1,M_1,M_2$ to be the local leaf spaces of $\cI_1,\cI_2,\cI_3,\cI_4,$ respectively, the corresponding double fibrations
  \begin{equation}
    \label{eq:N-respective-fibrations}
    M_1\leftarrow N\rightarrow M_2,\qquad M_1 \leftarrow  N\rightarrow \tilde M_2,\qquad \tilde M_1\leftarrow   N\rightarrow M_2,
  \end{equation}
define  para-CR structures  $(N,\scD_1,\scD_2),$ $(N,\tilde\scD_1,\scD_2)$ and $(N,\scD_1,\tilde \scD_2)$ as in Proposition \ref{prop:freestyle}. Moreover, the fundamental invariants of the path geometry  $(N,\scD_1,\scD_2)$ depend on the 6th jet of $A_1$ and the fundamental invariants of the other two path geometries depend on the 5th jet of $A_1.$   
\end{proposition}

\begin{proof}
The proof of the first part of the proposition directly follows from  Theorem \ref{thm2} and Propositions \ref{prop:torsion-curv-type-freestyling-reduction} and \ref{prop:torsionfree-freestyling}. As a result, the surfaces $\tilde M_2,\tilde M_1,M_1$ and $M_2,$ corresponding to para-CR structures  in Proposition \ref{prop:freestyle}, are the leaf spaces of $\cI_1,\cI_2,\cI_3$ and $\cI_4$, respectively. The last part of the proposition is shown by direct computation using  structure equations \eqref{eq:freestyling-streqs} as outlined below.

Let us define $N$ to be the leaf space of the Pfaffian system  $\{\alpha^0,\alpha^1,\beta^1\}$ with double fibration $M_1\leftarrow N\rightarrow \tilde M_2.$ By structure equations \eqref{eq:freestyling-streqs} one has
  \[
    \begin{aligned}
      \exd\alpha^1&\equiv -(\phi^1_1+2\phi^2_2)\w\alpha^1+\alpha^0\w\beta^1\\
      \exd\alpha^0&\equiv -(2\phi^2_2+\phi^1_1)\w\alpha^0\mod\{\alpha^1\}\\
      \exd\beta^1&\equiv -(\phi^2_2-\phi^1_1)\w\beta^1\mod \{\alpha^1\}.
    \end{aligned}
  \]
  where
  \begin{equation}
    \label{eq:psi1-2-UHF-1}
    \phi^1_1=\tfrac 43(x_{11}\alpha^1-x_{12}\alpha^2+x_{15}\alpha^0)+\psi^0_0-\tfrac 13\psi^1_1\qquad \phi^2_2=\tfrac 43(x_{11}\alpha^1-x_{12}\alpha^2+x_{15}\alpha^0)+\tfrac 23\psi^1_1.
      \end{equation}
Using the structure equations \eqref{eq:freestyling-streqs}, the Cartan connection \eqref{eq:2D-path-geom-cartan-conn} for the induced 2-dimensional path geometry on the 2-dimensional leaf space $M_1$ of the Pfaffian system $\{\alpha^1,\alpha^0\}$ is given below wherein  $\phi^1_1$ and $\phi^2_2$ are given as in \eqref{eq:psi1-2-UHF-1} and
  \begin{equation}
    \label{eq:UHF-1-2Dpath-Cartan-Conn}
      \def\arraystretch{1.3}
    \phi=
    \begin{pmatrix}
      -\phi^1_1-\phi^2_2 & \theta_2 &\theta_0 \\
      \alpha^0 & \phi^1_1 & \theta_1\\
      \alpha^1 & \beta^1 & \phi^2_2
    \end{pmatrix}\qquad 
    \left\{
    \begin{array}{ll}
          \theta_2=-x_{18}\alpha^2-\tfrac 53x_4\alpha^1-A_1\alpha^0+\mu_0 ,\\
          \theta_1=\tfrac 43 W_1\alpha^1+\alpha^2,\\
          \theta_0=8x_{15}\alpha^2-\tfrac 83 x_{14}\alpha^1-\tfrac{16}{3}x_4\alpha^0-\tfrac 43W_1\beta^1-\beta^2,
        \end{array}\right.
      \end{equation}
      and the fundamental invariants  are
      \begin{equation}
        \label{eq:T1-C1-tM-T}
        P_1=\tfrac 43 A_1A_{1;\underline{121}}-\tfrac 49A_{1;\underline 1}A_{1;\underline{12}}+2A_{1;1}-\tfrac 23A_{1;\underline 1 0},\qquad Q_1=\tfrac 29A_{1;\underline{11121}}
      \end{equation}
 which implies $P_1$ and $Q_1$ depend on the 3rd and 5th jets of $A_1,$ respectively.
      
Similarly, the leaf space of $\{\alpha^0,\alpha^2,\beta^2\}$ corresponds to the  2D path geometry on the 2-dimensional leaf space $M_2$ of the Pfaffian system $\{\alpha^0,\alpha^2\}.$ The Cartan connection \eqref{eq:2D-path-geom-cartan-conn} for this 2-dimensional path geometry can be obtained analogously wherein $\omega^0=\alpha^2,$$\omega^1=\alpha^0,$ $\omega^2=\beta^2$ and its fundamental invariants are given as
      \[P_1=-\tfrac 43A_1A_{1;\underline{212}}+\tfrac 23A_{1;\underline{2}}A_{1;\underline{21}} -2A_{1;2}+\tfrac 23 A_{2;\underline 2 0},\qquad Q_1=\tfrac 29 A_{1;\underline{22212}}.\]
      Lastly, as was shown previously, the leaf space of the Pfaffian system $\{\alpha^0,\alpha^1,\alpha^2\}$ corresponds to a 2-dimensional path geometry on the 2-dimensional leaf space of $\{\alpha^0,\alpha^1\}$ for which one can find the Cartan connection \eqref{eq:2D-path-geom-cartan-conn} wherein $\omega^0=\alpha^0,$  $\omega^1=\alpha^1$ and $\omega^2=\alpha^2.$ The expression for the fundamental invariants of this path geometry are much longer than the other two cases and will not be provided here. However, it is can be checked that the fundamental invariants $P_1$ and $Q_1$ in this case depend on the 6th jet of $A_1.$
\end{proof}

As discussed in Proposition \ref{prop:torsionfree-freestyling}, if $\bC$ has two distinct real roots of multiplicity two then the path geometry arises as chains of the flat para-CR geometry. In order to rule out the possibility of repeated real roots when torsion is nonzero, we first show the following wherein para-CR structures are viewed as 2-dimensional path geometries, as explained in Remark ~\ref{rmk:para-cr-structures-2D-path}.  
\begin{proposition}\label{prop:3d-path-geometries-dancing-projective}
Given a path geometry satisfying conditions (a)-(c) in Theorem \ref{thm2}, if the curvature has a repeated root of multiplicity 2 then the 2D path geometry induced on $M_1$ is either a projective or co-projective structure, as defined in Definition \ref{def:co-proj-surf}.  
\end{proposition}
\begin{proof}
Similarly to the proof  $(2)\to (1)$ in Proposition \ref{prop:torsionfree-freestyling}, the proof is  done by direct computation by setting either $W_1=0$ or $W_3=0$ and checking $\exd^2=0$ and its differential consequences.
\end{proof}

Now we can state a Cartan geometric version of Theorem \ref{thm2}.
\begin{corollary}\label{cor:3d-path-geometries-dancing}
A 3D path geometry with non-zero torsion is defined by Lewy curves of a para-CR structure if and only if 
  \begin{enumerate}
  \item The quadric $\bT$  has two distinct real roots and the quartic $\bC$ has  at least two distinct real roots and no repeated roots.
  \item  After the natural reduction $\iota\colon\cG_{red}\to\cG$ described in   Proposition \ref{prop:torsion-curv-type-freestyling-reduction}, the Pfaffian systems $\iota^*\cI_1,\iota^*\cI_2,\iota^*\cI_3,\iota^*\cI_4$ given in \eqref{eq:I_1234-freestyling} are integrable 
  \item   The triple of para-CR geometries $(N,\scD_1,\scD_2),$ $(N,\tilde\scD_1,\scD_2)$ and $(N,\scD_1,\tilde \scD_2)$ in Proposition \ref{thm3} with  double fibrations  \eqref{eq:N-respective-fibrations} are equivalent. 
\end{enumerate}
\end{corollary}
\begin{proof}
Using  Proposition \ref{thm3}, it is clear that condition (3) is equivalent to   condition (d) in Theorem \ref{thm2} and  conditions (a)-(c) in Theorem \ref{thm2} are equivalent to conditions (1) and (2) in Proposition \ref{thm3}. The only difference is between condition (1) in Corollary \ref{cor:3d-path-geometries-dancing} and condition (1) in Proposition \ref{thm3}. In other words, we only need to show that  the binary quartic $\bC$ cannot have a  repeated root. Using the reduction in Proposition \ref{prop:torsion-curv-type-freestyling-reduction}, a repeated root of the quartic $\bC$ has at most multiplicity two and is either at zero or at infinity.  Assume $W_0=W_1=W_4=0,$ i.e. $\bC$ has a repeated root at zero. Using  relation \eqref{eq:W13-A1}, one obtains that $Q_1$ in \eqref{eq:T1-C1-tM-T} vanishes. Thus, by  Theorem \ref{thm:2D-path-geome} and our discussion after Proposition \ref{prop:3D-path-geom} about vanishing of $Q_1$, the double fibration $M_1\leftarrow N\rightarrow \tilde M_2,$ arising from the  2D path geometry $(N,\tilde\scD_1,\scD_2)$ in Proposition \ref{thm3}, induces a projective structure on $M_1$ (see Definition \ref{def:co-proj-surf}.)  Moreover, as in the proof of  Proposition \ref{prop:3d-path-geometries-dancing-projective}, direct computation shows if $W_0=W_1=W_4=0,$ then the double fibration $M_1\leftarrow N\to M_2$ arising from the path geometry $(N,\scD_1,\scD_2)$ induces a  co-projective structure on $M_1$.  However, in path geometry  of Lewy curves the triple of path geometries defined by $N$ are always equivalent.   This implies that the induced path geometry on $M_1$ is both projective and co-projective and, therefore, has to be locally flat. As a result, by Proposition \ref{prop:torsionfree-freestyling} the  path geometry of Lewy curves is torsion-free. Hence, if a  path geometry defined by Lewy curves has non-zero torsion, then the quartic $\bC$  cannot have a repeated root at zero. Similarly, the case  $W_0=W_3=W_4=0$ is treated which finished the proof.  
\end{proof}
As mentioned in the proof of Corollary \ref{cor:3d-path-geometries-dancing},  the last part of  condition (1), which states that the quartic cannot have any repeated root, follows from condition (3). Nevertheless, we have included it in condition (1) since as a necessary condition it can be checked easily. Also, the condition that the quadric $\bT$ has two distinct real roots follows from condition (2), but we have included it as a necessary condition for the same reason. Note that unlike conditions (1) and (2), in general it is not easy to check condition (3) in Corollary \ref{cor:3d-path-geometries-dancing} for a pair of path geometries

\begin{remark}
Let us point out that applying Cartan-K\"ahler machinery to the structure equations \eqref{eq:freestyling-streqs}, one obtains that the local generality of such path geometries depends on 3 functions of 3 variables. This agrees with a similar computation for chains of para-CR structures (see \cite[Remark 3.12]{KM-chains}). 
\end{remark}

\begin{remark}
The leaf space of the Pfaffian systems $\cI_1=\{\alpha^0,\alpha^1,\beta^1\}$ and $\cI_2=\{\alpha^0,\alpha^2,\beta^2\}$ in the 5-manifold $Q$ can be identified with any of the  leaves of the Pfaffian systems $\cJ_1=\{\alpha^2,\beta^2\}$ and $\cJ_2=\{\alpha^1,\beta^1\},$ respectively,  and the induced path geometries are equivalent. This is clear from the obtained Cartan connection in the proof of Corollary  \ref{cor:3d-path-geometries-dancing}. For instance, setting $\alpha^2=0$ and $\beta^2=0$ amounts to restricting to a leaf of $\cJ_1$ and thus the  Cartan connection of the induced 2D path geometry for this leaf is given as \eqref{eq:UHF-1-2Dpath-Cartan-Conn} by setting $\alpha^2=0$ and $\beta^2=0.$ 
\end{remark}

Now we can show that the path geometry Lewy curves in dimension three allows one to recover the para-CR structure up to a contact automorphism or  anti-automorphism. An anti-automorphism of a para-CR structure $(N,\scD_1,\scD_2)$ is a map $f\colon N\to N,$ such that $f_*(\scD_1)=\scD_2$ and $f_*(\scD_2)=\scD_1,$ i.e. it changes the sign of the almost para-complex structure on the contact distribution. The same statement has been shown for chains of CR and para-CR structures \cite{cheng,CZ-CR}. The difference between chains and Lewy curves is the fact that  the (para-)complex structure is recovered via the curvature of the path geometry of (para)-CR chains, unlike the path geometry of Lewy curves where the para-complex structure is determined via its torsion, as explained  below.  
\begin{corollary}\label{cor:determination}
The path geometry of para-CR Lewy curves in dimension three determine the underlying para-CR structure up to the sign of the almost para-complex structure.
\end{corollary}
\begin{proof}
  By Theorem \ref{thm2}  Lewy curves determine the splitting of the contact distribution  $\scC=\scD_1\oplus\scD_2$, up to the order of $\scD_1$ and $\scD_2$. Indeed, if a diffeomorphism $\phi\colon N\to  N$ preserves  Lewy curves then its lift $j^1\phi\colon\PP TN\to \PP T N$ establishes equivalence of the corresponding pair $(\scX,\scV)$  on $Q$. In particular, the tangent mapping $(j^1\phi)_*$  preserves the torsion $\bT$ up to a scale, as it is a relative invariant. If the torsion is zero, then by Proposition \ref{prop:torsionfree-freestyling}, the underlying para-CR structure is flat. Otherwise, by Remark \ref{remark:torsion}, $(j^1\phi)_*$ preserves eigenspaces of $\bT$ on $Q$ up to the sign of the eigenvalues. By part (a) of Proposition \ref{thm3}, the eigenspaces are one dimensional and consequently, the splitting $\scV=\scV_1\oplus\scV_2$ is unique, up to the order of eigenspaces which recovers the splitting $\scC=\scD_1\oplus\scD_2$ up to the order of $\scD_1$ and $\scD_2.$
\end{proof}

\appendix
\setcounter{equation}{0}
\setcounter{subsection}{0}
 \setcounter{theorem}{0} 
\section*{Appendix}
\renewcommand{\theequation}{A.\arabic{equation}}
\renewcommand{\thesection}{A}
The structure equations for path geometries after carrying out reductions on the 8-dimensional principal $B$-bundle $\cG_{red}\to Q$ in Proposition \ref{prop:torsion-curv-type-freestyling-reduction} is given by 
\begin{equation}
  \label{eq:freestyling-streqs}
  \begin{aligned}
    \exd\alpha^2&=(\psi^1_1-\psi^0_0)\w\alpha^2+\alpha^0\w\beta^2 -8x_{15}\alpha^0\w\alpha^2-8x_{11}\alpha^1\w\alpha^2\\
    \exd\alpha^1&=(\psi^1_1-\psi^0_0)\w\alpha^1+\alpha^0\w\beta^1 -4x_{15}\alpha^0\w\alpha^1- 4x_{12}\alpha^1\w\alpha^2 \\
    \exd\alpha^0& = -2\psi^0_0\w\alpha^0+\alpha^1\w\alpha^2+4x_{11}\alpha^0\w\alpha^1-4x_{12}\alpha^0\w\alpha^2 \\
    \exd\beta^2&=(\psi^0_0+\psi^1_1)\w\beta^2-\alpha^2\w\mu_0+ A_1\alpha^0\w\alpha^2+x_4\alpha^1\w\alpha^2\\
    &\ \  -4x_{15}\alpha^0\w\beta^2 -4x_{11}\alpha^1\w\beta^2+4x_{12}\alpha^2\w\beta^2\\
    \exd\beta^1&= (\psi^0_0-\psi^1_1)\w\beta^1-\alpha^1\w\mu_0-A_1\alpha^0\w\alpha^1+x_{18}\alpha^1\w\alpha^2\\
    \exd\psi^0_0&=-\half \alpha^2\w\beta^1+\half\alpha^1\w\beta^2 -W_1\alpha^1\w\beta^1+W_3\alpha^2\w\beta^2\\
    &\ \ +x_4\alpha^0\w\alpha^1-x_{18}\alpha^0\w\alpha^2-\alpha^0\w\mu_0 +2x_{15}\alpha^1\w\alpha^2\\
    \exd\psi^1_1 &= \tfrac 32 \alpha^1\w\beta^2+\tfrac{3}{2}\alpha^2\w\beta^1 +3W_{1}\alpha^1\w\beta^1
    +W_3\alpha^2\w\beta^2\\
    &\ \ -6x_4\alpha^0\w\alpha^1-2x_{18}\alpha^0\w\alpha^2-6x_{15}\alpha^1\w\alpha^2\\   
    \exd\mu_0& =2\psi^0_0\w\mu_0-\beta^1\w\beta^2 + x_1\alpha^0\w\alpha^1+x_7\alpha^0\w\alpha^2+x_{16}\alpha^1\w\alpha^2 \\
  &\ \  +3x_4\alpha^0\w\beta^1+x_{14}\alpha^1\w\beta^1-5x_{15} \alpha^2\w\beta^1-3x_{18}\alpha^0\w\beta^2+5x_{15}\alpha^1\w\beta^2\\
    &\ \ -x_3\alpha^2\w\beta^2+4x_{15}\alpha^0\w\mu_0-4x_{12}\alpha^2\w\mu_0+4x_{11}\alpha^1\w\mu_0\\
  \end{aligned}
\end{equation}
for some functions $x_1,\cdots,x_{18},A_1,W_1,W_3$ on $\cG_{red}.$

\subsection*{Acknowledgments}
The starting point of this article was several discussions that Maciej Dunajski initiated with the authors about his work on the dancing construction. The authors would like to thank him for his encouragement and sharing his ideas.

WK was partially supported by the grant 2019/34/E/ST1/00188 from the National Science Centre, Poland.  OM received partial funding from the Norwegian Financial Mechanism 2014-2021 with project registration number 2019/34/H/ST1/00636. OM gratefully acknowledges partial support by the grant  PID2020-116126GB-I00 provided via the Spanish Ministerio de Ciencia e Innovaci\'on MCIN/ AEI /10.13039/50110001103 as well as partial funding from the Norwegian Financial Mechanism 2014-2021 (project registration number 2019/34/H/ST1/00636), the Tromsø Research Foundation (project “Pure Mathematics in Norway”), and the UiT Aurora project MASCOT. The  EDS calculations  are done using Jeanne Clelland's \texttt{Cartan} package in Maple.

\bibliographystyle{alpha}      
\bibliography{lewycurves}  

\begin{thebibliography}{{Ch{e}}88}

\bibitem[BDSL24]{Florian2}
F.~Bertrand, G.~Della~Sala, and B.~Lamel.
\newblock Extremal discs and {S}egre varieties for real-analytic hypersurfaces
  in $\mathbb{C}^2$.
\newblock {\em To appear in Proc. AMS}, 2024.

\bibitem[BER99]{BookSegre}
M.~S. Baouendi, P.~Ebenfelt, and L.~P. Rothschild.
\newblock {\em Real submanifolds in complex space and their mappings},
  volume~47 of {\em Princeton Mathematical Series}.
\newblock Princeton University Press, Princeton, NJ, 1999.

\bibitem[BHLN18]{Bor}
G.~Bor, L.~Hern{\'a}ndez~Lamoneda, and P.~Nurowski.
\newblock The dancing metric, ${G}_2$-symmetry and projective rolling.
\newblock {\em Trans. Amer. Math. Soc.}, 370(6):4433--4481, 2018.

\bibitem[Bry97]{Bryant-ProjFlat}
R.~L. Bryant.
\newblock Projectively flat {F}insler {$2$}-spheres of constant curvature.
\newblock {\em Selecta Math. (N.S.)}, 3(2):161--203, 1997.

\bibitem[{\v{C}}ap05]{CapCrelle}
A.~{\v{C}}ap.
\newblock Correspondence spaces and twistor spaces for parabolic geometries.
\newblock {\em J. Reine Angew. Math.}, 582:143--172, 2005.

\bibitem[Car24]{Cartan-Proj}
E.~Cartan.
\newblock Sur les vari\'{e}t\'{e}s \`a connexion projective.
\newblock {\em Bull. Soc. Math. France}, 52:205--241, 1924.

\bibitem[Car33]{CartanSegre}
E.~Cartan.
\newblock Sur la g\'{e}om\'{e}trie pseudo-conforme des hypersurfaces de
  l'espace de deux variables complexes.
\newblock {\em Ann. Mat. Pura Appl.}, 11(1):17--90, 1933.

\bibitem[Che75]{Chern}
S.-S. Chern.
\newblock On the projective structure of a real hypersurface in {$C\sb{n+1}$}.
\newblock {\em Math. Scand.}, 36:74--82, 1975.

\bibitem[{Ch{e}}88]{cheng}
J.~H. {Ch{e}ng}.
\newblock Chain-preserving diffeomorphisms and {CR} equivalence.
\newblock {\em Proc. Amer. Math. Soc.}, 103(1):75--80, 1988.

\bibitem[CM74]{CM-CR}
S.-S. Chern and J.~K. Moser.
\newblock Real hypersurfaces in complex manifolds.
\newblock {\em Acta Math.}, 133:219--271, 1974.

\bibitem[{\v{C}}S09]{CS-Parabolic}
A.~{\v{C}}ap and J.~Slov{\'a}k.
\newblock {\em Parabolic geometries. {I}}, volume 154 of {\em Mathematical
  Surveys and Monographs}.
\newblock American Mathematical Society, Providence, RI, 2009.
\newblock Background and general theory.

\bibitem[{\v{C}}{\v{Z}}09]{CZ-CR}
A.~{\v{C}}ap and V.~{\v{Z}}{\'a}dn{i}k.
\newblock On the geometry of chains.
\newblock {\em Journal of Differential Geometry}, 82(1):1--33, 2009.

\bibitem[DMT20]{DoubrovThe1}
B.~Doubrov, A.~Medvedev, and D.~The.
\newblock Homogeneous integrable {L}egendrian contact structures in dimension
  five.
\newblock {\em J. Geom. Anal.}, 30(4):3806--3858, 2020.

\bibitem[Dun22]{D}
M.~Dunajski.
\newblock Twistor theory of dancing paths.
\newblock {\em SIGMA. Symmetry, Integrability and Geometry: Methods and
  Applications}, 18:027, 2022.

\bibitem[Far80]{Faran-Segre}
J.~J. Faran.
\newblock Segre families and real hypersurfaces.
\newblock {\em Invent. Math.}, 60(2):135--172, 1980.

\bibitem[Far81]{Faran-Lewy}
J.~J. Faran.
\newblock Lewy's curves and chains on real hypersurfaces.
\newblock {\em Trans. Amer. Math. Soc.}, 265(1):97--109, 1981.

\bibitem[Fel95]{Fels}
M.~E. Fels.
\newblock The equivalence problem for systems of second-order ordinary
  differential equations.
\newblock {\em Proc. London Math. Soc. (3)}, 71(1):221--240, 1995.

\bibitem[Gro00]{Grossman-Thesis}
D.~A. Grossman.
\newblock Torsion-free path geometries and integrable second order {ODE}
  systems.
\newblock {\em Selecta Math. (N.S.)}, 6(4):399--442, 2000.

\bibitem[KM21]{KM-Cayley}
W.~Kry{\'n}ski and O.~Makhmali.
\newblock The {C}ayley cubic and differential equations.
\newblock {\em The Journal of Geometric Analysis}, 31:6219--6273, 2021.

\bibitem[KM24]{KM-chains}
W.~Kry\'nski and O.~Makhmali.
\newblock A characterization of chains in dimension three.
\newblock {\em To appear in Ann. Sc. Norm. Super. Pisa Cl. Sci.}, 2024.

\bibitem[Mak16]{Omid-Thesis}
O.~Makhmali.
\newblock {\em Differential Geometric Aspects of Causal Structures}.
\newblock PhD thesis, McGill University, 2016.

\bibitem[MS23]{MS-cone}
O.~Makhmali and K.~Sagerschnig.
\newblock Parabolic quasi-contact cone structures with an infinitesimal
  symmetry.
\newblock {\em arXiv preprint arXiv:2302.02862}, 2023.

\bibitem[Seg31a]{Segre2}
B.~Segre.
\newblock {\em Intorno al problema di Poincar{\'e} della rappresentazione
  pseudoconforme}.
\newblock Mem. della R. Acc. d'It., 1931.

\bibitem[Seg31b]{Segre1}
B.~Segre.
\newblock Questioni geometriche legate colla teoria delle funzioni di due
  variabili complesse.
\newblock {\em Rendic. del Seminario Mat. della R. Università di Roma},
  7:59--107, 1931.

\bibitem[Tak94]{Takeuchi}
M.~Takeuchi.
\newblock Lagrangean contact structures on projective cotangent bundles.
\newblock {\em Osaka J. Math.}, 31(4):837--860, 1994.

\bibitem[Web77]{Webster}
S.~M. Webster.
\newblock On the mapping problem for algebraic real hypersurfaces.
\newblock {\em Invent. Math.}, 43(1):53--68, 1977.

\end{thebibliography}
\end{document}